\documentclass[reqno,oneside,12pt]{amsart}

\usepackage{amsmath, amsfonts, amssymb,amsbsy,eucal,mathrsfs, mathtools}

\usepackage[normalem]{ulem}

\usepackage{textpos}

\usepackage{amsthm}
\usepackage{dynkin-diagrams}

\usepackage[all]{xy}

\usepackage{chngcntr}

\usepackage{todonotes}

\usepackage{fullpage}

\usepackage[utf8]{inputenc}

\numberwithin{equation}{section}
\theoremstyle{plain}

\newtheorem{theorem}{Theorem}[section]
\newtheorem{lemma}[theorem]{Lemma}
\newtheorem{proposition}[theorem]{Proposition}
\newtheorem{corollary}[theorem]{Corollary}
\newtheorem{conjecture}[theorem]{Conjecture}

\theoremstyle{definition}

\newtheorem{definition}[theorem]{Definition}

\newtheorem{assumption}[theorem]{Assumtion}

\theoremstyle{remark}

\newtheorem{remark}[theorem]{Remark}

\DeclareMathOperator{\aut}{Aut}
\DeclareMathOperator{\G}{G}
\DeclareMathOperator{\Q}{Q}
\DeclareMathOperator{\IG}{IG}
\DeclareMathOperator{\OG}{OG}
\DeclareMathOperator{\Fl}{Fl}

\DeclareMathOperator{\QH}{QH}
\DeclareMathOperator{\BQH}{BQH}

\DeclareMathOperator{\Spec}{Spec}

\DeclareMathOperator{\codim}{{codim}}
\DeclareMathOperator{\coeff}{{coeff}}

\DeclareMathOperator{\roots}{\Phi}
\DeclareMathOperator{\shortroots}{\Phi_s}
\DeclareMathOperator{\simpleroots}{\Delta}

\DeclareMathOperator{\F}{F}
\DeclareMathOperator{\Fpt}{F_{pt}}
\DeclareMathOperator{\pointclass}{[pt]}

\DeclareMathOperator{\Z}{Z}

\renewcommand{\and}{\quad \text{and} \quad}

\newcommand{\D}{{\Delta}}

\newcommand{\QQ}{{\mathbb Q}}
\newcommand{\bQ}{{\mathbb Q}}
\newcommand{\bP}{{\mathbb P}}
\newcommand{\starz}{{\!\ \star_{0}\!\ }}
\newcommand{\pt}{{{\rm pt}}}

\newcommand{\ptclass}{[{\rm pt}]}
\newcommand{\C}{{\mathbb C}}
\newcommand{\p}{\mathbb{P}}

\newcommand{\scal}[1]{\langle #1 \rangle}

\newcommand{\gs}{{\mathfrak{s}}}
\newcommand{\gm}{{\mathfrak{m}}}
\newcommand{\gt}{{\mathfrak{t}}}
\newcommand{\ha}{{\hat{\alpha}}}

\def \Yo {{\mathring{Y}}}

\def \Xo {{\mathring{X}}}
\def \Yo {{\mathring{Y}}}
\def \Zo {{\mathring{Z}}}
\def \spec {{{\rm Spec}}}

\newcommand{\cE}{\mathcal{E}}

\newcommand{\cO}{\mathcal{O}}
\newcommand{\cU}{\mathcal{U}}

\newcommand{\cQ}{\mathcal{Q}}

\newcommand{\cV}{\mathcal{V}}

\newcommand{\cR}{\mathcal{R}}

\newcommand{\rA}{\mathrm{A}}
\newcommand{\rB}{\mathrm{B}}
\newcommand{\rC}{\mathrm{C}}
\newcommand{\rD}{\mathrm{D}}
\newcommand{\rE}{\mathrm{E}}
\newcommand{\rF}{\mathrm{F}}
\newcommand{\rG}{\mathrm{G}}

\newcommand{\rL}{\mathrm{L}}
\newcommand{\rP}{\mathrm{P}}
\newcommand{\rQ}{\mathrm{Q}}
\newcommand{\rR}{\mathrm{R}}
\newcommand{\rT}{\mathrm{T}}

\newcommand{\rTs}{\mathrm{T}_{\mathrm{short}}}

\newcommand{\rW}{\mathrm{W}}
\newcommand{\rX}{\mathrm{X}}

\newcommand{\hrG}{\hat{\rG}}
\newcommand{\hrP}{\hat{\rP}}
\newcommand{\hrR}{\hat{\rR}}
\newcommand{\hrT}{\hat{\rT}}
\newcommand{\hrB}{\hat{\rB}}
\newcommand{\hrQ}{\hat{\rQ}}

\newcommand{\homega}{\hat{\omega}}
\newcommand{\halpha}{\hat{\alpha}}
\newcommand{\hbeta}{\hat{\beta}}
\newcommand{\hgamma}{\hat{\gamma}}

\usepackage{hyperref}
\hypersetup{
colorlinks = true,
urlcolor = blue,
citecolor = .
}

\newcommand{\bA}{{\mathbb A}}
\newcommand{\bC}{{\mathbb C}}

\newcommand{\bZ}{{\mathbb Z}}

\newcommand{\adj}{\mathrm{ad}}
\newcommand{\coadj}{\mathrm{coad}}

\DeclareMathOperator{\QS}{QS}
\newcommand{\QSo}{\QS^\circ}
\newcommand{\QSx}{\QS^\times}
\newcommand{\QHcan}{\QH_{\mathrm{can}}}

\newcommand{\Db}{{\mathbf D^{\mathrm{b}}}}
\DeclareMathOperator{\Aut}{Aut}
\DeclareMathOperator{\Pic}{Pic}

\begin{document}


\title{On the big quantum cohomology of coadjoint varieties}

\author{Nicolas Perrin}
\address{
\parbox{0.95\textwidth}{
Laboratoire de Math\'ematiques de Versailles,
UVSQ,
CNRS,
Universit\'e Paris--Saclay,
78035 Versailles,
France
\smallskip
}}
\email{nicolas.perrin@uvsq.fr}

\author{Maxim N. Smirnov}
\address{
\parbox{0.95\textwidth}{
Universit\"at Augsburg,
Institut f\"ur Mathematik,
Universit\"atsstr.~14,
86159 Augsburg,
Germany
\smallskip
}}
\email{maxim.smirnov@math.uni-augsburg.de}
\email{maxim.n.smirnov@gmail.com}

\begin{abstract}
  This paper is devoted to the study of the quantum cohomology of coadjoint varieties
  of simple algebraic groups across all Dynkin types. We determine the non-semisimple
  factors of the small quantum cohomology ring and relate them to ADE-singularities.
  Moreover, we show that the big quantum cohomology of a coadjoint variety is always
  generically semisimple even though in most cases the small quantum cohomology is not.
\end{abstract}

\maketitle

\tableofcontents

\section{Introduction}
\label{section:introduction}

Quantum cohomology of rational homogeneous spaces $\rG/\rP$ is a rich subject
that received considerable attention in the last decades. A particular aspect
that has been studied is the generic semisimplicity of the quantum cohomology
and its connections to the bounded derived category $\Db(\rG/\rP)$ of coherent
sheaves via Dubrovin's conjecture. Recall that Dubrovin's conjecture predicts
that for a smooth Fano variety $X$ the existence of a full exceptional
collection in $\Db(X)$ is equivalent to the generic semisimplicity of
the \textsf{big quantum cohomology} ring $\BQH(X)$.

A folklore conjecture predicts that the big quantum cohomology of a rational
homogenous space is always generically semisimple.

\begin{conjecture}
  \label{conjecture:introduction-semisimplicity-of-BQH}
  Let $\rG$ be a semisimple algebraic group and $\rP \subset \rG$ a parabolic subgroup.
  Then the big quantum cohomology $\BQH(\rG/\rP)$ is generically semisimple.
\end{conjecture}

We discuss a companion conjecture on the derived category $\Db(\rG/\rP)$ later in
Section~\ref{subsection:derived-category-of-coherent-sheaves-and-results-of-KuSm21}.

\smallskip

Let us list homogeneous varieties of simple
algebraic groups $\rG$ corresponding to maximal parabolic subgroups $\rP \subset \rG$,
where Conjecture \ref{conjecture:introduction-semisimplicity-of-BQH} is known to hold.
Such a group $\rG$ is determined by its Dynkin diagram that falls into types
$\rA, \rB, \rC, \rD, \rE, \rF, \rG$, and its maximal parabolic subgroups correspond
to vertices of the Dynkin diagram, for which we use the standard labelling  \cite{Bo}.
We denote by $\rP_k$ the maximal parabolic subgroup of $\rG$ corresponding
to the $k$-th vertex of the Dynkin diagram of $\rG$.
\begin{description}
  \item[Type $\rA_n$] \hspace{5pt} $\rA_n/\rP_k = \G(k,n+1)$ for $k \in [1,n]$.
  \item[Type $\rB_n$] \hspace{5pt} $\rB_n/\rP_1 = Q_{2n-1}$, $\rB_n/\rP_2 = \OG(2,2n+1)$, and $\rB_n/\rP_n = \OG(n,2n+1)$.
  \item[Type $\rC_n$] \hspace{5pt} $\rC_n/\rP_1 = \bP^{2n-1}$, $\rC_n/\rP_2 = \IG(2,2n)$ and $\rC_n/\rP_n = \IG(n,2n)$.
  \item[Type $\rD_n$] \hspace{5pt} $\rD_n/\rP_1 = Q_{2n-2}$ and $\rD_n/\rP_{n-1} = \rD_n/\rP_n = \OG_{+}(n,2n)$.

  \item[Type $\rE_n$] \hspace{5pt} $\rE_6/\rP_1 \simeq \rE_6/\rP_6$ and $\rE_7/\rP_7$.

  \item[Type $\rF_4$] \hspace{5pt} $\rF_4/\rP_1$ and $\rF_4/\rP_4$.

  \item[Type $\rG_2$] \hspace{5pt} $\rG_2/\rP_1$ and $\rG_2/\rP_2$.
\end{description}

In all of the above cases except for $\IG(2,2n)$ and $\rF_4/\rP_4$,
the generic semisimplicity of the big quantum
cohomology $\BQH(\rG/\rP)$ follows from the semisimplicity of the \textsf{small quantum cohomology $\QH(\rG/\rP)$}.
The small quantum cohomology of a Fano variety is a much simpler gadget than its
big quantum cohomology, as the former involves only finitely many Gromov--Witten
invariants, whereas the latter involves infinitely many of them.

In type $\rA_n$ the semisimplicity of the small quantum cohomology $\QH(\rA_n/\rP_k)$
is known for any maximal parabolic $\rP_k$.
For types $\rB_n, \rC_n, \rD_n$ it is known (see \cite[Table~on~p.~326]{ChPe})
that very often, roughly speaking in at least half of the cases, the small quantum cohomology
$\QH(\rG/\rP)$ is not semisimple. In the exceptional types $\rE_6, \rE_7, \rE_8, \rF_4$ a similar
behaviour persists (see \cite[Table~on~p.~326]{ChPe}). Therefore, in general one must
work with the big quantum cohomology $\BQH(\rG/\rP)$ in Conjecture \ref{conjecture:introduction-semisimplicity-of-BQH}.
Up to now the only cases with non-semisimple $\QH(\rG/\rP)$, where Conjecture
\ref{conjecture:introduction-semisimplicity-of-BQH} is proved to hold are the symplectic
isotropic Grassmannians $\rC_n/\rP_2 = \IG(2,2n)$ and the $\rF_4$-Grassmannian
$\rF_4/\rP_4$ (see \cite{CMMPS,GMS,Pe,MPS}). Both cases are the so called
\textsf{coadjoint varieties} in respective Dynkin types and one of the main results of this paper
is the proof of Conjecture \ref{conjecture:introduction-semisimplicity-of-BQH} for
coadjoint varieties in all Dynkin types.

\begin{remark}
  Note that in the classical types $\rB_n, \rC_n, \rD_n$ we only listed examples
  that fit into infinite series. Since by \cite{BKT} presentations for the small
  quantum cohomology rings in types $\rB_n, \rC_n, \rD_n$ are known, it could be
  possible to do a computer check of the semisimplicity of the small quantum cohomology
  for some isolated small rank examples (e.g. $\QH(\IG(3,8))$ can easily be checked to be semisimple).
  On the contrary, in the exceptional types $\rE_6, \rE_7, \rE_8, \rF_4$, even though
  we only have finitely many varieties to consider, the problem in extending the above
  list by a single example is non-trivial, as already presentations even for the small
  quantum cohomology are known only for (co)minuscule or (co)adjoint varieties.
\end{remark}

\subsection{Statements of results}

Recall that an \textsf{adjoint} (resp. \textsf{coadjoint}) variety of a simple
algebraic group $\rG$ is the highest weight vector orbit in the projectivization of
the irreducible $\rG$-representation, whose highest weight is the highest \emph{long}
(resp. \emph{short}) root of $\rG$. Clearly, if the group $\rG$ is simply laced, then
adjoint and coadjoint varieties coincide.

In Table \ref{table:adjoint-and-coadjoint-varieties} we give an explicit list of adjoint and coadjoint varieties.
In type $\rA_n$ the parabolic $\rP_{1,n}$ is the submaximal parabolic
subgroup defined by the subset $\{1,n\}$ of the set of vertices of the Dynkin
diagram of type $\rA_n$.
Note that the Picard rank of (co)adjoint varieties is one, except for type $\rA_n$,
where it is two. In a given Dynkin type we denote the adjoint (resp.~coadjoint)
variety by $X^\adj$ (resp.~$X^\coadj$).

\begin{table}[h!]
\centering
\begin{tabular}{ccccc}
  \hline
  Type of $\rG$ & Coadjoint variety &  Adjoint variety \\
  \hline
  $\rA_n$ & $\rA_n/\rP_{1,n} = \Fl(1,n;n+1)$ & $\Fl(1,n;n+1)$ \\
  $\rB_n$ & $\rB_n/\rP_1 = \Q_{2n-1}$ & $\rB_n/\rP_2 = \OG(2,2n+1)$ \\
  $\rC_n$ & $\rC_n/\rP_2 = \IG(2,2n)$  & $\rC_{n}/\rP_1 = \p^{2n-1}$ \\
  $\rD_n$ & $\rD_n/\rP_2 = \OG(2,2n)$ & $\rD_n/\rP_2 = \OG(2,2n)$ \\
  $\rE_6$ & $\rE_6/\rP_2$ & $\rE_6/\rP_2$ \\
  $\rE_7$ & $\rE_7/\rP_1$ & $\rE_7/\rP_1$ \\
  $\rE_8$ & $\rE_8/\rP_8$ & $\rE_8/\rP_8$ \\
  $\rF_4$ & $\rF_4/\rP_4$ & $\rF_4/\rP_1$ \\
  $\rG_2$ & $\rG_2/\rP_2 = \Q_5$ & $\rG_2/\rP_1$ \\
  \hline
\end{tabular}
\caption{Adjoint and coadjoint varieties}
\label{table:adjoint-and-coadjoint-varieties}
\end{table}

In this paper we concentrate our attention on coadjoint varieties, as for non simply laced
groups $\rG$, for which there is a distinction between $X^\adj$ and $X^\coadj$, it is known by
\cite{ChPe} that already the small quantum cohomology $\QH(X^\adj)$ is semisimple.

Our first result shows that for all coadjoint varieties
the presentation of the small quantum cohomology ring has some common features.
Let us fix some notation before stating this result. Recall that for a smooth
projective Fano variety $X$ the biggest integer that divides
the class of the canonical bundle $\omega_X \in \Pic(X)$ is called \textsf{index} of $X$;
we denote it by $r$.
For a cohomology class $\gamma \in H^{2d}(X, \bQ)$ we define
$\deg(\gamma) \coloneqq d$, i.e., this is the \textsf{Chow degree} of $\gamma$.
Finally, we recall that the small quantum cohomology is defined over the field
\begin{equation*}
  K = \overline{\bQ((q))},
\end{equation*}
where $q$ is a formal variable of degree $\deg(q) = r$
(see Section \ref{subsection:conventions-and-notation-for-quantum-cohomology} for details).

\begin{theorem}
  \label{theorem:introduction-uniform-presentation-for-QH}
  Let $X^\coadj$ be a coadjoint variety not of type $\rA_n$.
  There exists a presentation
  \begin{equation}\label{eq:presentation-small-QH}
    \QH(X^\coadj) =
    \begin{cases}
      K[h]/(E + \lambda q h) \quad & \text{in types $\rB_n$ and $\rG_2$}, \\[3pt]
      K[h,\delta_1]/(E_1 , E + \lambda q h) \quad & \text{in types $\rC_n$ and $\rF_4$}, \\[3pt]
      K[h,\delta_1 , \delta_2]/(E_1 , E_2 , E + \lambda q h) \quad & \text{in types $\rD_n, \rE_6, \rE_7, \rE_8$},
    \end{cases}
  \end{equation}
  where $\lambda \in \bZ_{\neq 0}$, $h$ is the hyperplane class,
  $\delta_i$ are Schubert classes of degrees $\deg(\delta_i) = d_i$, and
  $E, E_1, E_2$ are graded homogeneous polynomials with rational coefficients
  of degrees
  \begin{equation*}
    \deg(E) = r+1 \quad \text{and} \quad \deg(E_i) = r+1-d_i,
  \end{equation*}
  where $r$ is the index of $X^\coadj$.
\end{theorem}
The values of the constants appearing in Theorem \ref{theorem:introduction-uniform-presentation-for-QH}
are collected in Table \ref{table:constants} below.
\begin{table}[h!]
\centering
\begin{tabular}{cccccc}
  \hline
  Type of $\rG$ & $X^\coadj$ &  $r$ & $d_1$ & $d_2$ \\
  \hline
  $\rB_n$ & $\rB_n/\rP_1 = \Q_{2n-1}$ & $2n-1$ & & \\
  $\rC_n$ & $\rC_n/\rP_2 = \IG(2,2n)$  & $2n-1$ & $2$ & \\
  $\rD_n$ & $\rD_n/\rP_2 = \OG(2,2n)$ & $2n-3$ & $2$ & $n-2$ \\
  $\rE_6$ & $\rE_6/\rP_2$ & $11$ & $3$ & $4$ \\
  $\rE_7$ & $\rE_7/\rP_1$ & $17$ & $4$ & $6$ \\
  $\rE_8$ & $\rE_8/\rP_8$ & $29$ & $6$ & $10$ \\
  $\rF_4$ & $\rF_4/\rP_4$ & $11$ & $4$ & \\
  $\rG_2$ & $\rG_2/\rP_2 = \Q_5$ & $5$ \\
  \hline
\end{tabular}
\caption{Constants appearing in Theorem \ref{theorem:introduction-uniform-presentation-for-QH}}
\label{table:constants}
\end{table}

Most cases of this theorem are already known. Indeed, for types $\rB_n$ and $\rG_2$
this is \cite{Beauville, ChMaPe}, for type $\rC_n$ this is \cite{BKT, CMMPS},
and for type $\rE_6$, $\rE_7$, $\rE_8$ and $\rF_4$ these are Propositions 5.4, 5.6, 5.7 and
5.3 of \cite{ChPe}. Thus, we only need to give a proof in type $\rD_n$, which is
done in Section \ref{section:type-D} (see Corollary \ref{corollary:presentation-small-QH-type-D}).

\medskip

The second main result of this paper is a uniform description of the non-reduced
factor of $\QH(X^\coadj)$. Before stating it we need to introduce some notation.
For a simple algebraic group $\rG$ we denote by $\rT(\rG)$ its Dynkin diagram
and by $\rTs(\rG)$ its subdiagram of short roots. In simply laced types we view
all roots as both short and long. For convenience of the reader we collect the
resulting Dynkin types into a table:
\begin{equation*}
\begin{array}{|c|c|c|c|c|c|c|c|}
\hline
\rT & \rA_n & \rB_n & \rC_n & \rD_n & \rE_n & \rF_4 & \rG_2 \\
\hline
\rTs & \rA_n & \rA_1 & \rA_{n-1} & \rD_n & \rE_n & \rA_2 & \rA_1 \\
\hline
\end{array}
\end{equation*}

\medskip

In view of \eqref{eq:presentation-small-QH} we define the \textsf{origin} of
$\Spec \QH(X^\coadj)$ as
\begin{equation}\label{eq:origin}
  \begin{cases}
    h = 0 \quad & \text{in types $\rB_n$ and $\rG_2$}, \\[3pt]
    h = \delta_1 = 0 \quad & \text{in types $\rC_n$ and $\rF_4$}, \\[3pt]
    h = \delta_1 = \delta_2 = 0 \quad & \text{in types $\rD_n, \rE_6, \rE_7, \rE_8$},
  \end{cases}
\end{equation}

With this notation we can formulate the following.
\begin{theorem}
  \label{theorem:introduction-fat-points-of-QH}
  Let $\rG$ be a simple algebraic group not of type $\rA_n$ and let $X^\coadj$ be
  the corresponding coadjoint variety. Then we have:
  \begin{enumerate}
    \item All points of $\Spec \QH(X^\coadj)$ different from the point \eqref{eq:origin} are reduced.

    \smallskip

    \item The localisation of $\QH(X^\coadj)$ at the point \eqref{eq:origin} is isomorphic to the
    Jacobian ring of a simple hypersurface singularity of Dynkin type $\rTs(\rG)$.
  \end{enumerate}
  In particular, since the Jacobian ring of an $\rA_1$-singularity is just the
  ground field $K$, the small quantum cohomology $\QH(X^\coadj)$ is semisimple
  if and only if the group $\rG$ is of Dynkin type $\rB_n$ or $\rG_2$.
\end{theorem}

In type $\rD_n$ this is a consequence of Proposition \ref{prop:decomp-A-B} and
Lemma \ref{lemma:description-fat-point-type-D}.
In types $\rE_6, \rE_7, \rE_8, \rF_4$ this is proved in Lemmas \ref{lemma:description-fat-point-E6},
\ref{lemma:description-fat-point-E7}, \ref{lemma:description-fat-point-E8} and
\ref{lemma:description-fat-point-F4} respectively.
In type $\rC_n$ this is \cite[Proposition 4.3]{CMMPS}.
In types $\rB_n, \rG_2$ this is known by \cite{Beauville, ChMaPe}.

\medskip

After examining the structure of the small quantum cohomology of coadjoint varieties
we are ready to proceed with our study of their big quantum cohomology.
Here we adopt the strategy of \cite{CMMPS, MPS}.
Namely, for coadjoint varieties with non-semisimple $\QH(X^\coadj)$, we determine
the first order deformation of $\QH(X^\coadj)$ inside the big quantum cohomology
$\BQH(X^\coadj)$ along the directions
\begin{equation*}
  \delta \coloneqq
  \begin{cases}
    \delta_1 \quad & \text{in types $\rC_n$ and $\rF_4$}, \\[3pt]
    \delta_1 , \delta_2 \quad & \text{in types $\rD_n, \rE_6, \rE_7, \rE_8$},
  \end{cases}
\end{equation*}
as in Theorem \ref{theorem:introduction-uniform-presentation-for-QH}.
We denote this deformation by $\BQH_{\delta}(X^\coadj)$.

Let $t_{\delta_i}$be the quantum variable associated to $\delta_i$
(see Subsection \ref{subsection:conventions-and-notation-for-quantum-cohomology})
and consider the ideals of $\BQH_{\delta}(X^\coadj)$ defined as
\begin{equation*}
  \begin{cases}
    \gt = (t_{\delta_1}) \subset \gm = (h,\delta_1,t_{\delta_1}) \quad & \text{in types $\rC_n$ and $\rF_4$}, \\[3pt]
    \gt = (t_{\delta_1},t_{\delta_2}) \subset \gm = (h,\delta_1,\delta_2,t_{\delta_1},t_{\delta_2}) \quad & \text{in types $\rD_n, \rE_6, \rE_7, \rE_8$}.
  \end{cases}
\end{equation*}

With this notation we can state our third main result.

\begin{theorem}
  \label{theorem:introduction-presentation-of-BQH}
  Let $X^\coadj$ be the coadjoint variety not of type $\rA_n$, $\rB_n$ or $\rG_2$.
  Then for the ring $\BQH_\delta(X^\coadj)$ we have a presentation of the form
  \begin{equation}\label{eq:presentation-big-QH}
    \BQH_\delta(X^\coadj) =
    \begin{cases}
      K[h,\delta_1] \Big[ \Big[ t_{\delta_1} \Big] \Big]/ \big( \widetilde{E}_1 , \widetilde{E} \big) \quad & \text{in types $\rC_n$ and $\rF_4$}, \\[15pt]
      K[h,\delta_1 , \delta_2]\Big[ \Big[ t_{\delta_1}, t_{\delta_2} \Big] \Big]/ \big( \widetilde{E}_1 , \widetilde{E}_2 , \widetilde{E} \big) \quad & \text{in types $\rD_n, \rE_6, \rE_7, \rE_8$},
    \end{cases}
  \end{equation}
  where for the relations $\widetilde{E}_i, \widetilde{E}$ we have congruences
  \begin{equation*}
    \begin{aligned}
      & \widetilde{E}_i \equiv E_i + \lambda_i q t_{\delta_i} \ ({\rm mod}\  \gt\gm), \\
      & \widetilde{E} \equiv E + \lambda qh \ ({\rm mod}\  \gt\gm),
    \end{aligned}
  \end{equation*}
  with $h, \delta_i, E_i, E, \lambda$ defined in Theorem~\ref{theorem:introduction-uniform-presentation-for-QH},
  and some constants~$\lambda_i \in \bZ_{\neq 0}$.
\end{theorem}

Here are pointers to the proofs of Theorem \ref{theorem:introduction-presentation-of-BQH}.
In Dynkin types $\rD_n, \rE_6, \rE_7, \rE_8, \rF_4$ these are Propositions
\ref{proposition:Dn-presentation-big-QH},
\ref{proposition:E6-presentation-big-QH},
\ref{proposition:E7-presentation-big-QH},
\ref{proposition:E8-presentation-big-QH},
\ref{proposition:F4-presentation-big-QH}
respectively.
In type $\rC_n$ this is \cite[Theorem 6.4]{CMMPS}.

\smallskip

As an immediate corollary of Theorem \ref{theorem:introduction-presentation-of-BQH}
we obtain the following.

\begin{corollary}\label{corollary:introduction-corollary-new}
  Let $X^\coadj$ be the coadjoint variety not of type $\rA_n$. Then we have:
  \begin{enumerate}
    \item $\BQH(X^\coadj)$ is a regular ring.

    \item $\BQH(X^\coadj)$ is generically semisimple.
  \end{enumerate}
\end{corollary}

The first claim follows easily from the presentation given in
Theorem~\ref{theorem:introduction-presentation-of-BQH}, and the second claim
follows easily from the first.
For completeness a proof is given in Section \ref{section:type-D} for type $\rD_n$
(see Corollary \ref{corollary:regularity-and-semisimplicity-of-BQH-type-D}).
The same proof works verbatim in types $\rE_6, \rE_7, \rE_8, \rF_4$.
In types $\rB_n, \rG_2$ there is nothing to prove, as here already the
small quantum cohomology is known to be semisimple by
Theorem \ref{theorem:introduction-fat-points-of-QH}
and, therefore, both the regularity
and the generic semisimplicity of the big quantum cohomology hold automatically.

In types $\rC_n$ and $\rF_4$ the generic semisimplicity
of the big quantum cohomology is already known by \cite{CMMPS, GMS, MPS, Pe}. However,
in type $\rF_4$ our approach gives a new proof of this fact. In types $\rD_n, \rE_6, \rE_7, \rE_8$
our results are new.

\medskip

In view of Corollary \ref{corollary:introduction-corollary-new}(1) we propose the following.
\begin{conjecture}
  \label{conjecture:introduction-regularity-of-BQH}
  Let $\rG$ be a semisimple algebraic group and $\rP \subset \rG$ a parabolic
  subgroup. Then the big quantum cohomology ring $\BQH(\rG/\rP)$ is regular.
\end{conjecture}

Note that Conjecture \ref{conjecture:introduction-regularity-of-BQH} implies
Conjecture \ref{conjecture:introduction-semisimplicity-of-BQH},
as in Corollary \ref{corollary:introduction-corollary-new}.

\medskip

Our proofs of Theorems \ref{theorem:introduction-uniform-presentation-for-QH}
and \ref{theorem:introduction-presentation-of-BQH} rely on the following results.
First, we extensively use combinatorial formulas for Littlewood--Richardson coefficients
on (co)minuscule and (co)adjoint varieties proved in \cite{ChPeLR,ThYo}.
In particular, these allow us to compute all the necessary Gromov--Witten invariants
(see Sections~\ref{section:geometry-of-the-space-of-lines}~and~\ref{section:coadjoint-varieties}).
Second, we use the quantum Chevalley formula proved in \cite{FuWo,ChPe}.
The formulas for Littlewood--Richardson coefficients and the quantum Chevalley formula
have been implemented by the second author in the computer algebra system
SageMath \cite{sagemath} and are available in \cite{LRCalc,QCCalc}.
Our computations for coadjoint varieties in exceptional Dynkin types rely
on \cite{LRCalc} and all the necessary scripts ca be found at
\begin{center}
  \texttt{https://github.com/msmirnov18/bqh-coadjoint}
\end{center}

\subsection{Coadjoint variety in type $\rA$}

In type $\rA$ the situation is slightly different due to the fact that the
Picard rank is $2$ in this case. Indeed, in type $\rA_n$ the coadjoint variety is
the two-step flag variety $\Fl(1,n,n+1)$. The small quantum cohomology has the
following explicit description.
Here we deviate from our conventions on quantum cohomology
(see Section \ref{subsection:conventions-and-notation-for-quantum-cohomology})
and view the small quantum cohomology of $\Fl(1,n,n+1)$ as an algebra over
the polynomial ring $\bQ[q_1, q_2]$.

\begin{proposition}[{\cite[Proposition 7.2]{ChPe}}]
\label{proposition:coadjoint-type-A}
  The small quantum cohomology of $\Fl(1,n,n+1)$ is the quotient of
  \begin{equation*}
    \bQ[h_1, h_2, q_1, q_2]
  \end{equation*}
  modulo the relations
  \begin{equation}\label{eq:coadjoint-type-A-relations}
    \sum_{k=0}^n (-1)^{n-k} h_1^k h_2^{n-k} = q_1 + (-1)^n q_2 \quad \text{and} \quad h_1^{n+1} = q_1(h_1+h_2).
  \end{equation}
  If $q_1 + (-1)^n q_2 \neq 0$, then the algebra is semisimple. Otherwise the algebra
  has a unique non-reduced factor of the form $\bQ[\varepsilon]/(\varepsilon^n)$.
\end{proposition}

\begin{proof}
  The above presentation follows from the quantum Chevalley formula \cite{FuWo}.
  The only claim not explicitly contained in \cite[Proposition 7.2]{ChPe} is the
  type of the non-reduced point for the values of the parameters $q_1, q_2$
  satisfying $q_1 + (-1)^n q_2 = 0$. This follows easily by setting $q_1 + (-1)^n q_2 = 0$
  in \eqref{eq:coadjoint-type-A-relations}, then eliminating $h_2$ to get the relation
  \begin{equation*}
    h_1^n  \sum_{k=0}^n (q_1 - h_1^n)^{n-k} q_1^{k-n} = 0.
  \end{equation*}
  Since under our assumptions we have $q_1 \neq 0$, the claim follows.
\end{proof}

In particular, from Proposition \ref{proposition:coadjoint-type-A} we see that the locus in
the space of parameters $q_1, q_2$, where the small quantum cohomology of $\Fl(1,n,n+1)$
is not semisimple, depends on the parity of $n$.

\subsection{Derived category of coherent sheaves and results of \cite{KuSm21}}
\label{subsection:derived-category-of-coherent-sheaves-and-results-of-KuSm21}

Recall that according to Dubrovin's conjecture the generic semisimplicity of the
big quantum cohomology $\BQH(X)$ is equivalent to the existence of a full exceptional collection
in the bounded derived category $\Db(X)$ of coherent sheaves. Thus, the following
folklore conjecture fits very well with Conjecture \ref{conjecture:introduction-semisimplicity-of-BQH}.

\begin{conjecture}
  \label{conjecture:introduction-derived-category-G/P}
  Let $\rG$ be a semisimple
  algebraic group and $\rP \subset \rG$ a parabolic
  subgroup. Then the bounded derived category $\Db(\rG/\rP)$ of coherent sheaves
  has a full exceptional collection.
\end{conjecture}

We refer to \cite[§1.1]{KuPo} for an overview of the state-of-the art on this conjecture
from a few years ago. Since then the main advances are \cite{Fo19,Gu18,KuSm21,BeKuSm21,Sm21}.

\smallskip

In \cite{KuSm20, KuSm21} Alexander Kuznetsov and the second named author proposed
a conjecture \cite[Conjecture 1.3]{KuSm21} relating the structure of the small quantum
cohomology $\QH(X)$ of a Fano variety $X$ with generically semisimple big quantum
cohomology to the existence of certain Lefschetz exceptional collections in $\Db(X)$.
This conjecture could be seen as refinement of Dubrovin's conjecture.
For coadjoint varieties, using the additional knowledge of the structure of $\QH(X)$
provided by the present paper, an even more precise conjecture can be formulated
\cite[Conjecture 1.9]{KuSm21}.
\begin{conjecture}[{\cite[Conjecture 1.9]{KuSm21}}]
\label{conjecture:introduction-derived-coadjoint}
Let $X$ be the coadjoint variety of a simple algebraic group~$\rG$ over an algebraically
closed field of characteristic zero. Then $\Db(X)$ has an~$\Aut(X)$-invariant rectangular
Lefschetz exceptional collection with residual category~$\cR$ and
\begin{enumerate}
\item
if $\rT(\rG) = \rA_n$ and $n$ is even, then $\cR = 0$;
\item
otherwise, $\cR$ is equivalent to the derived category of representations of a quiver of Dynkin type~$\rTs(\rG)$.
\end{enumerate}
\end{conjecture}

Guided by the structure of $\QH(X)$ this conjecture has been proved in all cases
except for types $\rE_6$, $\rE_7$ and $\rE_8$. Specifically, this conjecture is
proved in type $\rA_n$ and $\rD_n$ in \cite{KuSm21}, in type $\rC_n$ in \cite[Appendix by A. Kuznetsov]{CMMPS}
and in type $\rF_4$ in \cite{BeKuSm21}. Types $\rB_n$ and $\rG_2$ are easy since
in this case $X^\coadj$ is a smooth quadric and the result is well known
(for example, see \cite[Example 1.6]{KuSm20}).

In \cite{KuSm21}, the necessary structural results to state the above conjecture
are summarized in \cite[Theorem 1.6]{KuSm21} with a reference to the present paper.
Thus, we feel obliged to give a proof.
This is the content of Section \ref{section:proof-of-KuSm21}.

\subsection{Structure of the paper}

The paper is structured as follows. Section \ref{section:preliminaries} contains
preliminaries on quantum cohomology and Schubert calculus. In Section \ref{section:geometry-of-the-space-of-lines}
we recall well-known facts that allow us to compute degree one Gromov--Witten
invariants of a rational homogeneous spaces using the Fano variety of lines.
These results suffice to compute the necessary Gromov--Witten invariants for coadjoint
varieties in simply-laced Dynkin types. In Section~\ref{section:coadjoint-varieties} we give some
generalities on coadjoint varieties and explain how to compute the necessary Gromov--Witten
for coadjoint varieties in non-simply laced Dynkin types. In Section \ref{section:type-D}
we consider the case of the coadjoint variety in type $\rD$. In Section \ref{section:type-E}
and Section \ref{section:type-F} we treat the conadjoint varieties in types
$\rE_6, \rE_7, \rE_8$ and $\rF_4$. In Section \ref{section:singularity-theory} we
relate the big quantum cohomology of coadjoint varieties to unfoldings of isolated
hypersurface singularities.

\medskip

\noindent {\bf Acknowledgements.} We thank Sergey Galkin, Vasily Golyshev,
and Alexander Kuznetsov for useful discussions and interest in our work.
We are indebted to Alexander Kuznetsov for his detailed comments on the first draft
of this paper.

M.S. is indebted to Pieter Belmans and Anton Mellit for their insights on programming
related issues that were crucial for the development of \cite{LRCalc,QCCalc}.
M.S. thanks the Hausdorff Research Institute for Mathematics and
the Max Planck Institute for Mathematics in Bonn for the great working
conditions during the preparation of this paper.

N.P. was supported by ANR Catore and ANR FanoHK. M.S. was partially supported by
the Deutsche Forschungsgemeinschaft (DFG, German Research Foundation) --- Projektnummer 448537907.

\section{Preliminaries}
\label{section:preliminaries}

In this section, we recall some definitions and basic properties of the small and
big quantum cohomology. We also fix some notation for algebraic groups and
Schubert varieties and prove some results on the behaviour of Schubert varieties
under equivariant maps between homogeneous spaces.
We will work over the field $\C$ of complex numbers.

\subsection{Conventions and notation for quantum cohomology}
\label{subsection:conventions-and-notation-for-quantum-cohomology}

Here we briefly recall the definition of the quantum cohomology ring for a smooth
projective variety $X$. To simplify the exposition and avoid introducing unnecessary
notation we impose from the beginning the following conditions on $X$: it is a Fano
variety of Picard rank~1 and $H^{odd}(X,\QQ)=0$. For a thorough introduction we refer
to \cite{Ma}. Recall that we will use Chow degrees which are half of cohomological
degrees. We write $\ptclass$ for the cohomology class of a point.

\subsubsection{Definition}
\label{SubSec.: Def of QH}

Let us fix a graded basis $\Delta_0, \dots , \Delta_s$ in $H^*(X, \QQ)$ and dual
linear coordinates $t_0, \dots, t_s$. It is customary to choose $\Delta_0=1$.
For cohomology classes we use the Chow grading, i.e. we divide the topological
degree by two. Further, for variables $t_i$ we set $\deg(t_i)=1-\deg(\D_i)$.

Let $R$ be the ring of formal power series $\QQ[[q]]$, $k$ its field of fractions,
and $K$ an algebraic closure of $k$. We set  $\deg(q)= \text{index} \, (X)$,
which is the largest integer $r$ such that $-K_X = rH$ for some ample divisor
$H$ on $X$, where $K_X$ is the canonical class of $X$.

The genus zero Gromov--Witten potential of $X$ is an element $F \in R[[t_0, \dots, t_s]]$
defined by the formula
\begin{align}\label{eq:GW-potential}
&F = \sum_{(i_0, \dots , i_s)}  \langle \Delta_0^{\otimes i_0}, \dots, \Delta_s^{\otimes i_s} \rangle \frac{t_0^{i_0} \dots t_s^{i_s} }{i_0!\dots i_s!},
\end{align}
where
\begin{equation}\label{eq:GW-potential-coefficients}
\langle \Delta_0^{\otimes i_0}, \dots, \Delta_s^{\otimes i_s} \rangle = \sum_{d=0}^{\infty} \langle \Delta_0^{\otimes i_0}, \dots, \Delta_s^{\otimes i_s} \rangle_d q^d,
\end{equation}
and $\langle \Delta_0^{\otimes i_0}, \dots, \Delta_s^{\otimes i_s} \rangle_d$ are
rational numbers called Gromov--Witten invariants of $X$ of degree~$d$. With respect
to the grading defined above $F$ is homogeneous of degree $3 - \dim X$.
Since $X$ is Fano, \eqref{eq:GW-potential-coefficients} is a polynomial in $q$.

Using \eqref{eq:GW-potential} one defines the \textit{big quantum cohomology ring} of $X$.
Namely, let us endow the $K[[t_0, \dots, t_s]]$-module
\begin{align*}
\BQH(X) = H^*(X, \bQ)\otimes_{\bQ} K[[t_0, \dots, t_s]]
\end{align*}
with a ring structure by setting
\begin{equation}\label{eq:big-quantum-product}
\Delta_a \star \Delta_b = \sum_c \frac{\partial^3 F}{\partial t_a \partial t_b \partial t_c} \Delta^c,
\end{equation}
on the basis elements and extending to the whole $\BQH(X)$ by $K[[t_0, \dots, t_s]]$-linearity.
Here  $\D^0, \dots, \D^s$ is the basis of $H^*(X, \QQ)$ dual to  $\D_0, \dots, \D_s$ with respect
to the Poincar\'e pairing. It is well known that~\eqref{eq:big-quantum-product}
makes $\BQH(X)$ into a commutative, associative, graded $K[[t_0, \dots, t_s]]$-algebra
with the identity element $\D_0$.

The algebra $\BQH(X)$ is called the \textit{big} quantum cohomology algebra of $X$
to distinguish it from a simpler object called the \textit{small} quantum cohomology
algebra which is the quotient of $\BQH(X)$ with respect to the ideal $(t_0, \dots, t_s)$.
We will denote the latter $\QH(X)$ and use $\starz$ instead of $\star$ for the product
in this algebra. It is a finite dimensional $K$-algebra. Equivalently one can say that
\begin{equation*}
  \QH(X)=H^*(X,\bQ) \otimes_{\bQ} K
\end{equation*}
as a vector space, and the $K$-algebra structure is defined by putting
\begin{equation}\label{eq:small-quantum-product}
\Delta_a \starz \Delta_b = \sum_c \langle \D_a, \D_b, \D_c \rangle \Delta^c.
\end{equation}

\begin{remark}
We are using a somewhat non-standard notation $\BQH(X)$ for the big quantum cohomology
and $\QH(X)$ for the small quantum cohomology to stress the difference between the two.
Note that this notation is different from the one used in \cite{GMS} and is closer to
the notation of \cite{Pe}.
\end{remark}

\begin{remark}\label{SubSubSec.: Remark on difference of notation with Manin's book}

The above definitions look slightly different from the ones given in~\cite{Ma}.
The differences are of two types. The first one is that $\QH(X)$ and $\BQH(X)$
are in fact defined already over the ring $R$ and not only over $K$. We pass to $K$
from the beginning, since in this paper we are only interested in generic semisimplicity
of quantum cohomology. The second difference is that in some papers on quantum
cohomology one unifies the coordinate $q$ with the coordinate $t_1$ which is dual
to $H^2(X, \bQ)$, but the resulting structures are equivalent.
\end{remark}

\subsubsection{Deformation picture}
\label{SubSec.: Def. Picture}

The small quantum cohomology, if considered over the ring $R$
(cf. Remark~\ref{SubSubSec.: Remark on difference of notation with Manin's book}),
is a deformation of the ordinary cohomology algebra, i.e. if we put $q=0$, then
the quantum product becomes the ordinary cup-product. Similarly, the big quantum
cohomology is an even bigger deformation family of algebras. Since we work not
over $R$ but over $K$, we lose the point of classical limit but still retain the
fact that $\BQH(X)$ is a deformation family of algebras with the special fiber being $\QH(X)$.

In this paper we view $\spec(\BQH(X))$ as a deformation family of zero-dimensional
schemes over $\spec (K[[t_0,\dots, t_s]])$. In the base of the deformation we
consider the following two points: the origin (the closed point given by the maximal
ideal $(t_0, \dots , t_s)$) and the generic point $\eta$. The fiber of this family
over the origin is the spectrum of the small quantum cohomology $\spec(\QH(X))$.
The fiber over the generic point will be denoted by $\spec(\BQH(X)_{\eta})$.
It is convenient to summarize this setup in the diagram
\begin{align}\label{Eq.: BQH family}
\vcenter{
\xymatrix{
\spec(\QH(X)) \ar[d] \ar[r]   &    \spec(\BQH(X)) \ar[d]^{\pi} & \ar[l] \spec(\BQH(X)_{\eta}) \ar[d]^{\pi_{\eta}} \\
\spec (K)    \ar[r]      & \spec (K[[t_0, \dots, t_s]])  &   \ar[l] \eta
}
}
\end{align}
where both squares are Cartesian.

By construction $\BQH(X)$ is a free module of finite rank over $K[[t_0, \dots,t_s]]$.
Therefore, it is a noetherian semilocal $K$-algebra which is flat and finite over
$K[[t_0, \dots,t_s]]$. Note that neither $K[[t_0, \dots,t_s]]$ nor $\BQH(X)$ are
finitely generated over the ground field $K$. Therefore, some extra care is required
in the standard commutative algebra (or algebraic geometry) constructions.
For example, the notion of smoothness is one of such concepts.

Often we consider not the full deformation family $\BQH(X)$ but only a subfamily
that corresponds to deformation directions along some basis elements $\Delta_{i_1}, \dots, \Delta_{i_m}$
defined by
\begin{equation*}
  \BQH_{\Delta_{i_1}, \dots, \Delta_{i_m}}(X) \coloneqq
  \BQH(X)/\left( t_0, \dots, \hat{t}_{i_1}, \dots, \hat{t}_{i_m}, \dots, t_s \right),
\end{equation*}
where $\hat{t}_{i_l}$ means that the variable $\hat{t}_{i_l}$ is ommitted.

\subsubsection{Semisimplicity}
\label{SubSec.: Semisimplicity}

Let $A$ be a finite dimensional commutative algebra over a field $F$ of characteristic zero.
It is called \textit{semisimple} if it is a product of fields. Equivalently,
the algebra $A$ is semisimple if the scheme $\spec(A)$ is reduced. Another equivalent
condition is to require the morphism $\spec(A) \to \spec(F)$ to be smooth.

\begin{definition}
We say that $\BQH(X)$ is \textit{generically semisimple} if $\BQH(X)_{\eta}$ is a semisimple algebra
over $\eta$.
\end{definition}

\subsection{Notation for algebraic groups}
\label{subsection:notation-for-algebraic-groups}

Let $\rG$ be a reductive algebraic group, let $\rT$ be a maximal torus in $\rG$
and let $\rB$ be a Borel subgroup containing $T$. A parabolic subgroup
$\rP \subset \rG$ is called standard if $\rB \subset \rP$.

Let $\rP$ be a parabolic subgroup, we write $R_u(\rP)$ for the unipotent radical of $\rP$.
If $\rP$ contains $\rT$ (for example if $\rP$ is standard), then $\rP$ contains a
unique reductive subgroup $\rL_\rP$ such that $\rT \subset \rL_\rP \subset \rP$ and $\rP = \rL_\rP R_u(\rP)$
with $\rL_\rP \cap R_u(\rP) = \{1\}$, the subgroup $\rL_\rP$ is called
\textsf{the Levi subgroup of} $\rP$.

For $\rP$ a standard parabolic subgroup, there exists a unique parabolic subgroup,
denoted $\rP^-$ and called \textsf{the opposite parabolic subgroup} such that
$\rP^- \cap \rP = \rL_\rP$. In the case $\rP = \rB$, then $\rB^-$ is the opposite
Borel subgroup and is characterized by $\rB \cap \rB^- = \rT$.

Denote by $\Phi$ the root system of $(\rG,\rT)$ and by $\Delta$ the set of simple
roots defined by $\rB$ such that the roots of $\rB$ is the set $\Phi^+$ of positive
roots. We set $\Phi^- = - \Phi^+$. If ${\rm rk}(\rG)$ is the semisimple rank of $\rG$, we will
number the simple roots $\Delta = \{ \alpha_i \ | \ i \in [1,{\rm rk}(\rG)] \}$ as in Bourbaki
\cite[Tables]{Bo}. If $\rP$ is a standard parabolic subgroup, then we write $\Phi_\rP$
for the root system of $\rL_\rP$. Set $\Phi_\rP^+ = \Phi_\rP \cap \Phi^+$  and
$\Phi_\rP^- = \Phi_\rP \cap \Phi^-$. If $\rP$ is a maximal standard parabolic subgroup,
we define $\alpha_\rP$ as the unique simple root such that $- \alpha_\rP \not \in \Phi_\rP$.

Denote by $\rW$ the Weyl group of the pair $(\rG,\rT)$ and by $s_\alpha \in \rW$
the reflection associated to any root $\alpha \in \Phi$. For $\alpha_i \in \Delta$
a simple root we will also use the notation $s_i := s_{\alpha_i}$ for the corresponding
simple reflection. Simple reflections define a length in $\rW$ and we let $w_0$ be
the longest element in $W$. Recall that $w_0^{-1} = w_0$. For $\rP$ a standard
parabolic subgroup, we denote by $\rW_\rP$ the Weyl group of the pair $(\rP,\rT)$.
Note that this is also the Weyl group of the pair $(\rL_\rP,\rT)$. Let $w_{0,\rP}$
be the longest element of $\rW_\rP$. We have $w_{0,\rP}^{-1} = w_{0,\rP}$. Denote
by $\rW^\rP \subset \rW$ the set of minimal length representatives of the quotient
$\rW/\rW_\rP$ and by $w_{0}^\rP$ the longest element of $\rW^\rP$. If $\rQ$ is
another standard parabolic subgroup, we denote by $\rW_{\rP}^{\rQ} \subset \rW_\rP$
the set of minimal length representatives of $\rW_{\rP}/\rW_{\rP \cap \rQ}$.
Let $w_{0,\rP}^{\rQ} \in \rW_{\rP}^{\rQ}$ be the longest element.

\subsection{Basics on Schubert varieties and classes}
\label{subsection:schubert-classes}

Let $\rP \subset \rG$ be a standard parabolic subgroup and consider the quotient
$X = \rG/\rP$, which is a rational projective variety homogeneous under the natural
left action of $\rG$.

The variety $X = \rG/\rP$ has a \textsf{Bruhat decomposition} --- this is a cell
decomposition with cells given by the orbits of the Borel subgroup $\rB$. Namely,
for $w \in \rW$ define the corresponding \textsf{Schubert cell} as $\Xo_w = \rB w.\rP$.
As $\Xo_w$ only depends on the coset of $w$ modulo $\rW_{\rP}$, we have a disjoint union decomposition
\begin{equation*}
  X = \coprod_{w \in \rW^\rP} \Xo_w.
\end{equation*}
The closures $X_w = \overline{\rB w. \rP}$ of the Schubert cells are called \textsf{Schubert varieties}.
The cohomology classes $\sigma_w = [X_w]$ form a basis of the singular cohomology
$H^*(X,\bQ)$ and are called \textsf{Schubert classes}.
The set $\rW^\rP$ of minimal length coset respresentatives is partially ordered
by the inclusion of Schubert varieties: $u \leq v$ if and only if $X_u \subset X_v$.
This is the \textsf{Bruhat order}.

Replacing in the above definitions $\rB$ by $\rB^-$
one defines \textsf{opposite Schubert cells} $\Xo^w$ that
give rise to the opposite Bruhat decomposition $X = \coprod_{w \in \rW^\rP} \Xo^w$.
Similarly, one defines \textsf{opposite Schubert varieties} $X^w$, and their cohomology
classes $\sigma^w = [X^w]$, called \textsf{opposite Schubert classes}, also form
a basis of the singular cohomology $H^*(X,\bQ)$.

The following lemma collects some basic well-known facts about Schubert varieties.

\begin{lemma}\label{lemma:basics-schubert-varieties-1}
  \
  \begin{enumerate}
    \item For any $u \in \rW$ there exists a unique factorisation $u = u^\rP u_\rP$
    with $u^\rP \in \rW^\rP$ and $u_\rP \in \rW_\rP$. Moreover, for this factorisation
    we have $\ell(u) = \ell(u^P) + \ell(u_P)$.

    \smallskip

    \item For the longest element $w_0 \in \rW$ the above factorisation is of
    the form
    \begin{equation*}
      w_0 = w_0^\rP w_{0,\rP}.
    \end{equation*}

    \item For any $u \in \rW^{\rP}$ we have
    \begin{equation*}
      \ell(u) = \dim X_u = \codim X^u.
    \end{equation*}

    \item For any $u \in \rW^{\rP}$ the element $u^\vee \coloneqq w_0 u w_{0,\rP}$
    lies in $\rW^\rP$ and we have the equality
    \begin{equation*}
      X^u = w_0. X_{u^\vee}
    \end{equation*}

    \item For any $u,v \in \rW^\rP$ the scheme-theoretic intersection $X_u^v \coloneqq X_u \cap X^v$
    is reduced. The intersection $X_u^v$ is non-empty if and only if $v \leq u$.
    If non-empty, $X_u^v$ is irreducible of dimension $\ell(u) - \ell(v)$.
  \end{enumerate}
\end{lemma}

\begin{proof}
For statements about Weyl group elements we refer to \cite{bb}. For statements about
Schubert varieties to \cite{BrionLakshmibai}.
\end{proof}

Let $\rR \subset \rG$ be a standard parabolic subgroup contained in $\rP$.
We define $Z = \rG/\rR$ and consider the canonical projection morphism
\begin{equation}\label{eq:canonical-projection-map}
  \pi \colon Z \to X.
\end{equation}
which is a fiber bundle with fiber $\rP/\rR$. Note that we have a natural inclusion
of sets of minimal length coset representatives
\begin{equation}\label{eq:inclusion-coset-reps}
  \rW^\rP \subset \rW^\rR
\end{equation}
indexing Schubert varieties on $X$ and $Z$ respectively.

Let $e \in X$ be the point corresponding the coset of the identity in $\rG$
and note that the fiber $\pi^{-1}(e)$ is canonically identified with $\rP/\rR$.
We define
\begin{equation}\label{eq:canonical-projection-fiber}
  Y \coloneqq \pi^{-1}(e).
\end{equation}
There is a natural isomorphism
\begin{equation*}
  Y = \rL / \rR_\rL,
\end{equation*}
where $\rL$ is the Levi subgroup of $\rP$ containing $\rT$ and $\rR_\rL = \rR \cap \rL$.
The group $\rL$ is reductive and $\rR_\rL \subset \rL$ is a
parabolic subgroup. Moreover, we can also define Borel subgroups of $\rL$ as
\begin{equation*}
  \rB_\rL = \rB \cap \rL \quad \text{and} \quad \rB^-_\rL = \rB^- \cap \rL.
\end{equation*}
Thus, we have all the data required to define Schubert varieties in $Y$.
Note also that for the Weyl group of the pair $(\rL,\rT)$ we have
$\rW_\rL = \rW_\rP$ and $\rW_\rL^{\rR_\rL} = \rW_\rP^\rR$.
We also note that there is a natural inclusion
\begin{equation*}
  \rW_\rP^\rR \subset \rW^\rR.
\end{equation*}

The next lemma collects some well-known facts about images and preimages of
Schubert varieties under the map \eqref{eq:canonical-projection-map} and describes
how opposite Schubert varieties in $Z$ intersect with the fiber \eqref{eq:canonical-projection-fiber}.
More general results in this direction can be found in \cite{BCMP}.

\begin{lemma}\label{lemma:basics-schubert-varieties-2}
  \
  \begin{enumerate}
    \item For any $u \in \rW$ we have $(u_\rP)^\rR = (u^\rR)_\rP$. We denote this element by $u_\rP^\rR$.

    \smallskip

    \item For any $u \in \rW^\rR$ we have $\pi(Z^u) = X^{u^\rP}$ and $\pi(Z_u) = X_{u^\rP}$.
    If moreover $u \in \rW^\rP$, then the morphisms $\pi \colon Z_u \to X_u$ and
    $\pi \colon Z^{u w_{0,\rP}^\rR} \to X^u$ are birational.

    \smallskip

    \item For $u \in \rW^\rP$, we have $\pi^{-1}(X^u) = Z^u$ and $\pi^{-1}(X_u) =Z_{uw_{0,\rP}^\rR}$.

    \smallskip

    \item For $u \in \rW^\rR$ the intersection $Y \cap Z^u$ is non-empty if and
    only if $u \leq w_{0,\rP}^\rR$, in which case we have $Y \cap \Z^u = Y^u$.
    The condition $u \leq w_{0,\rP}^\rR$ is equivalent to $u \in \rW^\rR_\rP$.
  \end{enumerate}
\end{lemma}

\begin{proof}
  (1) This follows from the uniqueness of the factorisation
  in Lemma \ref{lemma:basics-schubert-varieties-1}(1).

  \smallskip

  (2) The first part of the claim follows immediately from the definition of
  Schubert varieties. For the birationality of the first map we refer to \cite[Section 13.8]{jantzen}).
  The birationality of the second map follows from the birationality of the first
  and Lemma~\ref{lemma:basics-schubert-varieties-1}(4).

  \smallskip

  (3) This is an immediate consequence of the definition of Schubert varieties.

  \smallskip

  (4) By Part (3) we have $Y = \pi^{-1}(X_1) = Z_{w_{0,\rP}^\rR}$.
  By Lemma~\ref{lemma:basics-schubert-varieties-1}(5) the intersection
  $Y \cap Z^u$ is non-empty if and only if $u \leq w_{0,\rP}^\rR$.
  Since $w_{0,\rP}^\rR$ is the maximal element in $\rW_\rP^\rR$, the condition
  $u \leq w_{0,\rP}^\rR$ is clearly equivalent to the inclusion $u \in \rW^\rR_\rP$.

  To show $Y \cap \Z^u = Y^u$ we proceed as follows. For $u \in \rW^\rR_\rP$ we
  consider the intersection $\Zo^u \cap Y = \rB^- u.\rR \cap \rP .\rR$. Since
  $u \in \rW^\rR_\rP \subset \rW_\rP$ we can rewrite it as $\Zo^u \cap Y = (\rB^-\cap \rP) u.\rR$.
  Using the identifications $\rB^- \cap \rP = \rB^{-}_\rL$ and $\rP/\rR = \rL/\rR_\rL$,
  we obtain $\Zo^u \cap Y = \rB^{-}_\rL u . \rR_\rL = \Yo^u$. Now the claim follows
  by taking closures.
\end{proof}

\subsection{Borel's presentation}
\label{subsection:borel-presentation}

Let $\rG$ be a reductive group and $\rP \subset \rG$ be a parabolic subgroup. In this subsection, we briefly recall Borel's presentation of the cohomology ring $H^*(\rG/\rP,\QQ)$.

First assume that $\rP = \rB$ is a Borel subgroup of $\rG$ and let $\Lambda$ be the group of characters of $\rB$. For any character $\lambda \in \Lambda$, define the line bundle $L_\lambda \to \rG/\rB$ by $L_\lambda = (\rG \times \C)/\rB$ where $\rB$ acts via $(g,z).b = (gb,\lambda(b)^{-1}z)$ for $g \in \rG$, $z \in \C$ and $b \in \rB$. This induces a linear map $c : \Lambda \to H^2(\rG/\rB,\QQ)$ via $c(\lambda) = c_1(L_\lambda)$. This map extends to a $\QQ$-algebra morphism called the \textsf{characteristic map}
$$c : \QQ[\Lambda] \to H^*(\rG/\rB,\QQ).$$
Note that $\Lambda$ is also the group of characters of a maximal torus $\rT \subset \rB$ and therefore the Weyl group $\rW$ of the pair $(\rG,\rT)$ acts on $\Lambda$. A consequence of Borel's presentation of equivariant cohomology imples that the characteristic map is surjective and that its kernel is the ideal $\QQ[\Lambda]^{\rW}_+$
generated by the
non-constant invariant elements (see for example \cite[Proposition 26.1]{Bor} or \cite[Corollaire 2, page 292]{De}). We thus have an isomorphism
$$H^*(\rG/\rB,\QQ) \simeq \QQ[\Lambda]/\QQ[\Lambda]^\rW_+.$$

For $\rP \supset \rB$ a general parabolic subgroup, we use the projection map $p : \rG/\rB \to \rG/\rP$ which induces an inclusion $p^* : H^*(\rG/\rP,\QQ) \to H^*(\rG/\rB,\QQ)$. This induces an isomorphism
$$H^*(\rG/\rP,\QQ) \simeq \QQ[\Lambda]^{\rW_{\rP}}/\QQ[\Lambda]^\rW_+.$$
See for example \cite[Theorem 26.1]{Bor} or \cite[Theorem 5.5]{BGG}). We will deform these presentations to get presentations for the small and the big quantum cohomology of adjoint and coadjoint varieties.

\section{Geometry of the space of lines}
\label{section:geometry-of-the-space-of-lines}

In this section we recall general methods for computing degree one Gromov--Witten
invariants of homogeneous varieties of Picard rank one.

Let $\rP \subset \rG$ be a maximal standard parabolic subgroup and $\alpha_{\rP}$
the corresponding simple root. Let $X = \rG/\rP$ and let $\F(X)$ be the Fano
variety of lines on $X$. We have the diagram
\begin{equation}
  \label{eq:universal-family-of-lines}
  \begin{aligned}
    \xymatrix{
\Z(X) \ar[r]^-p  \ar[d]_-q & X \\
\F(X) \\}
  \end{aligned}
\end{equation}
where $\Z(X) = \{ (x,\ell) \in X \times \F(X) \ | \ x \in \ell \}$ is the universal
family of lines and $p \colon \Z(X) \to X$ and $q \colon \Z(X) \to \F(X)$ are
canonical projections.

The variety $\F(X)$ is described by the following theorem of Landsberg and Manivel.

\begin{theorem}[{\cite[Theorem 4.3]{LaMa03}}]
  \label{theorem:landsberg-manivel}
  Let $\rQ \subset \rG$ be the standard parabolic subgroup defined by those nodes
  of the Dynkin diagram of $\rG$ that are connected to the node of $\alpha_\rP$.
  \begin{enumerate}
  \item If $\alpha_\rP$ is a long root, then there is an isomorphism $\F(X) = \rG/\rQ$,
  and the universal family
  \eqref{eq:universal-family-of-lines} is given by
  \begin{equation}
    \label{eq:universal-family-of-lines-long-root}
    \begin{aligned}
      \xymatrix{
        \rG/\rR \ar[r]^-p  \ar[d]_-q & \rG/\rP \\
        \rG/\rQ \\}
    \end{aligned}
  \end{equation}
  where $\rR = \rP \cap \rQ$ and both maps are canonical $\rG$-equivariant projections.

  \smallskip

  \item If $\rG$ is not simply laced and $\alpha_\rP$ is a short root,
  then $\F(X)$ is has two $\rG$-orbits, and the closed orbit is isomorphic to $\rG/\rQ$.
  \end{enumerate}
\end{theorem}

\bigskip

Degree $d$ Gromov--Witten invariants of $X$ are defined via intersection numbers on the moduli
space of stable maps $\overline{M}_{0,n}(X,d)$. In the case $d = 1$ we can express
these invariants in terms of some interesection numbers on the Fano variety of
lines $\F(X)$. The following lemma makes this precise.

\begin{lemma}\label{lemma:quantum-to-classical-new}
  Let $w_1, \dots, w_n \in \rW^{\rP}$ and consider the corresponding opposite
  Schubert varieties $X^{w_1}, \dots, X^{w_n} \subset X$ and Schubert classes
  $\sigma^{w_1}, \dots, \sigma^{w_n} \in H^*(X, \bQ)$. Let us further assume that
  \begin{equation}\label{eq:gw-invariant-dimension-axiom}
    \sum_{i = 1}^n \deg(\sigma^{w_i})
      = n - 3 + (-K_X, \ell) + \dim X,
  \end{equation}
  where $\ell$ is the class of a line in $X$ and $\deg(\gamma)$
  stands for the cohomological degree.

  \begin{enumerate}
    \item For general elements $g_1, \dots, g_n \in \rG$ the scheme theoretic intersection
    \begin{equation}\label{eq:gw-intersection-moduli-of-stable-maps-as-subset}
      ev_1^{-1}(g_1X^{w_1}) \cap \dots \cap ev_n^{-1}(g_nX^{w_n})
    \end{equation}
    is either empty or is a finite number of reduced points supported in the locus
    of automorphism-free stable maps with irreducible domain, and we have
    \begin{equation}\label{eq:gw-intersection-moduli-of-stable-maps}
      \scal{\sigma^{w_1}, \dots, \sigma^{w_n}}_1 = \# \left( ev_1^{-1}(g_1X^{w_1}) \cap \dots \cap ev_n^{-1}(g_nX^{w_n}) \right).
    \end{equation}

    \medskip

    \item For general elements $g_1, \dots, g_n \in \rG$ the intersection in $\F(X)$
    of the translates
    \begin{equation}\label{eq:gw-intersection-fano-variety-of-lines-as-subset}
      g_1 \left( qp^{-1}(X^{w_1}) \right) \cap \dots \cap g_n \left( qp^{-1}(X^{w_n}) \right)
    \end{equation}
    is either empty or is a finite number of reduced points, and we have
    \begin{equation}\label{eq:gw-intersection-fano-variety-of-lines}
      \scal{\sigma^{w_1}, \dots, \sigma^{w_n}}_1 = \# \left( g_1 \left( qp^{-1}(X^{w_1}) \right) \cap \dots \cap g_n \left( qp^{-1}(X^{w_n}) \right) \right).
    \end{equation}
  \end{enumerate}
\end{lemma}

\begin{proof}
  (1)  This is a special case of \cite[Lemma 14]{FuPa}.

  \medskip

  (2) Recall that we have
  \begin{equation}\label{eq:dimension-fano-variety-of-lines}
    \dim \F(X) = (-K_X, \ell) + \dim X - 3.
  \end{equation}
  Since the map $q$ is a $\bP^1$-bundle, for any $w \in \rW^\rP$ we have
  \begin{equation*}
    \codim_X(X^w) \geq \codim_{\F(X)}(qp^{-1}(X^w)) \geq \codim_X(X^w) - 1.
  \end{equation*}
  Hence, using \eqref{eq:gw-invariant-dimension-axiom} we obtain the inequality
  \begin{equation}\label{eq:intersection-fano-variety-of-lines-codimension-inequality}
    \sum_i \codim_{\F(X)}(qp^{-1}(X^{w_i})) \geq \dim \F(X).
  \end{equation}
  Using this inequality, Kleiman's transversality and Theorem \ref{theorem:landsberg-manivel},
  one easily argues that for general $g_1, \dots, g_n \in \rG$ the interesection of the translates
  $g_1 \left( qp^{-1}(X^{w_1}) \right), \dots, g_n \left( qp^{-1}(X^{w_n}) \right)$
  is either empty or is a finite number of reduced points contained in the open
  $\rG$-orbit of $\F(X)$.

  On the one hand, a point of the intersection~\eqref{eq:gw-intersection-fano-variety-of-lines-as-subset}
  is given by a line $\ell \subset X$ that non-trivially intersects the translates
  $g_1X^{w_1} , \dots , g_nX^{w_n}$. On the other hand, we know by Part (1) that
  a point of the intersection~\eqref{eq:gw-intersection-moduli-of-stable-maps-as-subset}
  is given by a line $\ell \subset X$ that non-trivially intersects the translates
  $g_1X^{w_1} , \dots , g_nX^{w_n}$ together with a choice of points $p_i \in \ell \cap g_i X^{w_i}$.
  Thus, we need to show that the additional choice of points $p_i \in \ell \cap g_i X^{w_i}$
  does not affect the result.

  Since any Schubert variety is a linear section (see \cite[Theorem 3.11(ii)]{Ra}),
  the intersections $\ell \cap g_i X^{w_i}$ are either a reduced point or the line $\ell$ itself.
  If all intersections $\ell \cap g_i X^{w_i}$ are a reduced point,
  then there is a unique choice of the points $p_i \in \ell \cap g_i X^{w_i}$ and,
  therefore, we have a bijection between \eqref{eq:gw-intersection-moduli-of-stable-maps-as-subset}
  and \eqref{eq:gw-intersection-fano-variety-of-lines-as-subset}. If for some $i$
  we have $\ell \cap g_i X^{w_i} = \ell$ for $g_i$ general in $\rG$, then we have $X^{w_i} = X$ and the GW
  invariant $\scal{\sigma^{w_1}, \dots, \sigma^{w_n}}_1$ vanishes by \cite[III.5.3]{Ma}.
  Similarly, in this case the inequality \eqref{eq:intersection-fano-variety-of-lines-codimension-inequality}
  is strict and the intersection \eqref{eq:gw-intersection-fano-variety-of-lines-as-subset}
  is empty.
\end{proof}

Define the variety $\Fpt(X)$ of lines in $X$ passing through a point as
\begin{equation*}
  \Fpt(X) \coloneqq qp^{-1}(e) = \{ \ell \in \F(X) \ | \ e \in \ell\},
\end{equation*}
where $e$ is the point in $X$ given by the coset of the identity element in $\rG$.
By definition $\Fpt(X)$ is a subvariety of $\F(X)$ and we denote the natural
inclusion by $i \colon \Fpt(X) \to \F(X)$. Moreover, the restriction
$q_{|p^{-1}(e)} \colon p^{-1}(e) \to q(p^{-1}(e))$ is an isomorphism and we can identify
the variety $\Fpt(X)$ with the fiber $p^{-1}(e)$. It is convenient to collect
this data into the commutative diagram
\begin{equation}\label{diagram:variety-of-lines-through-a-point}
  \xymatrix{
    \Fpt(X) \ar[r]^\varphi \ar[rd]_i  & \Z(X) \ar[r]^p \ar[d]^q & X  \\
                                                     & \F(X)
  }
\end{equation}
where $\varphi$ denotes the embedding of $\Fpt(X)$ as the fiber $p^{-1}(e)$ and
we have $q \circ \varphi = i$.

\begin{corollary}
  \label{corollary:quantum-to-classical}
  For $w_1, \dots, w_n \in \rW^{\rP}$ we have
  \begin{equation}\label{eq:corollary-quantum-to-classical-1}
    \scal{\sigma^{w_1}, \dots, \sigma^{w_n}}_1 =
    \deg_{\rF(X)}(q_*p^*(\sigma^{w_1}) \cup \dots \cup q_*p^*(\sigma^{w_n}))
  \end{equation}
  and
  \begin{equation}\label{eq:corollary-quantum-to-classical-2}
    \scal{\ptclass, \sigma^{w_1}, \dots, \sigma^{w_n}}_1 =
    \deg_{\Fpt(X)}(i^*q_*p^*(\sigma^{w_1}) \cup \dots \cup i^*q_*p^*(\sigma^{w_n})).
  \end{equation}
\end{corollary}

\begin{proof}
  If the codimension condition \eqref{eq:gw-invariant-dimension-axiom} is satisfied,
  then \eqref{eq:corollary-quantum-to-classical-1} holds by Lemma \ref{lemma:quantum-to-classical-new}(2).

  If the codimension condition \eqref{eq:gw-invariant-dimension-axiom} is not satisfied,
  then left-hand side of \eqref{eq:corollary-quantum-to-classical-1} vanishes
  according to the properties of GW invariants. To see the vanishing of the right-hand
  side we consider two cases:
  \begin{enumerate}
    \item If for some $i$ we have $w_i = 1$, then $q_*p^*(\sigma^{w_i}) = 0$ and
    the right-hand side obviously vanishes.

    \item If $w_i \neq 1$ for all $i$, then arguing by codimension as in
    the proof of Lemma~\ref{lemma:quantum-to-classical-new}(2) we also get the desired vanishing.
  \end{enumerate}

  The equality \eqref{eq:corollary-quantum-to-classical-2} follows
  from \eqref{eq:corollary-quantum-to-classical-1} by projection formula.
\end{proof}

\bigskip

Let us now make the above formulas for GW invariants even more explicit in the
case, where the Fano variety of lines $\F(X)$ is homogeneous under the group $\rG$.
Thus, in view of Theorem~\ref{theorem:landsberg-manivel},
let us assume that the root $\alpha_\rP$ is a long root.
In this case not only the variety $\F(X) = \rG/\rQ$ is homogeneous, but also the
variety $\Fpt(X)$ of lines passing through a point is homogenous with respect
to a subgroup of $\rG$. Indeed, using the identification of $\Fpt(X)$ with the
fiber $p^{-1} = \rP / \rR$, we see that $\Fpt(X)$ is homogeneous under the Levi
subgroup of~$\rP$. In this case the diagram \eqref{diagram:variety-of-lines-through-a-point}
becomes
\begin{equation}\label{diagram:identification-long-root-universal-family}
  \xymatrix{
    \rP/ \rR \ar[r]^\varphi \ar[rd]_i  & \rG/\rR \ar[r]^p \ar[d]^q & \rG/\rP  \\
                                                     & \rG/\rQ
  }
\end{equation}

If no confusion arises, we denote $\F = \F(X)$ and $\Fpt = \Fpt(X)$.

\begin{lemma}\label{lemma:quantum-to-schubert-calculus-long-root}
  Under the assumptions of the diagram \eqref{diagram:identification-long-root-universal-family}
  we have the following.
  \begin{enumerate}
    \item For w $\in \rW^\rP \setminus \{1\}$ we have
    \begin{equation*}
      w^{\rQ} = ws_\rP \in \rW^\rQ \quad \text{and} \quad \ell(w^{\rQ}) = \ell(w) - 1.
    \end{equation*}

    \medskip

    \item For $w \in \rW^\rP$ we have
    \begin{equation}\label{eq:pull-push-universal-family}
      q_*p^* \left( \sigma_X^w \right) =
      \begin{cases}
        \sigma_F^{ws_\rP} & \text{if} \quad w \neq 1, \\
        0 & \text{if} \quad w = 1.
      \end{cases}
    \end{equation}

    \medskip

    \item For $w_1,\cdots,w_n \in \rW^\rP \setminus \{ 1 \}$ we have
    \begin{equation*}
       \scal{\sigma^{w_1}_X, \dots, \sigma^{w_n}_X}_1 = \deg_{\rF} \left(\sigma^{w_1 s_\rP}_F \cup \dots \cup \sigma^{w_n s_\rP}_F \right).
     \end{equation*}

     \medskip

     \item For $w_1,\cdots,w_n \in \rW^\rP \setminus \{ 1 \}$ we have
     \begin{equation*}
       \scal{\ptclass,\sigma_X^{w_1},\dots,\sigma_X^{w_n}}_1 =
       \begin{cases}
         \deg_{\Fpt} \left( \cup_{i=1}^n \sigma_{\Fpt}^{w_is_\rP} \right)
         \quad \text{if} \quad w_is_\rP \leq w_{0,\rP}^{\rR} \text{  for all  } i, \\
         0 \quad \text{otherwise}.
       \end{cases}
     \end{equation*}
  \end{enumerate}
\end{lemma}

\begin{proof}
  (1) Define $v \in \rW$ by the equality $v = w s_{\rP}$. By \cite[Proposition 1.4.2]{bb}
  we have $l(v) = l(w) \pm 1$.
  Let $k = \ell(w)$ and let $s_{i_1} \dots s_{i_k}$ be a reduced expression for $w$.
  Since $w \in \rW^\rP$, then by \cite[Lemma 2.4.3]{bb} we must necessarily have
  $s_{i_k} = s_{\rP}$. Therefore, $s_{i_1} \dots s_{i_{k-1}}$ is a reduced expression
  for $v$ and we have $l(v) = l(w) - 1$. Moreover, appending $s_{i_k} = s_{\rP}$
  as above gives a bijection between reduced expression for $v$ and $w$.

  To show the inclusion
  $v \in \rW^\rQ$ it is enough prove that $s_{i_{k-1}}$ corresponds to a
  vertex of the Dynkin diagram of $\rG$ defining $\rQ$, i.e., it must be a
  neighbour of $\rP$. Indeed, if this were not the case, then $s_{i_{k-1}}$
  and $s_{i_k}$ would commute and $s_{i_1} \dots s_{i_{k-2}}, s_{i_k}, s_{i_{k-1}}$
  would be a reduced expression for $w$, which is a contradiction with the
  inclusion $w \in \rW^\rP$.

  \smallskip

  (2) For $w = 1$ the statement is clear, as the map $q$ is a $\bP^1$-bundle.
  For $w \neq 1$ we argue as follows. At the level of Schubert varieties
  by Lemma~\ref{lemma:basics-schubert-varieties-2} we have $p^{-1}(X^w) = Z^w$
  and $qp^{-1}(X^w) = \F^{w^\rQ}$, and it is enough to prove that the induced
  map $q \colon Z^w \to \F^{w^\rQ}$ is birational. Since we have $w = w^\rQ s_\rP$
  and $w_{0, \rQ}^\rR = s_\rP$, we get the birationality by applying
  Lemma~\ref{lemma:basics-schubert-varieties-2}(2).

  \smallskip

  (3) This is an immediate consequence of \eqref{eq:corollary-quantum-to-classical-1}
  and \eqref{eq:pull-push-universal-family}.

  \smallskip

  (4) This is a consequence of \eqref{eq:corollary-quantum-to-classical-2}
  and \eqref{eq:pull-push-universal-family}.
  Indeed, for $w \in \rW^\rP \setminus \{ 1 \}$ using \eqref{eq:pull-push-universal-family},
  the commutativity of the diagram \eqref{diagram:identification-long-root-universal-family},
  and Lemma~\ref{lemma:basics-schubert-varieties-2} we get
  \begin{equation*}
    i^*q_*p^*(\sigma_X^w)  =
    \begin{cases}
      \sigma_{\Fpt}^{w s_\rP} & \text{if} \quad w s_\rP \leq w_{0,\rP}^{\rR}, \\
      0 & \text{otherwise},
    \end{cases}
  \end{equation*}
  where we identify $\Fpt$ with $\rP/\rR$ as agreed upon earlier. Now we just
  apply \eqref{eq:corollary-quantum-to-classical-2}.
\end{proof}

\section{Coadjoint varieties}
\label{section:coadjoint-varieties}

An \textsf{adjoint} (resp. \textsf{coadjoint}) variety of a simple
algebraic group $\rG$ is the highest weight vector orbit in the projectivization of
the irreducible $\rG$-representation, whose highest weight is the highest \emph{long}
(resp. \emph{short}) root of $\rG$. Clearly, if the group $\rG$ is simply laced, then
adjoint and coadjoint varieties coincide. We refer to the Table \ref{table:adjoint-and-coadjoint-varieties} in the Introduction
for an explicit list of (co)adjoint varieties.

\smallskip

In this paper we are interested in quantum cohomology of coadjoint varieties
and, therefore, in this section we discuss several results on Gromov--Witten theory and (quantum)
Schubert calculus for this class of varieties.

\smallskip

In Section \ref{subsection:gw-invariants-for-coadjoint-varieties-of-non-simply-laced-groups}
we discuss computations of degree one Gromov--Witten invariants for coadjoint
varieties of non-simply laced groups.
Recall that using Lemma \ref{lemma:quantum-to-schubert-calculus-long-root}
for a coadjoint variety $X$ of a simply laced group we can explicitly reduce
computations of degree one Gromov--Witten invariants to Schubert calculus on
the Fano variety of lines $\F(X)$ and on the variety of lines through a point $\Fpt(X)$.
Here we explain a construction that allows to reduce computations
of degree one Gromov--Witten invariants for coadjoint varieties of non-simply laced groups
to Schubert calculus on some auxiliary varieties.

\smallskip

In Section \ref{subsection:quantum-schubert-calculus-for-coadjoint-varieties} we recall
some specifics of the (quantum) Schubert calculus for coadjoint varieties based on \cite{ChPe}.

\subsection{GW invariants for coadjoint varieties of non-simply laced groups}
\label{subsection:gw-invariants-for-coadjoint-varieties-of-non-simply-laced-groups}

Results of this section rely on the observation that any coadoint variety
$\rG/\rP$ of a non-simply laced group $\rG$ is a hyperplane section of an auxiliary
homogenous variety $\hrG/\hrP$ for a simply laced group $\hrG$.
Here is the explicit list of such pairs:

\medskip

\begin{table}[h!]
\centering
\begin{tabular}{ccccc}
  \hline
  $\rG/\rP$ &  \quad & $\hrG/\hrP$ &  \quad\\
  \hline
  $\rB_n/\rP_1 = \Q_{2n-1}$ & \quad & $\rD_{n+1}/\rP_1 = \Q_{2n}$ \\
  $\rC_n/\rP_2 = \IG(2,2n)$ & \quad & $\rA_{2n-1}/\rP_2 = \G(2,2n)$ \\
  $\rF_4/\rP_4$ & \quad & $\rE_6/\rP_1$ \\
  $\rG_2/\rP_1 = \Q_5$ & \quad & $\rD_4/\rP_1 = \Q_6$\\
  \hline
\end{tabular}
\smallskip
\caption{Hyperplane sections}
\label{table:hyperplane-sections}
\end{table}

The groups $\rG$ and $\hrG$ are related by the operation called \textsf{folding}.
Let us make this setup more precise.

Let $\hrG$ be a connected simply connected simple algebraic group whose Dynkin
diagram is as in the second column of the above table, i.e., it is either $\rD_{n+1}$
with $n \geq 4$, or $\rA_{2n-1}$ with $n \geq 2$, or $\rE_6$, or $\rD_4$.
We also fix a maximal torus $\hrT$, Borel subgroups $\hrB, \, \hrB^-$, and
a maximal parabolic subgroup $\hrP$ as prescribed by the second column of the above
table, and adhereing to our conventions.

Let $\sigma$ be the automorphism of $\hrG$ corresponding to the symmetry of the
Dynkin diagram of $\hrG$, which we also denote by $\sigma$.
The automorphism $\sigma$ is an involution except in type $\rD_4$, where it is of order $3$.
Define a closed subgroup $\rG \subset \hrG$ as the fixed locus of $\sigma$
\begin{equation*}
  \rG \coloneqq \hrG^\sigma \subset \hrG
\end{equation*}
and all the standard subgroups of $\rG$ as
\begin{equation*}
  \rT = \rG \cap \hrT, \quad \rB = \rG \cap \hrB, \quad \rB^- = \rG \cap \hrB^-, \quad \rP = \rG \cap \hrP.
\end{equation*}
The embedding of maximal tori gives rise to a surjective map of root systems
\begin{equation}\label{eq:folding-map-quotient}
  \pi \colon \Phi_{\hrG} \to \Phi_{\rG},
\end{equation}
so that $\Phi_{\rG}$ is the quotient of $\Phi_{\hrG}$ with respect to $\sigma$.
The map \eqref{eq:folding-map-quotient} is defined by its values on the simple roots
and can be conveniently encoded by \textsf{folding} of Dynkin diagrams.
For example, in the third case of Table \ref{table:hyperplane-sections} we have
\begin{equation}
\label{diagram:folding-E6-F4}
  \dynkin[edge length = 2em,labels*={1,...,6},involutions={16;35}]E6
  \quad = \quad
  \dynkin[edge length = 2em,labels*={1,...,4}]F4.
\end{equation}
We refer to \cite[Section 3.6]{FoGo} for more details on folding.

From the above definitions it follows that there is a closed embedding
\begin{equation}\label{eq:hyperplane-section-embedding}
  j \colon \rG/\rP \hookrightarrow \hrG / \hrP.
\end{equation}
Moreover, the following lemma shows that \eqref{eq:hyperplane-section-embedding}
is a hyperplane section. Let us consider the natural closed embedding
\begin{equation}\label{eq:projective-embedding-ghat-phat}
  \hrG / \hrP \hookrightarrow \bP \left( V^\vee \right),
\end{equation}
where
\begin{equation*}
  V \coloneqq H^0 \Big( \hrG / \hrP, \cO_{\hrG / \hrP}(1) \Big) = V_{\hrG}^{\homega_{\hrP}},
\end{equation*}
$\homega_{\hrP}$ is the fundamental weight of $\hrG$ that corresponds to $\hrP$,
and $V_{\hrG}^{\homega_{\hrP}}$ is the irreducible representation of $\hrG$ with
highest weight $\homega_{\hrP}$.

\begin{lemma}\label{lemma:hyperplane-section}
  Up to a scalar multiple there exists a unique non-zero section
  \begin{equation}\label{eq:section-defining-hyperplane}
    s \in V
  \end{equation}
  such that
  \begin{equation*}
    \rG/\rP = Z(s) \subset \hrG/\hrP.
  \end{equation*}
\end{lemma}

\begin{proof}
  We begin with a preparation.
  A direct computation shows that under the embedding of maximal tori $\rT \subset \hrT$
  the fundamental weight $\homega_{\hrP}$ maps to $\omega_{\rP}$ (use folding diagrams).
  Moreover, since in all cases we have
  \begin{equation*}
    \dim V_{\hrG}^{\homega_{\hrP}} = \dim V_{\rG}^{\omega_{\rP}} + 1,
  \end{equation*}
  we obtain a direct sum decomposition
  \begin{equation}\label{eq:direct-sum-hyperplane-proof}
    \left( V_{\hrG}^{\homega_{\hrP}} \right)_{| \rG} = V^{\omega_{\rP}}_{\rG} \oplus \bC.
  \end{equation}
  Note that the highest weight vector of $V_{\hrG}^{\homega_{\hrP}}$
  is also the highest weight vector of the $\rG$-representation~$V^{\omega_{\rP}}_{\rG}$.
  If we dualize \eqref{eq:direct-sum-hyperplane-proof}, we also see that the highest weight
  vector in $\left( V_{\hrG}^{\homega_{\hrP}} \right)^\vee$ is also the highest weight
  vector of the $\rG$-representation~$\left( V^{\omega_{\rP}}_{\rG} \right)^\vee$.

  Now we are ready to finish the proof.
  The embedding \eqref{eq:projective-embedding-ghat-phat} realizes $\hrG / \hrP$
  as the $\hrG$-orbit of the highest weight vector $v \in \left( V_{\hrG}^{\homega_{\hrP}} \right)^\vee$.
  Consider the hyperplane $H = \bP \left(\left( V^{\omega_{\rP}}_{\rG} \right)^\vee \right) \subset \bP(V^\vee)$.
  Clearly, we have the inclusion
  \begin{equation}\label{eq:orbit-inclusion-hyperplane-proof}
     \rG / \rP = \rG \, v \subset H \underset{\bP(V^\vee)}{\times} \hrG / \hrP
  \end{equation}
  Moreover, since $\cO(H)_{\mid \hrG/\hrP}$ is a generator of the Picard group of $\hrG/\hrP$,
  it follows that the intersection \eqref{eq:orbit-inclusion-hyperplane-proof} is reduced
  and irreducible. Therefore, since $\dim \hrG / \hrP = \dim \rG / \rP + 1$, the
  inclusion \eqref{eq:orbit-inclusion-hyperplane-proof} is an equality.
\end{proof}

Now we are ready to discuss the computation of degree one GW invariants for $\rG/\rP$.
Let us collect Fano varieties of lines and varieties of lines through a point on
$\rG/\rP$ and $\hrG/\hrP$ into the commutative diagram
\begin{equation}\label{diagram:hyperplane-section-fano-variety-of-lines}
  {
  \xymatrixcolsep{4pc}
  \xymatrix{
                & \F(\rG/\rP) \\
  \Fpt(\rG/\rP) \ar[r]^{\varphi_{\rG/\rP}} \ar[d] \ar@/^1.5pc/[ru]^{i_{\rG/\rP}} & \Z(\rG/\rP) \ar[r]^{p_{\rG/\rP}} \ar[u]^{q_{\rG/\rP}} \ar[d] & \rG/\rP \ar[d]^j \\
  \Fpt(\hrG/\hrP) \ar[r]^{\varphi} \ar@/_1.5pc/[rd]^i & \Z(\hrG/\hrP) \ar[r]^p \ar[d]_q & \hrG/\hrP \\
                  & \F(\hrG/\hrP)
  }
  }
\end{equation}
Recall that by Theorem \ref{theorem:landsberg-manivel} we have
\begin{equation*}
  \F(\hrG/\hrP) = \hrG/\hrQ, \quad \Z(\hrG/\hrP) = \hrG/\hrR, \quad \Fpt(\hrG/\hrP) = \hrP/\hrR.
\end{equation*}
The parabolic $\hrQ$ is not necessarily maximal, i.e., it may be given by more than
one vertex of the Dynkin diagram of $\hrG$. In the next formula we denote the set
of these vertices by the same symbol $\hrQ$. For each vertex $\hat{q} \in \hrQ$
we denote by $\homega_{\hat{q}}$ the corresponding fundamental weight of $\hrG$.
With this notation in mind we define the $\hrG$-dominant weight
\begin{equation}\label{eq:weight-omega-Q-hat}
  \homega_{\hrQ} = \sum_{ \hat{q} \in \hrQ } \homega_{\hat{q}},
\end{equation}
and denote by $\cO(\homega_{\hrQ})$ the globally generated line bundle defined by
the weight \eqref{eq:weight-omega-Q-hat}.

We have the following lemma.
\begin{lemma}\label{lemma:hyperplane-section-fano-variety-of-lines}
  \
  \begin{enumerate}
    \item There is a canonical isomorphism
    \begin{equation}\label{eq:vector-bundle-fano-variety-of-lines}
      q_* p^* \left( \cO_{\hrG/\hrP}(1) \right)
      \cong \cU^{\homega_{\hrP}},
    \end{equation}
    where $\cU^{\homega_{\hrP}}$ is the globally generated $\hrG$-equivariant vector
    bundle of rank $2$ on $\F(\hrG/\hrP)$ given by the fundamental weight $\homega_{\hrP}$.

    \smallskip

    \item
    The global section $s$ from \eqref{eq:section-defining-hyperplane} defines a global
    section
    \begin{equation*}
      \tilde{s} \in H^0 \left( \F(\hrG/\hrP), \cU^{\homega_{\hrP}} \right)
    \end{equation*}
    and we have the identification
    \begin{equation*}
      \F(\rG/\rP) = Z(\tilde{s}) \subset \F(\hrG/\hrP).
    \end{equation*}

    \smallskip

    \item
    The global section $s$ from \eqref{eq:section-defining-hyperplane} defines a global
    section
    \begin{equation*}
      \tilde{\tilde{s}} \in H^0 \left( \Fpt(\hrG/\hrP), \cO ( \homega_{\hrQ} ) \right)
    \end{equation*}
    and we have the identification
    \begin{equation*}
      \Fpt(\rG/\rP) = Z(\tilde{\tilde{s}}) \subset \Fpt(\hrG/\hrP).
    \end{equation*}
  \end{enumerate}
\end{lemma}

\medskip

\begin{proof}
  (1) The isomorphism \eqref{eq:vector-bundle-fano-variety-of-lines} is a direct
  consequence of basic facts of homogeneous bundles on homogeneous varieties.

  \smallskip

  (2) By Lemma \ref{lemma:hyperplane-section} we know that $\rG/\rP$ is the zero
  locus $Z(s)$ of a global section $s$ of the line bundle $\cO_{\hrG/\hrP}(1)$.
  Hence, the claim is a standard fact about Fano varieties of lines
  (for example, see~\cite{AltmanKleiman}).

  \smallskip

  (3)
  The left small square in \eqref{diagram:hyperplane-section-fano-variety-of-lines}
  is Cartesian.
  Therefore, $\Fpt(\rG/\rP)$ inside $\Fpt(\hrG/\hrP)$ is identified with the zero
  locus of the global section of $\varphi^*q^* \left( \cU^{\homega_{\hrP}}  \right)$
  induced by the global section $\tilde{s}$.
  Let us do a more detailed analysis.

  The pullback $q^* \left( \cU^{\homega_{\hrP}}  \right)$ is a rank two
  $\hrG$-equivaraint vector bundle, whose irreducible factors are easily seen to be
  $\cO(\homega_{\hrP})$ and $\cO(\homega_{\hrP} - \halpha_{\hrP})$. More precisely,
  there is a $\hrG$-equivaraint short exact sequence
  \begin{equation}\label{eq:seq-fano-variety}
    0 \to \cO(\homega_{\hrP} - \halpha_{\hrP}) \to q^* \left( \cU^{\homega_{\hrP}}  \right) \to \cO(\homega_{\hrP}) \to 0
  \end{equation}
  Since we have the equality of weights
  \begin{equation*}
    \halpha_{\hrP} = 2 \homega_{\hrP} - \sum_{\hat{q} \in \hrQ} \homega_{\hat{q}},
  \end{equation*}
  the restriction of \eqref{eq:seq-fano-variety} to $\Fpt(\hrG/\hrP) = \hrR/\hrP$ becomes
  \begin{equation*}
    0 \to \cO ( \homega_{\hrQ} ) \to
    \varphi^* q^*  \left( \cU^{\homega_{\hrP}}  \right) \to
    \cO \to 0,
  \end{equation*}
  which is a split short exact sequence, as $\cO ( \homega_{\hrQ} )$ has vanishing
  higher cohomology groups.

  The global section $\tilde{s}$ gives rise to a non-zero global section of
  $\varphi^* q^*  \left( \cU^{\homega_{\hrP}}  \right)$.
  Its component along the summand $\cO$ must vanish, and, therefore, we get an
  induced global section
  \begin{equation*}
    \tilde{\tilde{s}} \in H^0 \left( \Fpt(\hrG/\hrP), \cO ( \homega_{\hrQ} ) \right)
  \end{equation*}
  as claimed. Its zero locus is then automatically $\Fpt(\rG/\rP)$.
\end{proof}

Since $\F(\hrG/\hrP)$ and $\Fpt(\hrG/\hrP)$ are homogeneous, using the above lemma
we can reduce the computation of certain GW invariants
of $\rG/\rP$ to Schubert calculus on $\F(\hrG/\hrP)$ and $\Fpt(\hrG/\hrP)$.

\begin{corollary}
  \label{corollary:quantum-to-schubert-calculus-non-simply-laced}
  For $\hat{w}_1, \dots, \hat{w}_n \in \rW_{\hrG}^{\hrP} \setminus \{ 1 \}$ we have
  \begin{equation*}
    \scal{j^*\sigma^{\hat{w}_1}_{\hrG/\hrP}, \dots , j^*\sigma^{\hat{w}_n}_{\hrG/\hrP}}_1^{\rG/\rP} =
    \deg_{\F(\hrG/\hrP)} \left( c_2(\cU^{\homega_{\hrP}}) \cup \left( \bigcup_{i = 1}^n \sigma^{\hat{w}_i s_{\hrP}}_{\F(\hrG/\hrP)} \right) \right),
  \end{equation*}
  where $\cU^{\homega_{\hrP}}$ is defined in \eqref{eq:vector-bundle-fano-variety-of-lines},
  and
  \begin{multline}
    \scal{\ptclass, j^*\sigma^{\hat{w}_1}_{\hrG/\hrP}, \dots ,
    j^*\sigma^{\hat{w}_n}_{\hrG/\hrP}}_1^{\rG/\rP} = \\
    =
    \begin{cases}
      \deg_{\Fpt(\hrG/\hrP)} \left( h \cup
      \left( \bigcup_{i = 1}^n \sigma^{\hat{w}_i s_{\hrP}}_{\Fpt(\hrG/\hrP)} \right) \right)
      & \text{if} \quad \hat{w}_i s_{\hrP} \leq w_{0,\hrP}^{\hrR} \text{  for all  } i, \\
      0
      & \text{otherwise}.
    \end{cases}
  \end{multline}
  where
  \begin{equation*}
    h = c_1 ( \cO ( \homega_{\hrQ} ) ).
  \end{equation*}
\end{corollary}

\begin{proof}
  Use Diagram \eqref{diagram:hyperplane-section-fano-variety-of-lines},
  projection formula for~$j$,
  and Lemmas \ref{lemma:hyperplane-section-fano-variety-of-lines},
  \ref{lemma:quantum-to-classical-new}, \ref{lemma:quantum-to-schubert-calculus-long-root}.
\end{proof}

In the rest of this section we discuss how to express the pullbacks $j^*\sigma^{\hat{w}}_{\hrG/\hrP}$
appearing in Corollary \ref{corollary:quantum-to-schubert-calculus-non-simply-laced}
in terms of Schubert classes on $\rG / \rP$.
As usual we denote by $W_{\hrG}$ and $W_\rG$ the Weyl groups of the pairs $(\hrG, \hrT)$
and $(\rG, \rT)$.

\begin{lemma}\label{lemma:weyl-groups-folding}
  \
  \begin{enumerate}
    \item The embedding $\rG \subset \hrG$ induces an embedding
    of the Weyl groups
    \begin{equation}\label{eq:weyl-groups-embedding}
      \iota \colon \rW_{\rG} \hookrightarrow \rW_{\hrG},
    \end{equation}
    such that
    \begin{equation}\label{eq:weyl-groups-fixed-locus}
      \rW_{\rG} = \left( \rW_{\hrG} \right)^\sigma,
    \end{equation}
    where the action of $\sigma$ on $\rW_{\hrG}$ is induced from $\hrG$.

    \smallskip

    \item For simple reflections \eqref{eq:weyl-groups-embedding} takes the form
    \begin{equation*}
      s_{\alpha} \mapsto \prod_{\ha \in \pi^{-1}(\alpha)} s_{\ha},
    \end{equation*}
    where $\pi$ is the quotient map \eqref{eq:folding-map-quotient}.
    The order of the product is not important, as simple reflections $s_{\ha}$ for
    $\ha \in \pi^{-1}(\alpha)$ commute with each other.

    \smallskip

    \item Let $\rW_{\hrG}^{fc} \subset \rW_{\hrG}$ be the subset of fully commutative
    elements. There is a well-defined map
    \begin{equation*}
      \begin{aligned}
        & \iota^* \colon \rW_{\hrG}^{fc} \to \rW_\rG, \\
        & \qquad \hat{w} \, \mapsto  \, s_{\pi(\hat{\alpha}_{i_1})} \cdots s_{\pi(\hat{\alpha}_{i_r})},
      \end{aligned}
    \end{equation*}
    where $\hat{w} = s_{\hat{\alpha}_{i_1}} \cdots s_{\hat{\alpha}_{i_r}}$ is any
    reduced expression for $\hat{w}$. Moreover, since the homogeneous space $\hrG/\hrP$ is minuscule,
    we have $\rW^{\hrP}_{\hrG} \subset \rW_{\hrG}^{fc}$,
    and, therefore, $\iota^*$ is well-defined on $\rW^{\hrP}_{\hrG}$.
  \end{enumerate}
\end{lemma}

\begin{proof}
(1) This claim follows immediately from the definitions of Weyl groups and from
the fact that we have $C_{\hrG}(T_{\rG}) = T_{\hrG}$.
We leave the details to the reader.

\smallskip

(2) See \cite[Section 3.6]{FoGo}.

\smallskip

(3) Since $\hat{w}$ is fully commutative, its reduced word is unique up to
commuting relations (see \cite{Stembridge}). The claim now follows from the following observation.
If $\halpha$ and $\hbeta$ are simple roots of $\hrG$ such that
$s_{\halpha} s_{\hbeta} = s_{\hbeta}s_{\halpha}$, then we also have
$s_{\pi(\halpha)} s_{\pi(\hbeta)} = s_{\pi(\hbeta)}s_{\pi(\halpha)}$.

\end{proof}

Recall that by Lemma \ref{lemma:hyperplane-section} the variety $\rG/\rP$ is a hyperplane
section of $\hrG/\hrP$. Therefore, by the Lefschetz hyperplane section theorem, the pullback map
\begin{equation*}
  j^* \colon H^l(\hrG/\hrP,\bQ) \to H^l(\rG/\rP,\bQ)
\end{equation*}
is an isomorphism for $l \leq \dim \rG/\rP$. We want to give an explicit formula
for $j^*$ in terms of Schubert classes.
By Lemma \ref{lemma:weyl-groups-folding} we have well-defined maps
\begin{equation*}
  \begin{aligned}
    & \rW^\rP_\rG \to \rW^{\hrP}_{\hrG} \\
    & w \quad \mapsto \, w_* \coloneqq \iota(w)^{\hrP}
  \end{aligned}
  \hspace{30pt} \text{and} \hspace{30pt}
  \begin{aligned}
    & \rW^{\hrP}_{\hrG} \to \rW^\rP_\rG \\
    & \hat{w} \quad \mapsto \, \hat{w}^* \coloneqq \iota^*(\hat{w})^{\rP}
  \end{aligned}
\end{equation*}

We have the following lemma.
\begin{lemma}\label{lemma:weyl-group-push-pull}
  \
  \begin{enumerate}
    \item Consider $w \in \rW^\rP_\rG$ with $2\ell(w) \leq \dim \rG/\rP$ and a reduced expression
    \begin{equation*}
      w = s_{\alpha_{i_1}} \cdots s_{\alpha_{i_r}}.
    \end{equation*}
    For each $k \in [1,r]$, there exists a unique simple root
    $\hat{\beta}_k \in \pi^{-1}(\alpha_{i_k})$ such that we have
    a reduced expression
    \begin{equation*}
      w_* = s_{{\hat{\beta}}_1} \cdots s_{{\hat{\beta}}_r}.
    \end{equation*}
    In particular, we have $\ell(w_*) = \ell(w)$.

    \smallskip

    \item For $w \in \rW_\rG^\rP$ with $2 \ell(w) \leq \dim \rG/\rP$ we have
    \begin{equation*}
      (w_*)^* = w.
    \end{equation*}

    \smallskip

    \item For $w \in \rW_\rG^\rP$ with $2 \ell(w) \leq \dim \rG/\rP$ the closed
    embedding \eqref{eq:hyperplane-section-embedding} gives rise to isomorphisms
    of Schubert varieties
    \begin{equation*}
      j \left( (\rG/\rP)_w \right) = (\hrG/\hrP)_{w_*}.
    \end{equation*}

    \smallskip

    \item For $\hat{w} \in \rW_{\rG}^{\rP}$ with $2 \ell(\hat{w}) \leq \dim \hrG/\hrP$ we have
    \begin{equation*}
      j^* \left( \sigma^{\hat{w}}_{\hrG/\hrP} \right) = \sigma^{\hat{w}^*}_{\rG/\rP}.
    \end{equation*}
  \end{enumerate}
\end{lemma}

\begin{proof}
  (1) In the third case of Table \ref{table:hyperplane-sections}, where the order of $\sigma$ is $3$, we have
  $\dim \rG/\rP = 5$ and there are only two non-trivial coset representatives
  $\hat{w} \in \rW_{\rG}^{\rP}$ for which we need to prove the claim and for both
  of them the claim is obvious. We leave the details to the reader.

  For the rest of the proof we assume that the automorphism $\sigma$ is an involution,
  which covers all the other cases in Table \ref{table:hyperplane-sections}.

  Let us denote $X = \rG/\rP$ and $Y = \hrG/\hrP$. Unraveling the definitions
  of Schubert varieties we immediately obtain $j(\Xo_w) \subset \Yo_{w_*}$. Hence,
  we get the inequality
  \begin{equation}\label{eq:inequality-length}
    \ell(w_*) \geq \ell(w),
  \end{equation}
  which we will soon show to be an equality.

  The proof is by induction on $\ell(w)$. The base of induction is the case $\ell(w) = 1$,
  which is clear. Thus, we assume that we have $w = s_{\alpha_{i_1}} \cdots s_{\alpha_{i_r}}$
  and that we already know the statement for $v = s_{\alpha_{i_2}} \cdots s_{\alpha_{i_r}}$,
  i.e., there exist a unique $\hbeta_{i_k} \in \pi^{-1}(\alpha_{i_k})$ for all $k \in [2,r]$
  such that $v_* = s_{\hat{\beta}_{i_2}} \cdots s_{\hat{\beta}_{i_r}}$.
  Let us consider the identity
  \begin{equation*}
    w_* = (\iota(\alpha_{i_1}) v_*)^{\hrP}.
  \end{equation*}
  If $\# \pi^{-1}(\alpha_{i_1}) = 1$, then the claim is clear due to \eqref{eq:inequality-length}.
  Therefore, from now on we assume that $\pi^{-1}(\alpha_{i_1}) = \{ \hbeta, \hgamma \}$
  and we have
  \begin{equation}\label{eq:proof-formula-w_*}
    w_* = (s_{\hbeta}s_{\hgamma} v_*)^{\hrP},
  \end{equation}
  where the order of $s_{\hbeta}$ and $s_{\hgamma}$ is not important, as they commute.
  By \cite[Proposition 2.11]{ChPe} any $u \in\rW^{\rP}_{\rG}$ with $2 \ell(u) \leq \dim \rG/\rP$
  is $\omega_{\rP}$-minuscule, and, hence, we have
  \begin{equation*}
    \left(v(\omega_{\rP}), \alpha_{i_1}^\vee \right) = 1.
  \end{equation*}
  Since we have $\pi(v_*(\homega_{\hrP})) = v(\omega_{\rP})$, we obtain by
  \cite[Section 3.6]{FoGo} the equality
  \begin{equation*}
    (v_*(\homega_{\hrP}), \hbeta^\vee ) + (v_*(\homega_{\hrP}), \hgamma^\vee )= 1.
  \end{equation*}
  Since the group $\hrG$ is simply laced and the space $\hrG/\hrP$ is minuscule
  (see \cite[Definition 2.1]{ChPe}) for any $\halpha \in \Phi_{\hrG}$ we have
  \begin{equation*}
    (v_*(\homega_{\hrP}), \halpha^\vee ) \in \{-1, 0, 1 \}.
  \end{equation*}
  Therefore, without loss of generality we can assume
  \begin{equation*}
    (v_*(\homega_{\hrP}), \hbeta^\vee) = 1
    \quad \text{and} \quad
    (v_*(\homega_{\hrP}), \hgamma^\vee) = 0.
  \end{equation*}
  The latter equality implies $s_{v_*^{-1}}(\hgamma) \in \rW_{\hrP}$, and,
  since $v_*^{-1} s_{\gamma} v_* = s_{v_*^{-1}(\hgamma)}$, we obtain
  \begin{equation*}
    (s_{\hgamma} v_*)^{\hrP} = v_*.
  \end{equation*}
  Therefore, from \eqref{eq:proof-formula-w_*} and \eqref{eq:inequality-length} we get
  \begin{equation*}
    w_* = (s_{\hbeta} v_*)^{\hrP} = s_{\hbeta} v_*,
  \end{equation*}
  which finishes the proof.

  \smallskip

  (2) The claim follows immediately from Part (1) and Lemma \ref{lemma:weyl-groups-folding}(3).

  \smallskip

  (3) Let us denote $X = \rG/\rP$ and $Y = \hrG/\hrP$. As in Part (1) we have $j(\Xo_w) \subset \Yo_{w_*}$.
  Since by Part (1) we now also have $\ell(w_*) = \ell(w)$, we conclude $j(X_w) = Y_{w_*}$.

  \smallskip

  (4) The claim follows from Part (3) using Poincaré duality (see Lemma \ref{lemma:basics-schubert-varieties-1}(5)).
\end{proof}

\subsection{Quantum Schubert calculus for coadjoint varieties}
\label{subsection:quantum-schubert-calculus-for-coadjoint-varieties}

As discussed in Section \ref{subsection:schubert-classes}, Schubert classes on $\rG/\rP$
are labelled by $\rW^\rP$. If $\rG/\rP$ is coadoint, there is an alternative labelling
introduced in \cite{ChPe}.
Throughout this section we assume $X = \rG/\rP$ to be a coadjoint variety not of type $\rA$.

Let $\roots$ be the root system of $\rG$. For a root $\alpha = \sum_{\beta \in \Delta} m_\beta \beta$
we define its \textsf{height} and \textsf{support} to be respectively
\begin{equation*}
  |\alpha| = \sum_{\beta \in \Delta} |m_\beta| \qquad \text{and} \qquad
  {\rm Supp}(\alpha) = \{ \beta \in \Delta \mid m_\beta \neq 0 \}.
\end{equation*}

It is well-known that for the canonical line bundle of $X$ we have
\begin{equation*}
  \omega_{X} = \cO_{X}(-r), \quad \text{with} \quad r = |\Theta|,
\end{equation*}
where $\Theta$ is the highest root.

Let $\shortroots \subset \roots$ be the subset of short roots and $\theta \in \shortroots$
be the highest short root.
It is well-known that we have
\begin{equation*}
  \dim X = 2 r_0 - 1, \quad \text{with} \quad r_0 = |\theta|.
\end{equation*}

Finally we recall that the Weyl group $\rW$ preserves the length
of roots and acts transitively on $\shortroots$. Also recall that
for a simply laced group $\rG$ we consider all roots to be both short and long.

\begin{lemma}[{\cite[Proposition 2.9]{ChPe}}]
  \label{lemm:wp-root}
  Let $X = \rG/\rP$ be a coadjoint variety not of type $\rA$.
  \begin{enumerate}
    \item The map
    \begin{equation*}
      \begin{aligned}
        & \rW^{\rP} \to \shortroots \\
        & w \hspace{10pt} \mapsto \hspace{2pt} w(\theta)
      \end{aligned}
    \end{equation*}
    is bijective.
    Therefore, Schubert varieties and classes can be indexed by $\shortroots$ as
    \begin{equation*}
      X_{w(\theta)} \coloneqq X^w \quad \text{and} \quad \sigma_{w(\theta)} \coloneqq \sigma^w \qquad \text{for} \quad w \in \rW^\rP.
    \end{equation*}

    \medskip

    \item We have the equivalence
    \begin{equation*}
      X_\alpha\subset X_\beta \Leftrightarrow
      \begin{cases}
        \alpha \leq \beta  & \text{if } \alpha \text{ and } \beta \text{ have the same sign}, \\
        {\rm Supp}(\alpha) \cup {\rm Supp}(\beta) \text{ is connected}  & \text{for } \alpha <0 \text{ and } \beta > 0. \\
      \end{cases}
    \end{equation*}

    \medskip

    \item We have
    \begin{equation*}
      \deg(\sigma_{\alpha}) =
      \begin{cases}
        r_0 - |\alpha|      & \text{for } \alpha \in \Phi_{\rm s}^+, \\
        r_0 + |\alpha| - 1 & \text{for } \alpha \in \Phi_{\rm s}^-, \\
      \end{cases}
    \end{equation*}
    where $r_0 = |\theta|$.

    \medskip

    \item The Poincaré duality takes the form
    \begin{equation*}
      \sigma_\alpha^\vee = \sigma_{w_0(\alpha)}.
    \end{equation*}
  \end{enumerate}
\end{lemma}

\bigskip

We also recall the quantum Chevalley formula for coadjoint varieties using the labelling
of Schubert classes introduced above.

\begin{theorem}\label{theorem:quantum-chevalley}
  Let $X = \rG/\rP$ be a coadjoint variety not of type $\rA$ and let $h \in H^*(X, \bQ)$ be the hyperplane class.
  \begin{enumerate}
    \item We have
      \begin{equation*}
        \label{eq:classical-chevalley-formula}
        h \cup \sigma_\alpha =
        \left\{
        \begin{aligned}
          & \sum_{\beta \in \shortroots, \alpha - \beta \in \simpleroots} \sigma_\beta & \text{if $\alpha$ is not simple} \\
          & 2 \sigma_{-\alpha} +
          \sum_{{\beta \in \simpleroots_s \setminus \{\alpha\}}, \ \langle \alpha^\vee , \beta \rangle \neq 0} \sigma_{-\beta}
          & \text{if $\alpha$ is simple}
        \end{aligned}
        \right.
      \end{equation*}

    \bigskip

    \item If $\rG$ is simply laced, then we have
      \begin{equation*}
        h \star_{0} \sigma_\alpha =  h \cup \sigma_\alpha +
        \left\{
        \begin{aligned}
          & q & \textrm{if $\alpha$ is simple with $\langle \alpha^\vee , \Theta \rangle \neq 0$}, \\
          & q \sigma_{\Theta - \alpha} & \textrm{if $\alpha < 0$ with $\langle \alpha^\vee , \Theta \rangle = -1$}, \\
          & 2q^2 + q \sigma_{-\alpha_\rP} & \textrm{if $\alpha = - \Theta$},  \\
          & 0 & \textrm{otherwise}.
        \end{aligned}
        \right.
      \end{equation*}

      \bigskip

      \item If $\rG$ is non-simply laced, then we have
      \begin{equation*}
        h \star_{0} \sigma_\alpha =  h \cup \sigma_\alpha +
        \left\{
        \begin{aligned}
          & q \sigma_{\alpha + \Theta} & \textrm{if $\alpha \leq - \delta$}, \\
          & 0 & \textrm{otherwise},
        \end{aligned}
        \right.
      \end{equation*}
      where $\delta = \Theta - \theta$.
  \end{enumerate}
\end{theorem}

\begin{remark}
Analogues of the above results also exist in type $\rA$. However, as in type $\rA$
the coadjoint variety is of Picard rank two, one needs to be more careful with the statement.
We refer to \cite{ChPe} for details.
\end{remark}

\section{Type $\rD_n$}
\label{section:type-D}

Let $\rG$ be a semisimple group of type $\rD_n$ and of adjoint type. We assume $n \geq 4$. Since $\rG$ is simply laced, the adjoint and the coadoint varieties agree. Let $X$ be this (co)adjoint variety, we have $X = \rG/\rP_2$ where $\rP_2$ is the maximal parabolic subgroup of $\rG$ associated to the simple root $\alpha_2$ with notation as in \cite{Bo}. This variety can be described as isotropic Grassmannian of lines as follows.

Let $V$ be a $2n$-dimensional complex vector space endowed with a non degenerate quadratic form $q$. Define the algebraic variety $\OG(2, V)$ parametrizing $2$-dimensional isotropic subspaces of $V$. As for different non degenerate quadratic forms on $V$ we obtain isomorphic isotropic Grassmannians, it is unambiguous to write $\OG(2,2n)$. We have $X \simeq \OG(2,2n)$ and $\rG = \aut(X)$. We have $\dim(X) = 4n - 7 = 2r - 1$ with $r = {\rm Index}(X) = 2n - 3$.

\subsection{Borel presentation for $\OG(2,2n)$}

We will give an explicit description of Borel's presentation (see Subsection \ref{subsection:borel-presentation}). Consider the short exact sequence of vector bundles on $X=\OG(2, V)$
\begin{align}\label{Eq.: Tautological S.E.S I}
0 \to \cU \to \cV \to \cV/\cU \to 0,
\end{align}
where $\cV$ is the trivial vector bundle with fiber $V$, where $\cU$ is the subbundle of isotropic subspaces, and where $\cQ := \cV/\cU$ is the quotient bundle. Usually one refers to $\cU$ and $\cQ$ as \textit{tautological subbundle} and \textit{tautological quotient bundle} respectively.

One also defines a vector bundle $\cU^\perp$ as the kernel of the composition $\cV \overset{q}{\to} \cV^\vee \to \cU^\vee$, where the first morphism is the isomorphism induced by the quadratic form, and the second one is the dual of the natural inclusion $\cU \to \cV$. From the definition of $\cU^\perp$ we immediately obtain an isomorphism
\begin{align}\label{Eq.: V/U^perp iso U^*}
\cV/\cU^\perp \simeq \cU^\vee.
\end{align}

Further, we have inclusions of vector bundles $\cU \subset \cU^\perp \subset \cV$, and can also consider the short exact sequence
\begin{align}\label{Eq.: Tautological S.E.S II}
0 \to \cU^{\perp}/\cU \to Q = \cV/\cU \to \cV/\cU^\perp \to 0.
\end{align}

The vector bundle $\cQ = \cV/\cU$ is of rank $2n-2$ and and we denote its Chern classes by $\psi_k = c_k(\cQ)$. The vector bundle $\cV/\cU^\perp \simeq \cU^\vee$ is of rank $2$ and we denote its Chern classes by $h = c_1(\cU^\vee)$ and $p = c_2(\cU^\vee)$. The vector bundle $\cU^\perp/\cU$ is of rank $2n-4$. Furthermore it is self dual (via the restriction of $q$), therefore it only has non vanishing even Chern classes. We set $\Sigma_k = c_{2k}(\cU^\perp/\cU)$.

From \eqref{Eq.: Tautological S.E.S II} and the multiplicativity of Chern polynomial on short exact sequences we obtain the equality
\begin{equation*}
  c_t(\cQ) = c_t(\cU^\vee)c_t(\cU^\perp/\cU),
\end{equation*}
whose graded pieces imply the following
\begin{equation}\label{eq:psi-via-h-p-sigma}
  \begin{aligned}
    & \psi_{2j} = \Sigma_j + p \Sigma_{j-1} \quad \text{for} \quad j \in [1,n-1], \\
    & \psi_{2j+1} = h\Sigma_j \quad \text{for} \quad j \in [0,n-2].
  \end{aligned}
\end{equation}

Let $\rT \subset\rB \subset \rP$ be a maximal torus and a Borel subgroup contained in $\rP$. The natural projection $\rG/\rT \to \rG/\rB$ induces an isomorphism $H^*(\rG/\rB,\QQ) \simeq H^*(\rG/\rT,\QQ)$. Furthermore, the projection $\pi \colon \rG/\rT \to \rG/\rP$ induces the embedding of cohomology rings $\pi^* \colon H^*(\rG/\rP,\QQ) \to H^*(\rG/\rT,\QQ)$. 
Given a vector bundle $\cE$ on $\rG/\rP$ its pull-back $\pi^*\cE$ splits into a direct sum line bundles $L_1, \dots, L_n$. Hence, in $H^*(\rG/\rT,\QQ)$ we have the equality
\begin{equation*}
  \pi^* c_t(\cE) = c_t(L_1) \dots c_t(L_n).
\end{equation*}
First Chern classes of the line bundles $L_i$ are called \textsf{Chern roots of $\cE$} and are given by images of characters of $\rT$ (or equivalently of $\rB$) via the characteristic map (see Subsection \ref{subsection:borel-presentation}).

Let $x_1$ and $x_2$ be the Chern roots of $\cU^\vee$, so that we have
\begin{equation*}
  c_t(\cU^\vee) = (1 + x_1 t)(1 + x_2t).
\end{equation*}
Further, due to the self-duality of $\cU^\perp/\cU$, its possible to choose Chern roots $x_3, \dots, x_n$ of $\cU^\perp/\cU$, so that we have
\begin{equation*}
  c_t(\cU^\perp/\cU) = (1 - x_3^2t^2) \cdots (1 - x_n^2t^2).
\end{equation*}
From the definition of classes $h, p, \Sigma_j$ it now immediately follows
\begin{equation}\label{eq:h-p-Sigma-via-chern-roots}
  \begin{aligned}
    & h = \gs_1(x_1,x_2) = x_1 + x_2, \\
    & p = \gs_2(x_1,x_2) = x_1x_2, \\
    & \Sigma_j = (-1)^j\gs_j(x_3^2,\cdots,x_n^2) \quad \text{for} \quad j \in [1,n-2],
  \end{aligned}
\end{equation}
where $\gs_j$ is the $j$-symmetric function. Note an easy computation of the weights of the line bundles occuring in the decomposition of $\pi^*\cU^\vee$ or $\pi^*(\cU^\perp/\cU)$ imply that that the Chern roots $x_1,x_2,x_3,\cdots,x_n$ are the images under the characteristic map of the standard basis $\varepsilon_1,\varepsilon_2,\varepsilon_3,\cdots,\varepsilon_n$ of $\rG$ (see \cite[Planche IV, page 256]{Bo}).

Recall Borel's presentation (see Subsection \ref{subsection:borel-presentation}). Let $\Lambda = \rX(T)$ be the group of characters of $\rT$. Then we have an isomorphism $H^*(X,\QQ) \simeq \QQ[\Lambda]^{\rW_{\rP}}/\QQ[\Lambda]^{\rW}_+$. An explicit description of the invariants $\QQ[\Lambda]^{\rW_{\rP}}$ appearing in Borel's presentation 
can be found, for example, in \cite[p.~210]{Bo}. According to this description the $\QQ$-algebra $\QQ[\Lambda]^{\rW_{\rP}}$ is freely generated by the polynomials
\begin{equation*}
  \begin{aligned}
    & x_1 + x_2 \\
    & x_1 x_2 \\
    & \gs_j(x_3^2, \cdots ,x_n^2) \quad \text{for} \quad j \in [1,n-3] \\
    & x_3 \dots x_n.
  \end{aligned}
\end{equation*}
Similarly, the ideal $\QQ[\Lambda]^{\rW}_+$ is generated by the polynomials
\begin{equation*}
  \begin{aligned}
    & \gs_j(x_1^2, \cdots ,x_n^2) \quad \text{for} \quad j \in [1,n-1] \\
    & x_1 \dots x_n.
  \end{aligned}
\end{equation*}

Hence, we obtain the following explicit description of Borel's presentation.

\begin{proposition}
  \label{cor-pres-xi}
  We have an algebra isomorphism
  \begin{equation*}
    H^*(X,\QQ) = \QQ[h,p,\gamma,(\Sigma_j)_{j \in [1,n-3]}]/((\Xi_i)_{i \in [1,n-1]} , \xi_n),
  \end{equation*}
  where
  \begin{equation}\label{eq:borel-presentation-II}
    \begin{aligned}
      & \gamma = \gs_{n-2}(x_3,\cdots,x_n) = x_3 \cdots x_n, \\
      & \Xi_j = \gs_j(x_1^2,\cdots,x_n^2) \quad \text{for} \quad j \in [1,n], \\
      & \xi_n = \gs_n(x_1,\cdots,x_n) = x_1 \cdots x_n,
    \end{aligned}
  \end{equation}
  and $h,p, \Sigma_j$ are defined by \eqref{eq:h-p-Sigma-via-chern-roots}.
\end{proposition}

Our next goal is to simplify the above presentation. In particular, we are going to eliminate the variables $\Sigma_j$ with $j \in [1,n-3]$ using the equations $\Xi_i$ with $i \in [1,n-3]$. For that we introduce the following two ingredients:
\begin{enumerate}
  \item For any $l \in [1,n]$ we have the following identity
  \begin{equation}\label{eq:identity-Xi-Sigma}
    \Xi_{l} = (-1)^{l}\Sigma_{l} + (x_1^2 + x_2^2)(-1)^{l-1}\Sigma_{l-1} + x_1^2x_2^2(-1)^{l-2} \Sigma_{l-2}
  \end{equation}
  of symmetric polynomials in $x_1, \dots, x_n$, where $\Xi_{l}$ and $\Sigma_{l}$ are defined in \eqref{eq:h-p-Sigma-via-chern-roots} and \eqref{eq:borel-presentation-II} respectively, and we use the convention $\Sigma_0 = 1$ and $\Sigma_{-1} = 0$.

  \item For any $j \geq 0$ we define polynomials $E_j(h,p) \in \bQ[h,p]$ as
  \begin{equation}\label{eq:definition-polynomials-Ej}
    E_j(h,p) =
    \begin{cases}
      E_0 = 1, \\
      E_1 = h^2 - 2p, \\
      E_j = (h^2 - 2p)E_{j-1} - p^2 E_{j-2} \quad \text{for} \quad j \geq 2.
    \end{cases}
  \end{equation}
  It is clear from the definition that, if we set $\deg(h) = 1$ and $\deg(p) = 2$, then the polynomial $E_j(h,p)$ is graded homogeneous of degree $2j$.
\end{enumerate}

\begin{lemma}
  \label{lemm-elim}
  For $j \in [1,n-2]$ we have the following equality in $H^*(X,\QQ)$
  \begin{equation}\label{eq:Sigma-via-x1-and-x2}
    \Sigma_j = \sum_{k = 0}^j x_1^{2k} x_2^{2(j-k)}.
  \end{equation}
  In particular, we have
  \begin{equation}\label{eq:Sigma-via-h-and-p}
    \Sigma_j = E_j(h,p).
  \end{equation}
\end{lemma}

\begin{proof}
  To show \eqref{eq:Sigma-via-x1-and-x2} we use \eqref{eq:identity-Xi-Sigma} and induction on $j$.
  For $j = 1$ the equality~\eqref{eq:identity-Xi-Sigma} becomes $\Xi_1 = -\Sigma_1 + x_1^2 + x_2^2$. Therefore, in $H^*(X,\QQ)$ we get $\Sigma_1 = x_1^2 + x_2^2$, which is a particular instance of \eqref{eq:Sigma-via-x1-and-x2}.

  Let us now assume that \eqref{eq:Sigma-via-x1-and-x2} is known to hold up to $j = r$. Using \eqref{eq:identity-Xi-Sigma}  we see that in $H^*(X,\QQ)$ we have the relation $\Sigma_{r+1} = (x_1^2 + x_2^2)\Sigma_r - x_1^2x_2^2 \Sigma_{r-1}$ and the claim now follows by induction.

  To show \eqref{eq:Sigma-via-h-and-p} use \eqref{eq:identity-Xi-Sigma}, \eqref{eq:Sigma-via-x1-and-x2}, and
  \eqref{eq:h-p-Sigma-via-chern-roots}.
\end{proof}

\begin{corollary}\label{corollary:borel-presentation-III-simplified}
  There is an isomorphism of algebras
  \begin{equation*}
    H^*(X,\QQ) \simeq \QQ[h,p,\gamma]/(EQ_{n}, EQ_{2n-4}, EQ_{2n-2}),
  \end{equation*}
  where
  \begin{equation}\label{eq:borel-presentation-III-simplified}
    \begin{aligned}
      & EQ_{n} = p\gamma , \\
      & EQ_{2n-4} = \gamma^2 + (-1)^{n-1} E_{n-2}(h,p) , \\
      & EQ_{2n-2} = (h^2-p) E_{n-2}(h,p) - p^2 E_{n-3}(h,p) .
    \end{aligned}
  \end{equation}
\end{corollary}

\begin{proof}
  Using \eqref{eq:identity-Xi-Sigma} we can eliminate variables $\Sigma_j$ for $j$ in $[1,n-3]$ using relations $\Xi_i$ for  $i \in [1,n-3]$. Thus, after the elimination we have
  \begin{equation*}
    H^*(X,\QQ) \simeq \QQ[h,p,\gamma]/(\xi_n, \Xi_{n-2}, \Xi_{n-1}).
  \end{equation*}
  Let us look at these relations one at a time:
  \begin{enumerate}
    \item Since $\xi_n = p \gamma$, we immediately get the first relation in \eqref{eq:borel-presentation-III-simplified}.

    \item After the elimination of the variables $\Sigma_j$ for $j$ in $[1,n-3]$ we need to replace their occurences in in the relations $\Xi_{n-2}$ and $\Xi_{n-1}$ by the polynomial $E_j(h,p)$. Thus, we can write $\Xi_{n-2}$ in the variables $h,p,\gamma$ as
    \begin{equation*}
      \Xi_{n-2} = \gamma^2 + (-1)^{n-3}\left((h^2 - 2p)E_{n-3}(h,p) - p^2 E_{n-4}(h,p) \right).
    \end{equation*}
    Finally, using \eqref{eq:definition-polynomials-Ej} we can rewrite the above relation as
    \begin{equation*}
      \Xi_{n-2} = \gamma^2 + (-1)^{n-3}E_{n-2}(h,p).
    \end{equation*}

    \item Again, we need to rewrite $\Xi_{n-1}$ in terms of $h,p,\gamma$. By \eqref{eq:identity-Xi-Sigma} we have
    \begin{equation*}
      \Xi_{n-1} = (x_1^2 + x_2^2)(-1)^{n-2}\Sigma_{n-2} + x_1^2x_2^2(-1)^{n-3} \Sigma_{n-3}.
    \end{equation*}
    As in the previous case we can rewrite
    \begin{equation*}
      \Xi_{n-1} = (h^2 - 2p)\gamma^2 + p^2(-1)^{n-3}E_{n-3}(h,p).
    \end{equation*}
    Now using the second relation of \eqref{eq:borel-presentation-III-simplified}, which we have already showed in the previous step, we can replace $\gamma^2$ and get
    \begin{equation*}
      (-1)^{n-2}\left((h^2 - 2p) E_{n-2}(h,p) - p^2 E_{n-3}(h,p)\right)
    \end{equation*}
    Finally, since by the first and the second relations we have $p E_{n-2}(h,p) = p \gamma = 0$, we can rewrite the above relation as
    \begin{equation*}
      (-1)^{n-2}\left((h^2 - p)E_{n-2}(h,p) - p^2E_{n-3}(h,p)\right),
    \end{equation*}
  \end{enumerate}
The claim now follows.
\end{proof}

\subsection{Schubert classes for $\OG(2,2n)$}

Let $e_1, \dots, e_{2n}$ be a basis of the vector space $V$ such that
\begin{equation*}
  B(e_i,e_j) = \delta_{2n+1 - i,j}
\end{equation*}
where $B$ is the non-degenerate symmetric bilinear form associated to the quadratic form $q$. For $k \in [1,2n]$, define $F_k = \scal{e_i \ | \ i \in [1,k]}$ and $F'_n = F_{n-1} + \scal{e_{n+1}}$ and set
$$\epsilon(k) = 2n - 2 - k + \left\{
\begin{array}{ll}
  1 & \textrm{ if $k \leq n - 2$} \\
  0 & \textrm{ if $k > n - 2$}. \\
\end{array}
\right.$$
For $k \in [1,2n-3]$ with $k \neq n-2$, we define Schubert varieties
\begin{equation*}
  X_k = \{ V_2 \in X \ | \ V_2 \cap F_{\epsilon(k)} \neq 0 \}.
\end{equation*}
For $k = n - 2$ we define two Schubert varieties
\begin{equation*}
  X_{n-2} = \{ V_2 \in X \ | \ V_2 \cap F_n \neq 0 \}
\end{equation*}
and
\begin{equation*}
  X'_{n-2} = \{ V_2 \in X \ | \ V_2 \cap F'_{n} \neq 0 \}.
\end{equation*}
We have $\codim(X_k) = k$ and $\codim(X'_{n-2}) = n-2$. Finally, we set
\begin{equation*}
  \tau_k = [X_k] \quad \text{for} \quad k \in [1,2n-3]
\end{equation*}
and
\begin{equation*}
  \tau'_{n-2} = [X'_{n-2}].
\end{equation*}
These classes are called \textsf{special Schubert classes}.

We have seen in Subsections \ref{subsection:schubert-classes} and \ref{subsection:quantum-schubert-calculus-for-coadjoint-varieties} that Schubert classes can be indexed by elements of the Weyl group and in the (co)adjoint case also by roots. We explicitely give both descriptions for the classes $(\tau_k)_{k \in [1,2n-3]}$ and $\tau'_{n-2}$.

In terms of elements of the Weyl group,
we can express these classes as follows. Consider the elements of $\rW^{\rP}$ defined by
\begin{equation*}
  w_j =
  \begin{cases}
    s_{\alpha_{j+1}} s_{\alpha_j} \dots s_{\alpha_2} \quad \text{for} \quad 1 \leq j \leq n-1, \\
    s_{\alpha_{2n-2 - j}} \dots s_{\alpha_{n-3}} s_{\alpha_{n-2}} s_{\alpha_n} s_{\alpha_{n-1}} \dots s_{\alpha_2} \quad \text{for} \quad n \leq j \leq 2n-3,
  \end{cases}
\end{equation*}
and
\begin{equation*}
  w'_{n-2} = s_{\alpha_n} s_{\alpha_{n-2}} s_{\alpha_{n-3}} \dots s_{\alpha_2}.
\end{equation*}
With this notation we have
\begin{equation*}
  \begin{aligned}
    & \tau_j = [X^{w_j}] = \sigma^{w_j} \quad \text{for} \quad \quad 1 \leq j \leq 2n-3, \\
    & \tau'_{n-2} = [X^{w'_{n-2}}] = \sigma^{w'_{n-2}}.
  \end{aligned}
\end{equation*}
To check that this is the correct description, since $\ell(w_k) = k$, $\ell(w'_{n-2}) = n-2$, it is enough to check that, using the action of the Weyl group $\rW$ on $V$, that we have $w_k(\scal{e_{2n},e_{2n-1}}) \in X_k$ and $w'_{n-2}(\scal{e_{2n},e_{2n-1}}) \in X'_{n-2}$.

We now describe these classes using roots. Applying Lemma \ref{lemm:wp-root}, we have
\begin{equation*}
  \begin{aligned}
    & \tau_k = \sigma_{\theta - \gamma_k} \quad \text{for} \quad k \in [1,2n-3], \\
    & \tau'_{n-2} = \sigma_{\theta - \gamma'_{n-2}},
  \end{aligned}
\end{equation*}
where
\begin{equation*}
  \begin{aligned}
    & \gamma_k =
    \begin{cases}
      \alpha_2 + \cdots + \alpha_{k+1} & \quad \text{for} \quad k \in [1,n-1], \\
      \alpha_2 + \cdots + \alpha_{n} + \alpha_{n-2} + \cdots + \alpha_{n-2-(k-n)} & \quad \text{for} \quad k \in [n,2n-4], \\
      \theta + \alpha_1  & \quad \text{for} \quad k = 2n-3, \\
    \end{cases}
    \\
    & \gamma'_{n-2} = \alpha_2 + \cdots + \alpha_{n-2} + \alpha_n.
  \end{aligned}
\end{equation*}
and the highest (short) root $\theta$ is given in this case by
\begin{equation*}
  \theta = \alpha_1 + 2\sum_{i = 2}^{n-2} \alpha_i + \alpha_{n-1} + \alpha_n.
\end{equation*}

Using the above description of Schubert classes, we reinterpret Borel's presentation given in Corollary \ref{corollary:borel-presentation-III-simplified} in terms of Schubert classes. Recall from \cite[Section 3.3]{BKT} that we have
\begin{equation}\label{eq:psi-via-tau}
  \psi_j = c_j(\cQ) =
  \begin{cases}
    \tau_j & \textrm{ for $j < n-2$} \\
    \tau_{n-2} + \tau'_{n-2} & \textrm{ for $j = n-2$} \\
    2\tau_j & \textrm{ for $j > n-2$} \\
  \end{cases}
\end{equation}
Together with \eqref{Eq.: Tautological S.E.S I} and the classical Chevalley formula \eqref{eq:classical-chevalley-formula}
this implies
\begin{equation}\label{eq:h-and-p-via-tau}
  h = \tau_1 \and p = \tau_1^2 - \psi_2 = \tau_{1,1},
\end{equation}
where $\tau_{1,1} = \sigma^{s_1s_2}$ is a Schubert class of degree $2$.

\begin{remark}
  Note that for $\OG(2,8)$ there are three Schubert classes in degree $2$
  \begin{equation*}
    \tau_2 = \sigma^{s_3s_2}, \quad \tau_{1,1} = \sigma^{s_1s_2}, \quad \tau_2' = \sigma^{s_4s_2},
  \end{equation*}
  whereas for $\OG(2,2n)$ with $n \geq 5$ there are only two such classes
  \begin{equation*}
    \tau_2 = \sigma^{s_3s_2} \quad \text{and} \quad \tau_{1,1} = \sigma^{s_1s_2}.
  \end{equation*}
  Nevertheless, the formula \eqref{eq:h-and-p-via-tau} holds for any $n \geq 4$ as we have
  \begin{equation}\label{eq:tau-one-squared-OG}
    \tau_1^2 =
    \begin{cases}
      \tau_2 + \tau_{1,1} + \tau_2' \quad \text{for} \quad n = 4, \\
      \tau_2 + \tau_{1,1} \quad \text{for} \quad n \geq 5,
    \end{cases}
  \end{equation}
  which can easily be deduced from the classical Chevalley formula \eqref{eq:classical-chevalley-formula}.
\end{remark}

We now want to express the classes $\gamma$ and $\Sigma_{n-2}$ in terms of Schubert classes.

\begin{lemma}\label{lemma:En-2-via-simple-roots}
  In $H^*(X,\QQ)$ we have the identity
  \begin{equation}\label{eq:En-2-via-simple-roots}
    \Sigma_{n-2} = 2 \sum_{i = 1}^{n-2} (-1)^{i-1} \sigma_{\alpha_i} + (-1)^{n}(\sigma_{\alpha_{n-1}} + \sigma_{\alpha_{n}}).
  \end{equation}
\end{lemma}

\begin{proof}
  According to \eqref{eq:psi-via-h-p-sigma} and \eqref{eq:psi-via-tau} we have
  \begin{equation*}
    h\Sigma_{n-2} = \psi_{2n-3} = 2\tau_{2n-3}.
  \end{equation*}
  By the hard Lefschetz theorem the map $H^{2(2n-4)}(X,\QQ) \to H^{2(2n-3)}(X,\QQ)$ given by the multiplication with $h$ is an isomorphism. Hence, to show the identity \eqref{eq:En-2-via-simple-roots}, it is enough to check that its right-hand side multiplied with $h$ is equal to $2\tau_{2n-3}$.
  This can be easily checked using the classical Chevalley formula \eqref{eq:classical-chevalley-formula}.
\end{proof}

\begin{lemma}
In $H^*(X,\QQ)$ we have the identity
\begin{equation*}
  \gamma = \pm(\tau_{n-2} - \tau'_{n-2}),
\end{equation*}
which is to be understood up to sign, which comes from exchanging $F_n$ with $F'_{n}$ in the definition of the flag.
\end{lemma}

\begin{proof}
  The presentation \eqref{eq:borel-presentation-III-simplified} shows that the map $H^{2n-4}(X,\QQ) \to H^{2n}(X,\QQ)$ given by multiplication with $p$ has kernel equal to the $\QQ$-span of $\gamma$. Indeed up to degree $n$, the only relation is given by $p\gamma = 0$. Forgetting this relation we would have a polynomial algebra which is a domain and the multiplication would be injective. Adding this only relation in degree $n$ leads to the claim.

  Now from the Pieri formula given in \cite[Theorem 3.1]{BKT} it is easy to check that $p\tau_{n-2} = p\tau'_{n-2}$ so that $\gamma = \lambda(\tau_{n-2} - \tau'_{n-2})$.

  We compute $\lambda$ using the equation $\gamma^2 = (-1)^{n-2}\Sigma_{n-2}$
  (this is the second relation in \eqref{eq:borel-presentation-III-simplified}).
  In particular we consider the coefficient of $\sigma_{\alpha_1} = \tau_{2n-4}$.
  For a cohomology class $\psi$ of degree $2n-4$ written in the basis
  $(\sigma_\alpha)_{\alpha \in R}$, we write $L(\psi)$ for its coordinate on
  $\sigma_{\alpha_1}$. We have $L(\gamma^2) = (-1)^{n-2}L(\Sigma_{n-2}) = 2(-1)^n$.
  An easy application of the Pieri formula (or a direct geometric check) shows that we have
$$L(\tau_{n-2}^2) = L({\tau'_{n-2}}^2)
  = \left\{
  \begin{array}{ll}
    1 & \textrm{ for $n$ even} \\
    0 & \textrm{ for $n$ odd,} \\
  \end{array}\right.
  \textrm{ and }
  L(\tau_{n-2}\tau'_{n-2})
  = \left\{
  \begin{array}{ll}
    0 & \textrm{ for $n$ even} \\
    1 & \textrm{ for $n$ odd.} \\
  \end{array}\right.$$
  We get $2(-1)^n = L(\gamma^2) = 2\lambda^2(-1)^n$ so that $\lambda^2 = 1$ and $\lambda = \pm1$. Exchanging $F_n$ and $F_n'$ changes the sign of $\gamma$ proving the result.
\end{proof}

\subsection{Small quantum cohomology}

In \cite{BKT} a presentation for the small quantum cohomology ring $\QH(X)$ in terms of special Schubert classes is given. Here we give another presentation based on the Borel presentation \eqref{eq:borel-presentation-III-simplified}, which is better suited for our needs.

To give a presentation for $\QH(X)$ from the presentation given in Corollary \ref{corollary:borel-presentation-III-simplified}, we only need to deform the equations $EQ_{n}$, $EQ_{2n-4}$ and $EQ_{2n-2}$ by replacing the product in cohomology by the quantum product $\star_0$. Write $EQ_k^{\star_0}$ for the polynomial in $h$, $p$ and $\gamma$ obtained from $EQ_k$ by replacing the product in cohomology by $\star_0$.

\begin{lemma}
  In $\QH(X)$ we have the following equalities
  \begin{equation*}
    \begin{aligned}
      & EQ_n^{\star_0} = 0, \\
      & EQ_{2n-4}^{\star_0} = 0, \\
      & EQ_{2n-2}^{\star_0} = -4qh.
    \end{aligned}
  \end{equation*}
\end{lemma}

\begin{proof}
  The first two formulas follow immediately from the fact that $\deg(q) = 2n-3$ is greater than the degree of the expression. Hence, the quantum product coincides with the classical product and the expression vanishes.

  First, note that in $H^*(X,\bQ)$ by Lemma \ref{lemm-elim}, \eqref{eq:psi-via-tau} and \eqref{eq:psi-via-h-p-sigma} we have the equality
  \begin{equation*}
    E_{n-2}(h,p)+pE_{n-3}(h,p) = \Sigma_{n-2} + p\Sigma_{n-3} = \psi_{2n-4} = 2\tau_{2n-4}.
  \end{equation*}
  Since its degree is smaller than $\deg(q) = 2n-3$, this equality also holds in $\QH(X)$. Thus, we conclude
  \begin{equation}\label{eq:EQ-2n-2-proof}
    EQ_{2n-2}^{\star_0} = h \star_0 h \star_0 E_{n-2}(h,p) - 2 p \star_0 \tau_{2n-4}.
  \end{equation}
  To compute the first summand in \eqref{eq:EQ-2n-2-proof} we use Lemma \ref{lemma:En-2-via-simple-roots} and the quantum Chevalley formula given in Theorem \ref{theorem:quantum-chevalley} to get
  \begin{equation*}
    h \star_0 \Sigma_{n-2} = 2 \tau_{2n-3} - 2q,
  \end{equation*}
  and then once more to get
  \begin{equation*}
    h \star_0 h \star_0 E_{n-2}(h^2,p) = 2\sigma_{-(\alpha_1 + \alpha_2)} - 2qh.
  \end{equation*}
  To compute the second summand in \eqref{eq:EQ-2n-2-proof} we apply the quantum Pieri formula \cite[Theorem 3.4]{BKT} and obtain
  \begin{equation*}
    p \star_0 \tau_{2n-4} = \tau_{2n-2} + q = \sigma_{-(\alpha_1 + \alpha_2)} + qh.
  \end{equation*}
  Putting everything together, we get the desired equality $EQ_{2n-2}^{\star_0} = -4qh$.
\end{proof}

\begin{corollary}
  \label{corollary:presentation-small-QH-type-D}
  We have the following presentation
  \begin{equation}\label{eq:presentation-small-QH-type-D}
    \QH(X) = K[h,p,\gamma]/(EQ_n, EQ_{2n-4}, EQ_{2n-2} + 4qh).
  \end{equation}
\end{corollary}

\vspace{10pt}

This proves Theorem \ref{theorem:introduction-uniform-presentation-for-QH} in type $\rD_n$.

\begin{remark}
  Using \eqref{eq:presentation-small-QH-type-D} it is not difficult to see that the small quantum deformation of the presentation given in Proposition \ref{cor-pres-xi} is
  \begin{equation}\label{eq:presentation-small-QH-more-variables}
    \QH(X) = K[h,p,\gamma,(\Sigma_j)_{j \in [1,n-3]}]/(\xi_n, (\Xi_i)_{i \in [1,n - 2]} , \Xi_{n-1} + (-1)^{n-2}4qh).
  \end{equation}
\end{remark}

\bigskip

Since the small quantum cohomology $\QH(X)$ is a finite dimensional $K$-algebra,
we can view it as the algebra of functions on the finite scheme $\spec(\QH(X))$.
Moreover, the presentation \eqref{eq:presentation-small-QH-type-D} defines the closed embedding
\begin{equation*}
  \spec(\QH(X)) \subset \spec(K[h,p,\gamma]) = \bA^3.
\end{equation*}
We call \textsf{origin} the point in $\spec(\QH(X))$ corresponding to the maximal ideal $(h,p,\gamma)$.
\begin{lemma}
  The origin is contained in $\Spec \QH(X)$ and the Zariski tangent space
  at this point is of dimension $2$. In particular, the origin is a fat point of
  $\Spec \QH(X)$ and the algebra $\QH(X)$ is not semisimple.
\end{lemma}

\begin{proof}
  All claims follow immediately from Corollary \ref{corollary:presentation-small-QH-type-D},
  as polynomials $E_j(h,p)$ appearing in the presentation have no constant term.
  Note that the relations $EQ_n^{\star_0} = 0$ and $EQ_{2n-4}^{\star_0} = 0$ give
  no relation in the tangent space at the origin, while the relation
  $EQ_{2n-2}^{\star_0} = -4qh$ induces the equation $h = 0$ explaning the dimension
  of the tangent space.
\end{proof}

\begin{remark}
    We recover the fact that $\QH(X)$ is not semisimple (see \cite{ChPe}).
\end{remark}

In the next proposition we study the structure of the finite scheme
$\Spec \QH(X)$ in more detail. Note that we have
$\dim_K \QH(X) = \dim_{\bQ} H^*(X, \bQ) = 2n(n-1)$.

\begin{proposition}
  \label{prop:decomp-A-B}
  We have a decomposition
  \begin{equation*}
    \QH(X) = A\times B,
  \end{equation*}
  where $A$ is a fat point of length $n$ supported at the origin and $B$ is a
  semisimple algebra corresponding to $n(2n-3)$ reduced points in $\Spec(\QH(X))$
  different from the origin.
\end{proposition}

\begin{proof}
  Since $\dim_{K} \QH(X) = 2n(n-1)$,
  to prove the claim it is enough to show the inequatities
  \begin{equation}\label{eq:dimension-inequality-A}
    \dim_K A \geq n
  \end{equation}
  and
  \begin{equation}\label{eq:dimension-inequality-B}
    \dim_K B \geq n(2n-3).
  \end{equation}

  To show \eqref{eq:dimension-inequality-A} it is enough to consider the intersection with the locus $h = 0$, i.e. to add $h$ to the relations in the presentation \eqref{eq:presentation-small-QH-type-D}. Since $E_j(0,p) = (-1)^{j}(j+1)p^j$, the resulting algebra becomes
  \begin{equation*}
    K[p,\gamma]/(p\gamma, \gamma^2 - (n-1)p^{n-2},  p^{n-1}).
  \end{equation*}
  As this algebra is of dimension $n$, the inequality \eqref{eq:dimension-inequality-A} follows.

  To show \eqref{eq:dimension-inequality-B} we proceed as follows. Consider the $2$-to-$1$ cover
  of $\Spec \QH(X)$ defined by the injective homomorphism of algebras (note that we use
  \eqref{eq:presentation-small-QH-type-D} and \eqref{eq:presentation-small-QH-more-variables} here):
  \begin{equation}\label{eq:covering-map}
    K[h,p,\gamma]/(EQ_n, EQ_{2n-4}, EQ_{2n-2} + 4qh) \to K[x_1, \dots, x_n]/I
  \end{equation}
with $I = (\xi_n, \left( \Xi_i \right)_{i \in [1,n-2]}, \Xi_{n-1} + (-1)^{n-2}4q(x_1 + x_2))$ induced by
  \begin{equation}\label{eq:covering-map-formulas}
    h \mapsto x_1 + x_2, \quad p \mapsto x_1x_2, \quad \gamma \mapsto x_3 \dots x_n,
  \end{equation}
  and where $\xi_n$ and $\Xi_i$ are defined in \eqref{eq:borel-presentation-II}.

  Note that the generators of $I$
  can be rewritten in the form of polynomial identities
  \begin{equation}\label{eq:covering-map-polynomial-identities}
    \begin{aligned}
      & \prod_{i = 1}^n(t^2 - x_i^2) = t^2(t^{2n-2} - 4q(x_1 + x_2)),  \\
      & x_1 \cdots x_n = 0,
    \end{aligned}
  \end{equation}
  which we are going to solve to estimate the dimension of $B$. From \eqref{eq:covering-map-formulas} it follows that that the set-theoretic the preimage of the origin $h = p = \gamma = 0$ is given by the equations $x_1 = x_2 = 0$ and $x_3 \dots x_n = 0$. Thus, we are only interested in the solutions of \eqref{eq:covering-map-polynomial-identities} outside of this set. We consider three cases.

  \begin{enumerate}
    \item \emph{Case $x_1 = 0, x_2 \neq 0$.}  The first identity in \eqref{eq:covering-map-polynomial-identities} becomes
      \begin{equation*}
        \prod_{i = 2}^n(t^2 - x_i^2) = (t^{2n-2} - 4qx_2).
      \end{equation*}
      Since $t = x_2$ is a solution, we get a relation on $x_2$ of the form
      \begin{equation*}
        x_2^{2n-3} = 4q.
      \end{equation*}
      Thus, we see that we have $2n-3$ choices for $x_2$. From here one also obtains $x_3, \dots, x_n$ up to permutations and signs.

      These solutions in terms of $x_1, \dots, x_n$ give rise to the following solutions in terms of $h,p,\gamma$
      \begin{equation*}
        h = x_2, \quad p = 0, \quad \gamma = x_3 \dots x_n.
      \end{equation*}
      Thus, in total we have $2(2n-3) = 4n - 6$ solutions in terms of $h,p,\gamma$, where the factor $2$ comes from the sign.

    \medskip

    \item \emph{Case $x_1 = 0, x_2 \neq 0$.} This can be treated identially to the previous one. However, these solutions give the same answers for $h$, $p$ and $\gamma$. So we don't need to take them into account.

    \medskip

    \item \emph{Case $x_1 \neq 0, x_2 \neq 0, x_1 \neq x_2$.} The second identity in \eqref{eq:covering-map-polynomial-identities} implies that at least one of the $x_i$ for $i \in [3,n]$ must vanish. Therefore, we immediately obtain $\gamma = 0$.

    By our assumption $x_1$ and $x_2$ are two distinct roots of the first identity in~\eqref{eq:covering-map-polynomial-identities}. Hence, by substituing them into this identity, we obtain the system of equations
    \begin{equation}\label{eq:covering-map-polynomial-identities-case3-I}
      \begin{aligned}
        & x_1^{2n-2} = 4q(x_1+x_2) \\
        & x_2^{2n-2} = 4q(x_1+x_2).
      \end{aligned}
    \end{equation}
    Eliminating $x_2$ from \eqref{eq:covering-map-polynomial-identities-case3-I} we obtain
    \begin{equation*}
      (x_1^{2n-2} - 4qx_1)^{2n-2} = (4qx_1)^{2n-2},
    \end{equation*}
    and,  since $x_1 \neq 0$, we can further rewrite it as
    \begin{equation}\label{eq:covering-map-polynomial-identities-case3-II}
      (x_1^{2n-3} - 4q)^{2n-2} = (4q)^{2n-2}.
    \end{equation}
    It is clear that \eqref{eq:covering-map-polynomial-identities-case3-II} is
    a polynomial of degree $(2n-3)(2n-2)$, which has zero as a root of multiplicity $2n-3$
    and all other roots are of multiplicity one. Therefore, taking into account
    further $2n-3$ solutions of the form $x_1 = x_2 \neq 0$, the number of solutions
    of \eqref{eq:covering-map-polynomial-identities-case3-I} satisfying
    $x_1 \neq 0, x_2 \neq 0, x_1 \neq x_2$ is equal to $(2n-3)(2n-4)$.

    Finally, the above $(2n-3)(2n-4)$ solutions for $x_i$'s give rise to $(2n-3)(n-2)$ solutions for $h, p, \gamma$.
  \end{enumerate}
  Adding the contributions of the first and the third case gives the desired bound \eqref{eq:dimension-inequality-B}.
\end{proof}

\begin{lemma}\label{lemma:description-fat-point-type-D}
  The non-reduced factor $A$ of $\QH(X)$ described in Proposition \ref{prop:decomp-A-B}
  has the following explicit presentation
  \begin{equation*}
    A \simeq K[x,y]/(xy, x^2 + (n-1)y^{n-2}).
  \end{equation*}
  In particular, $A$ is isomorphic to the Jacobi ring of the simple isolated
  hypersurface singularity of type $\rD_n$, i.e. we have
  \begin{equation*}
    A \simeq K[x,y]/(f'_x, f'_y) \quad \text{with} \quad f = x^2y + y^{n-1}.
  \end{equation*}
\end{lemma}

\begin{proof}
  Since $A$ is the localisation of $\QH(X)$ at the origin, to prove the above
  claims we can consider the relations \eqref{eq:presentation-small-QH-type-D}
  for $\QH(X)$ in the formal power series ring $K[[h,p,\gamma]]$.

  First we note that from \eqref{eq:definition-polynomials-Ej} we have congruences
  \begin{equation*}
    E_k \equiv h^{2k}\ ({\rm mod}\ p) \qquad \text{and} \qquad E_k \equiv (-1)^k(k+1)p^k\ ({\rm mod}\ h).
  \end{equation*}

  From the relation $EQ_{2n-2} + 4qh = 0$ it follows that there exists
  $d(p) \in (p) \subset K[[p]]$ such that
  \begin{equation}\label{eq:expression-for-h-via-p}
    h = (-p)^{n-1} \left( -\frac{1}{4q} + d(p) \right),
  \end{equation}
  which allows us to get rid of the variable $h$ and of the relation $EQ_{2n-2} + 4qh = 0$.

  Now we can substitue \eqref{eq:expression-for-h-via-p} instead of $h$ into the relation
  $EQ_{2n-4} = 0$ and obtain
  \begin{equation*}
    \gamma^2 - (n-1)p^{n-2} \left( 1 + e(p) \right) = 0
  \end{equation*}
  for some power series $e(p) \in (p) \subset K[[p]]$.

  Let us define
  \begin{equation*}
    x = \gamma \qquad \text{and} \qquad y = p\sqrt[n-2]{- 1 - e(p)}.
  \end{equation*}
  In this notation the relation $EQ_{2n-4} = 0$ becomes $x^2 + (n-1)y^{n-2} = 0$
  and the relation $EQ_{n-2} = 0$ becomes $xy = 0$. The claims now follow.
\end{proof}

\subsection{Big quantum cohomology}

Recall the presentation \eqref{eq:presentation-small-QH-type-D} of the small
quantum cohomology of $X$
\begin{equation*}
  \QH(X) = K[h,p,\gamma]/(EQ_n, EQ_{2n-4}, EQ_{2n-2} + 4qh).
\end{equation*}
In this subsection, we prove the generic semisimplicity of $\BQH(X)$ as outlined
in the introduction. We need to show that in the big quantum cohomology we have
a deformation, whose relations are of the following form with $\lambda_\gamma,\lambda_p \in \QQ^\times$:
\begin{equation*}
  \begin{aligned}
    & EQ_n + \lambda_\gamma qt_{\gamma} + \text{higher degree terms} \\
    & EQ_{2n-4} + \lambda_p qt_p + \text{higher degree terms} \\
    & EQ_{2n-2} + 4qh + \text{higher degree terms}
  \end{aligned}
\end{equation*}

We will see that the above form is prescribed by the degree of the generators and
equations so we will only need to explicitely compute the values of $\lambda_\gamma$
and $\lambda_p$. For this we will need to compute four-points Gromov--Witten invariants
of degree $1$ and we will use the results of Section \ref{section:geometry-of-the-space-of-lines}.

\begin{proposition}
  We have the following formulas.
  \begin{equation}\label{eq:GW-pt-p-tau-tau}
    \scal{\ptclass, p, \tau_{n-2}, \tau_{n-2}}_1^X = \scal{\ptclass, p, \tau'_{n-2}, \tau'_{n-2}}_1^X =
    \begin{cases}
      0 \quad \text{if $n$ is even} \\
      1 \quad \text{if $n$ is odd}
    \end{cases}
  \end{equation}
  \begin{equation}\label{eq:GW-pt-p-tau-tauprime}
    \scal{\ptclass, p, \tau_{n-2}, \tau'_{n-2}}_1^X =
    \begin{cases}
      1 \quad \text{if $n$ is even} \\
      0 \quad \text{if $n$ is odd}
    \end{cases}
  \end{equation}

  \begin{equation}\label{eq:GW-pt-p-gamma-gamma}
    \scal{\ptclass, p,\gamma,\gamma}_1^X = (-1)^{n-1}2.
  \end{equation}

  \begin{equation}\label{eq:GW-pt-gamma-gamma-gamma}
    \scal{\ptclass, \gamma, \gamma, \gamma}_1^X = 0
  \end{equation}

  \begin{equation}\label{eq:GW-pt-p-p-tau}
    \scal{\ptclass, p, p, \tau}_1^X = 0 \textrm{ for any class } \tau
  \end{equation}

  \begin{equation}\label{eq:GW-pt-h-tau-2n-5}
    \scal{\ptclass, p, h, \tau_{2n-5}}_1^X = 1
 \textrm{ and }
    \scal{\ptclass, \gamma, h, \tau_{2n-5}}_1^X = 0.
  \end{equation}

\end{proposition}

\begin{proof}
Recall from Section \ref{section:geometry-of-the-space-of-lines} that for $X = \OG(2,2n)$
the variety $\rF_x$ of lines passing through a point has the following description.
The variety $X$ corresponds to the second vertex of the Dynking diagram of type $\rD_n$
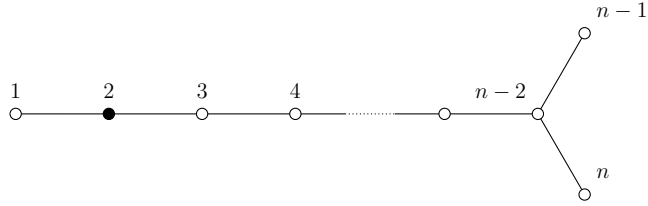
\begin{equation*}
\begin{dynkinDiagram}[edge length = 2em,labels*={1,2,3,4,,n-2, n-1, n},scale=1.5]D{o*oo.oooo}
\end{dynkinDiagram}
\end{equation*}
To obtain $\rF_x$ we remove the second vertex from the diagram, splitting it into
two connected components, and mark the two vertices that used to be adjacent to the
second vertex:
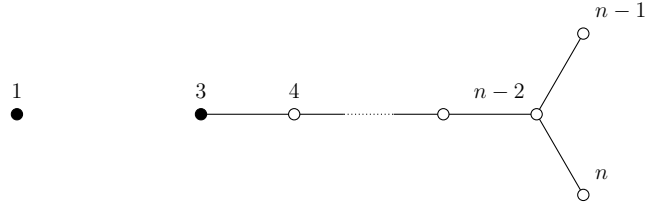
\begin{equation*}
\begin{dynkinDiagram}[edge length = 2em,labels*={1},scale=1.5]A{*}
\end{dynkinDiagram}
\hspace{60pt}
\begin{dynkinDiagram}[edge length = 2em,labels*={3,4,,n-2, n-1, n},scale=1.5]D{*o.oooo}
\end{dynkinDiagram}
\end{equation*}
in this process we keep the original indexing of the vertices. Hence, we obtain
\begin{equation*}
  \rF_x = \bP^1 \times \rQ^{2n-6}.
\end{equation*}
The Schubert basis for $H^*(\bP^1, \bQ)$ is given by
\begin{equation*}
  1 = \sigma^{e} \quad \text{and} \quad \zeta = \sigma^{s_1},
\end{equation*}
where $\zeta$ can alternatively be described as the hyperplane class or the class
of a point. The Schubert basis for $H^*(\rQ^{2n-6}, \bQ)$ is given by
\begin{equation*}
  \begin{aligned}
    & 1 = \sigma^{e}, \\
    & \xi_{i} = \sigma^{s_{i+2}s_{i+1} \dots s_3} \quad \text{for} \quad i \in [1,n-2], \\
    & \xi_{i} = \sigma^{s_{2n - 3 - i} \dots s_{n-2} s_{n} s_{n-1} \dots s_3} \quad \text{for} \quad i \in [n-1,2n-6], \\
    & \xi_{n-3}' = \sigma^{s_n s_{n-2} s_{n-3} \dots s_3}.
  \end{aligned}
\end{equation*}
Finally, the Schubert basis for $H^*(\bP^1 \times \rQ^{2n-6}, \bQ)$ is given by
the Künneth formula.

Let us recall (for example, see \cite[Theorem 1.13]{Reid}) that in $H^*(\rQ^{2n-6}, \bQ)$ the following identities hold
\begin{equation}\label{eq:prod-in-quad}
   \xi_{n-3}^2 = (\xi'_{n-3})^2 =
   \begin{cases}
     1 \quad \text{if $n - 3$ is even} \\
     0 \quad \text{if $n - 3$ is odd}
   \end{cases} 
   \textrm{ and }
   \xi_{n-3}\xi'_{n-3} =
   \begin{cases}
     0 \quad \text{if $n - 3$ is even} \\
     1 \quad \text{if $n - 3$ is odd}.
   \end{cases}
 \end{equation}

We also restate for convenience of the reader the following formulas
\begin{equation}\label{eq:recalling-formulas-for-p-tau-tau-prime}
  p = \sigma^{s_1s_2}, \quad \tau_{n-2} = \sigma^{s_{n-1} s_{n-2} s_{n-3} \dots s_2}, \quad \tau_{n-2}' = \sigma^{s_n s_{n-2} s_{n-3} \dots s_2}, \quad \gamma = \pm(\tau_{n-2} - \tau_{n-2}'),
\end{equation}
which appeared earlier in this section.

\medskip
Finally, we are now ready to compute the necessary GW invariants.

\smallskip
\noindent \textbf{Invariants \eqref{eq:GW-pt-p-tau-tau} -- \eqref{eq:GW-pt-p-tau-tauprime}:}
These follow immediately from Lemma \ref{lemma:quantum-to-schubert-calculus-long-root}
and \eqref{eq:prod-in-quad}--\eqref{eq:recalling-formulas-for-p-tau-tau-prime}.

\smallskip
\noindent \textbf{Invariant \eqref{eq:GW-pt-p-gamma-gamma}:} This invariant follows
from \eqref{eq:GW-pt-p-tau-tau} -- \eqref{eq:GW-pt-p-tau-tauprime}.

\smallskip
\noindent \textbf{Invariant \eqref{eq:GW-pt-gamma-gamma-gamma}:} This invariant
reduces via Lemma \ref{lemma:quantum-to-schubert-calculus-long-root} to a linear
combination of triple intersection numbers on $\rQ^{2n-6}$. Since all the cohomology
classes involved are of degree $n-3$, all such triple intersections vanish for
degree reasons.

\smallskip
\noindent \textbf{Invariant \eqref{eq:GW-pt-p-p-tau}:} Similarly to the above,
this vanishing follows from the fact that $\zeta^2 = 0$ in $H^*(\bP^1,\bQ)$.

\smallskip
\noindent \textbf{Invariant \eqref{eq:GW-pt-h-tau-2n-5}:} It follows from the fact
that $q_*p^*\tau_{2n-5} = 1 \otimes \xi_{2n-6}$ so its product is $1$ with
$q_*p^*p = \zeta \otimes 1$ and $0$ with $q_*p^*\gamma = 1 \otimes (\xi_{n-3} - \xi'_{n-3})$.
\end{proof}

\begin{proposition}\label{proposition:Dn-presentation-big-QH}
  The ring $\BQH_{p, \gamma}(\OG(2,2n))$ is the quotient of
  \begin{equation*}
    \left(K[[t_p, t_{\gamma}]]\right)[h,p,\gamma]
  \end{equation*}
  by the ideal generated by
  \begin{equation}\label{eq:presentation-big-QH-type-D}
    \begin{aligned}
      & EQ_n + (-1)^n2q t_{\gamma} + \gt^2, \\
      & EQ_{2n-4} +   (-1)^n 4 qt_p
      + \gt \gm, \\
      & EQ_{2n-2} + 4qh + \gt \gm,
    \end{aligned}
  \end{equation}
  where $\gt = (t_p,t_\gamma)$ and $\gm = (h,p,\gamma,t_p,t_\gamma)$.
\end{proposition}

\begin{proof}
  Relations in the big quantum cohomology ring are homogeneous deformations of relations in the small quantum cohomology ring. In our case we are deforming along $t_p$ and $t_{\gamma}$ directions only. Thus, we a looking for a homogenous deformation of \eqref{eq:presentation-small-QH-type-D} involving variables $t_p$ and $t_{\gamma}$. Moreover, the general properties of Gromov--Witten invariants ensure that any terms involving $t_p$ or $t_{\gamma}$ must necessarily have a factor of $q$ in it.

  Recall that we have
  \begin{equation*}
    \begin{aligned}
      & \deg(t_p) = 1 - |p| = -1, \\
      & \deg(t_{\gamma}) = 1 - |\gamma| = 3 - n, \\
      & \deg(q) = 2n-3.
    \end{aligned}
  \end{equation*}
  From these formulas and homogeneity it immediately follows that the necessary deformation of the small quantum cohomology relations \eqref{eq:presentation-small-QH-type-D} has to have the form prescribed by \eqref{eq:presentation-big-QH-type-D}. The only point that needs to be checked is the precise form of the linear terms
  $(-1)^n2q t_{\gamma}$ and $(-1)^n 4q t_p$ appearing in \eqref{eq:presentation-big-QH-type-D}. This is what we do in the rest of the proof.

  We denote by $EQ_k^{\star}$ the element of $\BQH_{p, \gamma}(\OG(2,2n))$ defined by the polynomial $EQ_k$, i.e. we use the product in $\star$ in $\BQH_{p, \gamma}(X)$ to multiply terms of the polynomial. For example, we have $EQ_n^\star = p \star \gamma$.

  \smallskip

  Let us consider the first relation in \eqref{eq:presentation-big-QH-type-D}. The linear terms of $p \star \gamma$ are given by $\scal{p, \gamma, \gamma, \ptclass}_1^Xqt_{\gamma}$ and $\scal{p, \gamma, p, \ptclass}_1^Xqt_p$, where for degree reasons the latter one is only potentially non-zero for $n = 4$.
  Applying \eqref{eq:GW-pt-p-gamma-gamma} and \eqref{eq:GW-pt-p-p-tau} we get the claim.

  In the rest of the proof we deal with the second relation in \eqref{eq:presentation-big-QH-type-D}. Recall from \eqref{eq:borel-presentation-III-simplified} that we have $EQ_{2n-4} = \gamma^2 + (-1)^{n-1}E_{n-2}(h,p)$. As above we need to compute $EQ_{2n-4}^\star$
  and we are only interested in the linear terms $q t_{\gamma}$ and $q t_p$. We treat the summands $\gamma \star \gamma$
  and $\left((-1)^{n-1}E_{n-2}(h,p)\right)^\star$ separately.

  First we consider $\gamma \star \gamma$. Here the linear terms are given by $\scal{\gamma,\gamma,p, \pt}_1 qt_p$ and $\scal{\gamma,\gamma,\gamma,\pt}_1 qt_{\gamma}$. Note that for degree reasons the latter term can only be potentially non-zero for $n = 4$. By \eqref{eq:GW-pt-p-gamma-gamma} we have $\scal{\gamma,\gamma,p, \pt}_1 qt_p = (-1)^{n-1}2 qt_p$
  and by \eqref{eq:GW-pt-gamma-gamma-gamma} we have $\scal{\gamma,\gamma,\gamma,\pt}_1 qt_{\gamma} = 0$. Thus, the linear term of $\gamma \star \gamma$ is
  \begin{equation}\label{eq:proof-BQH-linear-term-gamma-gamma}
    (-1)^{n-1}2 qt_p
  \end{equation}

  Now we consider $\left(E_{n-2}(h,p)\right)^\star$. From \eqref{eq:definition-polynomials-Ej} is clear that
  $E_{n-2}(h,p) \equiv h^{2n-4} \ ({\rm mod} \ p)$. Therefore, $\left(E_{n-2}(h,p)\right)^\star = h^{\star 2n-4} + p \star \tau$, where $\tau$ is a linear combination of Schubert classes of degree $2n-6$. We consider $h^{\star 2n-4}$ and $p \star \tau$ separately.

  The linear terms of $p \star \tau$ are given by $\scal{p,\tau,p,\ptclass}_1qt_p$ and $\scal{p,\tau,\gamma,\ptclass}_1qt_{\gamma}$, where the latter term can only be potentially non-zero for $n = 4$. By \eqref{eq:GW-pt-p-p-tau} we have $\scal{p,\tau,p,\ptclass}_1qt_p = 0$. To deal with $\scal{p,\tau,\gamma,\ptclass}_1qt_{\gamma}$
  in the case $n = 4$ we note that by \eqref{eq:definition-polynomials-Ej} we have $E_2(h,p) = h^4 + p(3p-4h^2)$.
  Thus, in the case $n = 4$ by \eqref{eq:psi-via-tau} -- \eqref{eq:tau-one-squared-OG}
  we have
  \begin{equation*}
    \tau = 3p-4h^2 = -4(\tau_2 + \tau'_2) - p.
  \end{equation*}
  Hence, using \eqref{eq:GW-pt-p-tau-tau}, \eqref{eq:GW-pt-p-tau-tauprime} and \eqref{eq:GW-pt-p-p-tau}
  we get
  \begin{equation*}
    \scal{p,\tau,\gamma,\ptclass}_1 = -4 \scal{p,\tau_2+\tau'_2,\tau_2 -\tau'_2,\ptclass}_1 - \scal{p,p,\gamma,\ptclass}_1 = 0.
  \end{equation*}
  Finally, we conclude that the linear term of $p \star \tau$ always vanishes.

  Now let us consider $h^{\star 2n-4}$. By the Chevalley formula (see Theorem \ref{theorem:quantum-chevalley}), we have
  \begin{equation*}
    h^{2n-5} = \tau_1^{2n-5} = (\sigma_{\theta - \alpha_2})^{2n-5} = 2 \tau_{2n-5} + \sigma,
  \end{equation*}
  (no quantum corrections for degree reasons), where $\sigma$ is a linear combination of Schubert classes of the form $\sigma_{\alpha_i + \alpha_j}$ with $i,j \geq 2$ (also for degree reasons). For $n = 4$, we have
    $\tau_{2n-5} = \sigma_{\alpha_1 + \alpha_2}$ and
    $\sigma = 2(\sigma_{\alpha_2 + \alpha_3} + \sigma_{\alpha_2 + \alpha_4}) = 2(\sigma^{s_4s_1s_2} + \sigma^{s_3s_1s_2})$.

  Multiplying with $h$, we have
  \begin{equation*}
    h^{\star 2n-4} = 2 \tau_1 \star \tau_{2n-5} + \tau_1 \star \sigma.
  \end{equation*}
  We treat these summands separately:
  \begin{enumerate}
  \item The linear terms of $\tau_1 \star \tau_{2n-5}$ are given by $\scal{\tau_1 , \tau_{2n-5},p,\ptclass}_1qt_p$ and $\scal{\tau_1 ,\tau_{2n-5},\gamma,\ptclass}_1qt_{\gamma}$, where the latter term can only be potentially non-zero for $n = 4$. According to \eqref{eq:GW-pt-h-tau-2n-5},
    we have $\scal{\tau_1 , \tau_{2n-5},p,\ptclass}_1qt_p = qt_p$ and $\scal{\tau_1 ,\tau_{2n-5},\gamma,\ptclass}_1qt_{\gamma} = 0$.

    \smallskip

    \item The linear terms of $\tau_1 \star \sigma$ are given by $\scal{\tau_1 , \sigma,p,\ptclass}_1qt_p$ and $\scal{\tau_1 ,\sigma,\gamma,\ptclass}_1qt_{\gamma}$, where the latter term can only be potentially non-zero for $n = 4$.

      To deal with $\scal{\tau_1 , \sigma,p,\ptclass}_1qt_p$ we note that $q_*p^*\sigma = \zeta \otimes \omega$ in $H^*(\bP^1 \times \Q^{2n-6})$ because $\sigma$ is a linear combinaison of classes of the form $\sigma_{\alpha_i + \alpha_j}$ with $i,j \geq 2$.
      Hence, since $q_*p^*{p} = \zeta \otimes 1$, we get $\scal{\tau_1 , \sigma,p,\ptclass}_1 = 0$.

    To deal with $\scal{\tau_1 ,\sigma,\gamma,\ptclass}_1qt_{\gamma}$ in the case $n = 4$ we note that $q_*p^*\sigma = \zeta \otimes (\xi_1 + \xi'_1)$ and $q_*p^*\gamma = 1 \otimes (\xi_1 - \xi'_1)$ and using Lemma \ref{lemma:quantum-to-schubert-calculus-long-root}, we have
    \begin{equation*}
      \scal{\tau_1 ,\sigma,\gamma,\ptclass}^X_1 = 2 \scal{1, \zeta \otimes (\xi_1 +\xi'_1), 1 \otimes (\xi_1 - \xi'_1) }_0^{\p^1 \times \rQ^2} = 0,
    \end{equation*}
  \end{enumerate}
  Therefore, we see that the linear term of $h^{\star 2n-4}$ is given by $2qt_p$ so that the linear term in $(-1)^{n-1}E_{n-2}(p,h)$ is given by
  \begin{equation}\label{eq:proof-BQH-linear-term-h-power-2n-4}
    (-1)^{n-1}2qt_p
  \end{equation}
Combining \eqref{eq:proof-BQH-linear-term-gamma-gamma} and
  \eqref{eq:proof-BQH-linear-term-h-power-2n-4}, we obtain that the linear term of $(\gamma^2 + (-1)^{n-1}E_{n-2}(h,p))^\star$ is given by $(-1)^{n-1}4 qt_p$.
\end{proof}

We obtain the following immediate corollary.

\begin{corollary}\label{corollary:regularity-and-semisimplicity-of-BQH-type-D}
  The following statements hold:
  \begin{enumerate}
    \item $\BQH(\OG(2,2n))$ is a regular ring.

    \item $\BQH(\OG(2,2n))$ is generically semisimple.
  \end{enumerate}
\end{corollary}

\begin{proof}
  (1) As the origin $h = p = \gamma = 0$ is the only non-reduced point of $\QH(\OG(2,2n))$,
  to prove that $\BQH_{p,\gamma}(\OG(2,2n))$ is regular it is enough to compute
  the dimension of the Zariski tangent space to $\BQH_{p,\gamma}(\OG(2,2n))$
  at the point $h = p = \gamma = t_p = t_\gamma = 0$. This dimension is easily seen
  to be equal to $\dim \BQH_{p,\gamma}(\OG(2,2n)) = 2$, by examining the linear
  terms of \eqref{eq:presentation-big-QH-type-D}.

  An identical argument proves the regularity of $\BQH(\OG(2,2n))$.

  \vspace{2pt}

  (2) Recall the deformation picture \eqref{Eq.: BQH family} and note that
  the ring $\BQH(X)_{\eta}$
  is a localisation of $\BQH(X)$. Since $\BQH(X)$ is regular by Part (1)
  and since regularity is preserved under
  localisation, we conclude that $\BQH(X)_{\eta}$ is also a regular ring. This implies
  that $\BQH(X)_{\eta}$ is a product of finite field extensions of $K((t_0, \dots, t_s))$,
  which was our definition of semisimplicity.
\end{proof}

\section{Types $\rE_6, \rE_7, \rE_8$}
\label{section:type-E}

In this section we treat the coadjoint varieties in Dynkin types $\rE_6, \rE_7, \rE_8$.
For convenience we recall the corresponding Dynkin diagrams
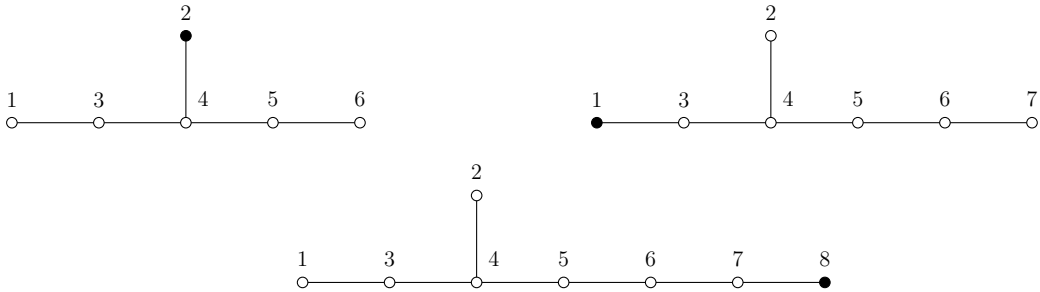
\begin{equation*}
  \begin{aligned}
    & \dynkin[edge length = 2em,labels*={1,...,6}, arrow width=2mm, scale=1.4]E{o*oooo}
    \hspace{80pt}
    \dynkin[edge length = 2em,labels*={1,...,7}, arrow width=2mm, scale=1.4]E{*oooooo} \\
    & \hspace{110pt} \dynkin[edge length = 2em,labels*={1,...,8}, arrow width=2mm, scale=1.4]E{ooooooo*}
  \end{aligned}
\end{equation*}
where the labelling of vertices follows \cite{Bo} and the black colored vertex corresponds
to the maximal parabolic subgroup defining the respective coadjoint variety.

Since $\rE_6, \rE_7, \rE_8$ are simply laced, according to Lemma \ref{lemm:wp-root}
we can index Schubert classes on the correponding coadjoint varieties by roots.
We denote simple roots by $\alpha_i$ and use the following shorthand notation
\begin{equation*}
  \alpha(a_1, \dots, a_n) \coloneqq \sum_{i = 1}^n a_i \alpha_i,
\end{equation*}
where $n \in \{6,7,8\}$ depending on the Dynkin type. This notation is used
throughout this section to label Schubert classes by roots.

On $X = \rG/\rP$ any cohomology class $\gamma \in H^*(X, \bQ)$ is a linear combination
of Schubert classes $\sigma^{w}_X$. The coefficients in this decomposition we denote by
$\coeff_{\sigma^{w}_X}(\gamma)$, so that we have
\begin{equation}\label{eq:coeff-definition}
  \gamma = \sum_{w \in \rW^{\rP}} \coeff_{\sigma^{w}_X}(\gamma) \, \sigma^{w}_X.
\end{equation}
In terms of the dual Schubert classes we have
\begin{equation*}
  \coeff_{\sigma^{w}_X}(\gamma) = \deg_{X} \left( \gamma \cup \left( \sigma^{w}_X \right)^\vee \right).
\end{equation*}

Since $\rE_6, \rE_7, \rE_8$ are simply laced, we rely on methods of Section \ref{section:geometry-of-the-space-of-lines}
to compute the necessary GW invariants. Recall the setting of \eqref{diagram:identification-long-root-universal-family}
and the content of Lemma \ref{lemma:quantum-to-schubert-calculus-long-root}.
We abbreviate $\F(X) = \rG/\rQ$ to $\F$ and $\Fpt(X) = \rP/\rR$ to $\Fpt$.
For any cohomology class $\gamma \in H^*(X, \bQ)$ we use the following shorthand notation
\begin{equation*}
  \bar{\gamma} \coloneqq i^*q_*p^*\gamma \in H^*(\Fpt, \bQ).
\end{equation*}
In particular, for Schubert classes we have
\begin{equation}\label{eq:bar-definition-schubert}
  \overline{\sigma^{w}_X} =
  \begin{cases}
    \sigma^{w s_{\rP}}_{\Fpt} \quad & \text{if} \quad w \neq 1 \quad \text{and} \quad w s_{\rP} \leq w^{\rR}_{\rP}, \\
    0 & \text{otherwise.}
  \end{cases}
\end{equation}
which are the type of cohomology classes appearing in
Lemma~\ref{lemma:quantum-to-schubert-calculus-long-root}(4).
From \eqref{eq:coeff-definition} and \eqref{eq:bar-definition-schubert} we obtain
that for any $\gamma \in H^*(X, \bQ)$ we have
\begin{equation}\label{eq:simple-fact-proofs}
  \coeff_{\sigma^{w}_{X}}(\gamma) = \coeff_{\sigma^{ws_\rP}_{\Fpt}}(\bar{\gamma})
  \quad \text{if} \quad w \neq 1 \quad \text{and} \quad w s_{\rP} \leq w^{\rR}_{\rP}.
\end{equation}

Our computations rely on the Littlewood--Richardson rules coming from Jeu de Taquin
proved in \cite{ThYo} for (co)minuscule varieties and in \cite{ChPeLR} for
coadjoint varieties and Schubert classes of degree at most $(\dim X)/2$.
These rules are purely combinatorial and have been implemented by the second
author \cite{LRCalc}, and we use this software for many computations below.
The scripts for the computations are available at
\begin{center}
  \texttt{https://github.com/msmirnov18/bqh-coadjoint}
\end{center}

\subsection{Type $\rE_6$}

Let $X = \rE_6/\rP_2$ be the coadjoint variety of type $\rE_6$. It is shown in \cite{ChPe} that the cohomology classes
\begin{equation*}
  \begin{aligned}
    & h = \sigma^{s_2}, \\
    & s = \sigma^{s_3s_4s_2}, \\
    & t = \sigma^{s_1s_3s_4s_2}.
  \end{aligned}
\end{equation*}
generate the cohomology ring $H^*(X, \bQ)$. The quantum parameter is of degree
\begin{equation*}
 \deg(q) = 11.
\end{equation*}
According to \cite{ChPe} we have the following presentation of the small quantum cohomology of $X$.
\begin{proposition}[{\cite[Proposition 5.4]{ChPe}}]
\label{proposition:E6-presentation-small-QH}
  The small quantum cohomology $\QH(X)$ is the quotient of the polynomial ring $K[h,s,t]$ modulo the ideal generated by
  \begin{equation*}
    h^8 - 6h^5s + 3h^4t + 9h^2s^2 - 12hst + 6t^2,
  \end{equation*}
  \begin{equation*}
    h^9 - 4h^6s + 3h^5t + 3h^3s^2 - 6h^2st + 2s^3,
  \end{equation*}
  \begin{equation*}
    -97h^{12} + 442h^9s - 247h^8t - 507h^6s^2 + 624h^5st - 156h^2s^2t + 48hq.
  \end{equation*}
\end{proposition}

\vspace{10pt}

As already mentioned in the introduction, the above proposition implies Theorem \ref{theorem:introduction-uniform-presentation-for-QH} in type $\rE_6$.

\begin{remark}
  A curious reader can verify the validity of the relations of Proposition \ref{proposition:E6-presentation-small-QH} using \cite{LRCalc}.
\end{remark}

\begin{lemma}
  \label{lemma:description-fat-point-E6}
  The small quantum cohomology $\QH(X)$ is not semisimple. Its unique non-semisimple factor is supported at the point $h = s = t = 0$ and is isomorphic to the Jacobian algebra of the isolated hypersurface singularity of type $\rE_6$.
\end{lemma}

\begin{proof}
  From the presentation in Proposition \ref{proposition:E6-presentation-small-QH}
  it is clear that the point $h = s = t = 0$ belongs to $\Spec \QH(X)$ and the
  Zariski tangent space to $\Spec \QH(X)$ at this point is of dimension $2$.
  Hence, since $\Spec \QH(X)$ is of dimension $0$, the point $h = s = t = 0$ is
  non-reduced.

  To determine the algebra structure we set $h = 0$ in the presentaion in
  Proposition \ref{proposition:E6-presentation-small-QH} and see that we have
  unique solution and the algebra structure at that point is given by $K[s,t]/(6t^2, 2s^3)$,
  i.e. this is exactly the Jacobian algebra of the isolated hypersurface singularity of type $\rE_6$.

  To finish the proof it is enough to show that the vector space dimension of the
  locus $h \neq 0$ is equal to $66$ and that this locus is reduced. This can be done
  easily using SageMath~\cite{sagemath}.
\end{proof}

\bigskip

\begin{lemma}
  In $H^*(X, \bQ)$ we have
  \begin{equation}\label{eq:E6-t^2-via-Schubert}
    t^{\cup 2}  = \sigma_{\alpha(010110)}
  \end{equation}
  \begin{equation}\label{eq:E6-s^2-via-Schubert}
    s^{\cup 2}  = \sigma_{\alpha(011210)} + 2 \sigma_{\alpha(011111)} + \sigma_{\alpha(111110)}
  \end{equation}
  \begin{equation}\label{eq:E6-s^3-via-Schubert}
    s^{\cup 3} =  6 \sigma_{\alpha(010100)} + 2 \sigma_{\alpha(000011)} + 9 \sigma_{\alpha(000110)}
    + 6 \sigma_{\alpha(001100)} + \sigma_{\alpha(101000)}
  \end{equation}
\end{lemma}
\begin{proof}
  This is a routine calculation using the Littlewood-Richardson rule for $\rE_6/\rP_2$. We used \cite{LRCalc} for this.
\end{proof}

\begin{proposition}
  \label{proposition:E6-4-point-GW-invariants}
  We have
  \begin{equation*}
    \scal{\ptclass,t,t,t}_1 = 1
  \end{equation*}
  \begin{equation*}
    \scal{\ptclass, h, t, \gamma}_1 = \coeff_{\sigma_{\alpha(010111)}}(\gamma)
  \end{equation*}
  \begin{equation*}
    \scal{\ptclass, h, s, \gamma}_1 = \coeff_{\sigma_{\alpha(010110)}}(\gamma)
  \end{equation*}
  \begin{equation*}
    \scal{\ptclass, s, s, \gamma}_1 = \coeff_{\sigma_{\alpha(011111)}}(\gamma) + \coeff_{\sigma_{\alpha(011210)}}(\gamma)
  \end{equation*}

\end{proposition}

\begin{proof}
  The proof is a combination of Lemma \ref{lemma:quantum-to-schubert-calculus-long-root} and some computations in the classical cohomology ring of $\G(3,6) = \rA_5/\rP_3$.

  \bigskip

  \noindent \emph{Invariant $\scal{\ptclass,t,t,t}_1$.} By Lemma \ref{lemma:quantum-to-schubert-calculus-long-root} we have
  \begin{equation*}
    \scal{\ptclass,t,t,t}_1 = \deg_{\G(3,6)}(\bar{t} \cup \bar{t} \cup \bar{t}).
  \end{equation*}
  Using \cite{LRCalc} we get
  \begin{equation*}
    \bar{t} \cup \bar{t} \cup \bar{t} = \ptclass.
  \end{equation*}

  \bigskip

  \noindent \emph{Invariant $\scal{\ptclass, h, t, \gamma}_1$.} By Lemma \ref{lemma:quantum-to-schubert-calculus-long-root} we have
  \begin{equation*}
    \scal{\ptclass, h, t, \gamma}_1 = \deg_{\G(3,6)}(\bar{h} \cup \bar{t} \cup \bar{\gamma})
    = \coeff_{\bar{t}^\vee}(\bar{\gamma}).
  \end{equation*}
  Using \cite{LRCalc} we get
  \begin{equation*}
    \bar{t}^\vee = \sigma_{\G(3,6)}^{s_2s_3s_4s_1s_2s_3}
  \end{equation*}
  Lifting $\bar{t}^\vee$ to $\rE_6/\rP_2$ by appending $s_2$ to the Weyl group element (and keeping in mind the relabelling vertices in Dynkin diagrams) we get the Schubert class
  \begin{equation*}
    \sigma_{\rE_6/\rP_2}^{s_3s_4s_5s_1s_3s_4s_2} = \sigma_{\rE_6/\rP_2}^{\alpha(010111)},
  \end{equation*}
  where we used \cite{LRCalc} on $\rE_6/\rP_2$ to convert into the root labelling (see Section \ref{subsection:quantum-schubert-calculus-for-coadjoint-varieties}).
  Now the claim follows by \eqref{eq:simple-fact-proofs}.

  \bigskip

  \noindent \emph{Invariant $\scal{\ptclass, h, s, \gamma}_1$.} By Lemma \ref{lemma:quantum-to-schubert-calculus-long-root} we have
  \begin{equation*}
    \scal{\ptclass, h, s, \gamma}_1 = \deg_{\G(3,6)}(\bar{h} \cup \bar{s} \cup \bar{\gamma})
    = \coeff_{\bar{s}^\vee}(\bar{\gamma}).
  \end{equation*}
  Using \cite{LRCalc} we get
  \begin{equation*}
    \bar{s}^\vee = \sigma_{\G(3,6)}^{s_5s_2s_3s_4s_1s_2s_3}
  \end{equation*}
  Lifting $\bar{s}^\vee$ to $\rE_6/\rP_2$ by appending $s_2$ to the Weyl group element (and keeping in mind the relabelling vertices in Dynkin diagrams) we get the Schubert class
  \begin{equation*}
    \sigma_{\rE_6/\rP_2}^{s_6s_3s_4s_5s_1s_3s_4s_2} = \sigma_{\rE_6/\rP_2}^{\alpha(010110)},
  \end{equation*}
  where we used \cite{LRCalc} on $\rE_6/\rP_2$ to convert into the root labelling (see Section \ref{subsection:quantum-schubert-calculus-for-coadjoint-varieties}).
  Now the claim follows by \eqref{eq:simple-fact-proofs}.

  \bigskip

  \noindent \emph{Invariant $\scal{\ptclass, s, s, \gamma}_1$.} By Lemma \ref{lemma:quantum-to-schubert-calculus-long-root} we have
  \begin{equation*}
    \scal{\ptclass, s, s, \gamma}_1 = \deg_{\G(3,6)}(\bar{s} \cup \bar{s} \cup \bar{\gamma}).
  \end{equation*}
  Using \cite{LRCalc} we get
  \begin{equation*}
    \bar{s} \cup \bar{s} = \sigma_{\G(3,6)}^{s_3s_4s_2s_3} + \sigma_{\G(3,6)}^{s_4s_1s_2s_3}
  \end{equation*}
  and
  \begin{equation*}
    (\bar{s} \cup \bar{s})^\vee = \sigma_{\G(3,6)}^{s_5s_4s_1s_2s_3} + \sigma_{\G(3,6)}^{s_3s_4s_1s_2s_3}.
  \end{equation*}
  Lifting $(\bar{s} \cup \bar{s})^\vee$ to $\rE_6/\rP_2$ by appending $s_2$ to the Weyl group element (and keeping in mind the relabelling vertices in Dynkin diagrams) we get the Schubert class
  \begin{equation*}
    \sigma_{\rE_6/\rP_2}^{s_6s_5s_1s_3s_4s_2} + \sigma_{\rE_6/\rP_2}^{s_4s_5s_1s_3s_4s_2} = \sigma_{\rE_6/\rP_2}^{\alpha(011210)} + \sigma_{\rE_6/\rP_2}^{\alpha(011111)},
  \end{equation*}
  where we used \cite{LRCalc} on $\rE_6/\rP_2$ to convert into the root labelling (see Section \ref{subsection:quantum-schubert-calculus-for-coadjoint-varieties}).
  Now the claim follows by \eqref{eq:simple-fact-proofs}.
\end{proof}

\begin{proposition}
\label{proposition:E6-presentation-big-QH}
  In $\BQH_{s,t}(X)$ we have the following equalities modulo $\gt\gm$:
  \begin{equation}\label{eq:presentation-big-QH-E6-relation-1}
    h^8 - 6h^5s + 3h^4t + 9h^2s^2 - 12hst + 6t^2 \equiv 2 qt_{\delta_2} \ ({\rm mod}\  \gt\gm)
  \end{equation}
  \begin{equation}\label{eq:presentation-big-QH-E6-relation-2}
    h^9 - 4h^6s + 3h^5t + 3h^3s^2 - 6h^2st + 2s^3 \equiv  -2 qt_{\delta_1} \ ({\rm mod}\  \gt\gm)
  \end{equation}
  \begin{multline}\label{eq:presentation-big-QH-E6-relation-3}
    -97h^{12} + 442h^9s - 247h^8t - 507h^6s^2 + \\ + 624h^5st - 156h^2s^2t + 48hq \equiv 0 \ ({\rm mod}\  \gt\gm)
  \end{multline}
\end{proposition}

\begin{proof}
  For degree reasons \eqref{eq:presentation-big-QH-E6-relation-3} holds automatically. Thus, we only need to take care of \eqref{eq:presentation-big-QH-E6-relation-1}
  and \eqref{eq:presentation-big-QH-E6-relation-2}.

  \medskip

  \noindent \emph{Relation \eqref{eq:presentation-big-QH-E6-relation-1}.} The left hand side of \eqref{eq:presentation-big-QH-E6-relation-1} is of the form $-h \sigma + 6 t^2$,
  where $\sigma \in H^{14}(X,\QQ)$ is a linear combination of Schubert classes. Thus, in the classical (and also in the small quantum) cohomology of $X$ we have the equality $h \sigma = 6 t^2$. By the hard Lefschetz theorem, multiplication by $h$, considered as a map from $H^{14}(X,\QQ) \to H^{16}(X,\QQ)$, is injective. Hence, there exist a unique class $\sigma \in H^{14}(X,\QQ)$ so that $h \sigma = 6 t^2$.
  Using \eqref{eq:E6-t^2-via-Schubert} and \cite{LRCalc} we compute
  \begin{equation*}
    \sigma = 4 \sigma_{\alpha(010111)} + 2 \sigma_{\alpha(011110)} - 4 \sigma_{\alpha(001111)} - 2 \sigma_{\alpha(111100)}
    + 2 \sigma_{\alpha(101110)}.
  \end{equation*}
  For degree reasons the linear term of $(-h \sigma + 6 t^2)^{\star}$ is given by
  \begin{equation*}
    \left( - \scal{h, \sigma, t, \ptclass}_1 + 6 \scal{t, t, t, \ptclass}_1 \right) q t_{\delta_2}.
  \end{equation*}
  Applying Proposition \ref{proposition:E6-4-point-GW-invariants} we conclude that the linear term is in fact equal to
  $2 q t_{\delta_2}$,
  and the claim is proved.

  \medskip

  \noindent \emph{Relation \eqref{eq:presentation-big-QH-E6-relation-2}.} The left hand
  side of \eqref{eq:presentation-big-QH-E6-relation-2} is of the form $-h\gamma + 2s^3$,
  where $\gamma = \gamma_1 + \gamma_2$ with $\gamma_1 \in H^{16}(X,\QQ)$ and
  $\gamma_2 \in \bQ \, qt_{\delta_2}$. Since we are only interested in the linear terms
  of $(-h\gamma + 2s^3)^\star$, we proceed as if we had $\gamma_2 = 0$. By the hard Lefschetz
  theorem, the classical multiplication by $h$, considered as a map from $H^{16}(X,\QQ) \to H^{18}(X,\QQ)$,
  is injective. Hence, there exist a unique class $\gamma \in H^{16}(X,\QQ)$ so that $h \cup \gamma = 2s^{\cup 3}$.
  Using \eqref{eq:E6-s^3-via-Schubert} and \cite{LRCalc} we compute
  \begin{equation*}
    \gamma = 8 \sigma_{\alpha(010110)} + 4 \sigma_{\alpha(011100)} + 4 \sigma_{\alpha(000111)} +
    6 \sigma_{\alpha(001110)} + 2 \sigma_{\alpha(101100)}.
  \end{equation*}
  Denoting $\rho = s^{\cup 2}$ we can rewrite
  \begin{equation*}
    -h\gamma + 2s^3 = -h\gamma + 2 s\rho.
  \end{equation*}
  For degree reasons the linear term of $(-h\gamma + 2 s\rho)^{\star}$ is given by
  \begin{equation*}
    \left(- \scal{h, \gamma, s, \ptclass}_1 + 2 \scal{s, \rho, s, \ptclass}_1 \right) q t_{\delta_1}.
  \end{equation*}
  Applying Proposition \ref{proposition:E6-4-point-GW-invariants} and \eqref{eq:E6-s^2-via-Schubert} we conclude that the linear term is in fact equal to
  $-2q t_{\delta_1}$,
  and the claim is proved.
\end{proof}

As in the case of $\OG(2,2n)$, we obtain the following immediate corollary.
\begin{corollary}\label{corollary:E6-regularity-and-semisimplicity-of-BQH}
  The following statements hold:
  \begin{enumerate}
    \item $\BQH(X)$ is a regular ring.

    \item $\BQH(X)$ is generically semisimple.
  \end{enumerate}
\end{corollary}

\bigskip
\subsection{Type $\rE_7$}

Let $X = \rE_7/\rP_1$ be the coadjoint variety of type $\rE_7$ and recall that in this case for the quantum parameter we have
\begin{equation*}
  \deg(q) = 17.
\end{equation*}

It is shown in \cite{ChPe} that the cohomology classes
\begin{equation*}
  \begin{aligned}
    & h = \sigma^{s_1}, \\
    & s = \sigma^{s_2s_4s_3s_1}, \\
    & t = \sigma^{s_7s_6s_5s_4s_3s_1}.
  \end{aligned}
\end{equation*}
generate the cohomology ring $H^*(X, \bQ)$ and the following presentation of the small quantum cohomology is given.
\begin{proposition}[{\cite[Proposition 5.6]{ChPe}}]
\label{proposition:E7-presentation-small-QH}
  The small quantum cohomology $\QH(X)$ is the quotient of the polynomial ring $K[h,s,t]$ modulo the ideal generated by
  \begin{equation*}
    h^{12} - 6h^8s - 4h^6t + 9h^4s^2 + 12h^2st - s^3 + 3t^2,
  \end{equation*}
  \begin{equation*}
    h^{14} - 6h^{10}s - 2h^8t + 9h^6s^2 + 6h^4st - h^2s^3 + 3s^2t,
  \end{equation*}
  \begin{equation*}
    232h^{18} - 1444h^{14}s - 456h^{12}t + 2508h^{10}s^2 + 1520h^8st - 988h^6s^3 + 133h^2s^4 + 36hq.
  \end{equation*}
\end{proposition}

\vspace{10pt}

As already mentioned in the introduction, the above proposition implies Theorem \ref{theorem:introduction-uniform-presentation-for-QH} in type $\rE_7$.

\begin{lemma}
  \label{lemma:description-fat-point-E7}
  The small quantum cohomology $\QH(X)$ is not semisimple. Its unique non-semisimple factor is supported at the point $h = s = t = 0$ and is isomorphic to the Jacobian algebra of the isolated hypersurface singularity of type $\rE_7$.
\end{lemma}

\begin{proof}
  The proof is identical to the proof of Lemma \ref{lemma:description-fat-point-E6}.
\end{proof}

\begin{lemma}
  In $H^*(X, \bQ)$ we have
  \begin{equation}\label{eq:E7-t^2-via-Schubert}
    t^{\cup 2}  = \sigma_{\alpha(0112100)}
  \end{equation}
  \begin{equation}\label{eq:E7-st-via-Schubert}
    s \cup t  = \sigma_{(1112110)}
  \end{equation}
  \begin{equation}\label{eq:E7-s^2-via-Schubert}
    s^{\cup 2}  = \sigma_{(0112221)} + \sigma_{(1112211)} + \sigma_{(1122111)} + \sigma_{(1122210)}
  \end{equation}
  \begin{multline}\label{eq:E7-s^3-via-Schubert}
    s^{\cup 3} = 3 \sigma_{\alpha(0011111)} + 3 \sigma_{\alpha(0101111)} + 11 \sigma_{\alpha(0111110)} + \\
    + 9 \sigma_{\alpha(0112100)} + 4 \sigma_{\alpha(1011110)} + 6 \sigma_{\alpha(1111100)}
  \end{multline}
  \begin{multline}\label{eq:E7-ts^2-via-Schubert}
     t \cup s^{\cup 2}  = 2 \sigma_{\alpha(0001110)} + 5 \sigma_{\alpha(0011100)} + 3 \sigma_{\alpha(0101100)} + \\
     + 3 \sigma_{\alpha(0111000)} + 2 \sigma_{\alpha(1011000)}
  \end{multline}
\end{lemma}

\begin{proof}
This is a routine calculation using the Littlewood-Richardson rule for $\rE_7/\rP_1$. We used \cite{LRCalc} for this.
\end{proof}

\begin{proposition}
  \label{proposition:E7-4-point-GW-invariants}
  We have
  \begin{equation*}
    \scal{\ptclass,t,t,t}_1 = 0
  \end{equation*}
  \begin{equation*}
    \scal{\ptclass, h, t, \gamma}_1 = \coeff_{\sigma_{\alpha(1011111)}}(\gamma)
  \end{equation*}
  \begin{equation*}
    \scal{\ptclass, h, s, \gamma}_1 = \coeff_{\sigma_{\alpha(1111000)}}(\gamma)
  \end{equation*}
  \begin{equation*}
    \scal{\ptclass, s, t, \gamma}_1 = \coeff_{\sigma_{\alpha(1122111)}}(\gamma)
  \end{equation*}
  \begin{equation*}
    \scal{\ptclass, s, s, \gamma}_1 = \coeff_{\sigma_{\alpha(1112110)}}(\gamma)
  \end{equation*}
\end{proposition}

\begin{proof}
  The proof is a combination of Lemma \ref{lemma:quantum-to-schubert-calculus-long-root} and some computations in the classical cohomology ring of $\OG(6,12)$. As such it is very similar to the proof of Proposition \ref{proposition:E6-4-point-GW-invariants} and so we try to be very concise here.

  \bigskip

  \noindent \emph{Invariant $\scal{\ptclass,t,t,t}_1$.} By Lemma \ref{lemma:quantum-to-schubert-calculus-long-root} we have
  \begin{equation*}
    \scal{\ptclass,t,t,t}_1 = \deg_{\OG(6,12)}(\bar{t} \cup \bar{t} \cup \bar{t}).
  \end{equation*}
  Using \cite{LRCalc} we get
  \begin{equation*}
    \bar{t} \cup \bar{t} \cup \bar{t} = 0,
  \end{equation*}
  and we get the desired vanishing.

  \bigskip

  \noindent \emph{Invariant $\scal{\ptclass, h, t, \gamma}_1$.} By Lemma \ref{lemma:quantum-to-schubert-calculus-long-root} we have
  \begin{equation*}
    \scal{\ptclass, h, t, \gamma}_1 = \deg_{\OG(6,12)}(\bar{h} \cup \bar{t} \cup \bar{\gamma})
    = \coeff_{\bar{t}^\vee}(\bar{\gamma}).
  \end{equation*}
  Using \cite{LRCalc} we get
  \begin{equation*}
    \bar{t}^\vee = \sigma_{\OG(6,12)}^{s_5s_4s_6s_3s_4s_5s_2s_3s_4s_6}
  \end{equation*}
  Lifting $\bar{t}^\vee$ to $\rE_7/\rP_1$ by appending $s_1$ to the Weyl group element (and keeping in mind the relabelling vertices in Dynkin diagrams) we get the Schubert class
  \begin{equation*}
    \sigma_{\rE_7/\rP_1}^{s_2s_4s_5s_6s_3s_4s_5s_2s_4s_3s_1} = \sigma_{\rE_7/\rP_1}^{\alpha(1011111)},
  \end{equation*}
  where we used \cite{LRCalc} on $\rE_7/\rP_1$ to convert into the root labelling (see Section \ref{subsection:quantum-schubert-calculus-for-coadjoint-varieties}).
  Now the claim follows by \eqref{eq:simple-fact-proofs}.

  \bigskip

  \noindent \emph{Invariant $\scal{\ptclass, h, s, \gamma}_1$.} By Lemma \ref{lemma:quantum-to-schubert-calculus-long-root} we have
  \begin{equation*}
    \scal{\ptclass, h, s, \gamma}_1 = \deg_{\OG(6,12)}(\bar{h} \cup \bar{s} \cup \bar{\gamma})
    = \coeff_{\bar{s}^\vee}(\bar{\gamma}).
  \end{equation*}
  Using \cite{LRCalc} we get
  \begin{equation*}
    \bar{s}^\vee = \sigma_{\OG(6,12)}^{s_3s_4s_6s_2s_3s_4s_5s_1s_2s_3s_4s_6}
  \end{equation*}
  Lifting $\bar{s}^\vee$ to $\rE_7/\rP_1$ by appending $s_1$ to the Weyl group element (and keeping in mind the relabelling vertices in Dynkin diagrams) we get the Schubert class
  \begin{equation*}
    \sigma_{\rE_7/\rP_1}^{s_5s_6s_7s_4s_5s_6s_3s_4s_5s_2s_4s_3s_1} = \sigma_{\rE_7/\rP_1}^{\alpha(1111000)},
  \end{equation*}
  where we used \cite{LRCalc} on $\rE_7/\rP_1$ to convert into the root labelling (see Section \ref{subsection:quantum-schubert-calculus-for-coadjoint-varieties}).
  Now the claim follows by \eqref{eq:simple-fact-proofs}.

  \bigskip

  \noindent \emph{Invariant $\scal{\ptclass, s, t, \gamma}_1$.} By Lemma \ref{lemma:quantum-to-schubert-calculus-long-root} we have
  \begin{equation*}
    \scal{\ptclass, s, t, \gamma}_1 = \deg_{\OG(6,12)}(\bar{s} \cup \bar{t} \cup \bar{\gamma}).
  \end{equation*}
  Using \cite{LRCalc} we get
  \begin{equation*}
    \bar{s} \cup \bar{t} = \sigma_{\OG(6,12)}^{s_6s_4s_5s_1s_2s_3s_4s_6}
  \end{equation*}
  and
  \begin{equation*}
    (\bar{s} \cup \bar{t})^\vee = \sigma_{\OG(6,12)}^{s_3s_4s_5s_2s_3s_4s_6}.
  \end{equation*}
  Lifting $(\bar{s} \cup \bar{t})^\vee$ to $\rE_7/\rP_1$ by appending $s_1$ to the Weyl group element (and keeping in mind the relabelling vertices in Dynkin diagrams) we get the Schubert class
  \begin{equation*}
    \sigma_{\rE_7/\rP_1}^{s_5s_6s_4s_5s_2s_4s_3s_1} = \sigma_{\rE_7/\rP_1}^{\alpha(1122111)},
  \end{equation*}
  where we used \cite{LRCalc} on $\rE_7/\rP_1$ to convert into the root labelling (see Section \ref{subsection:quantum-schubert-calculus-for-coadjoint-varieties}).
  Now the claim follows by \eqref{eq:simple-fact-proofs}.

  \bigskip

  \noindent \emph{Invariant $\scal{\ptclass, s, s, \gamma}_1$.} By Lemma \ref{lemma:quantum-to-schubert-calculus-long-root} we have
  \begin{equation*}
    \scal{\ptclass, s, s, \gamma}_1 = \deg_{\OG(6,12)}(\bar{s} \cup \bar{s} \cup \bar{\gamma}).
  \end{equation*}
  Using \cite{LRCalc} we get
  \begin{equation*}
    \bar{s} \cup \bar{s} = \sigma_{\OG(6,12)}^{s_4s_5s_2s_3s_4s_6}
  \end{equation*}
  and
  \begin{equation*}
    (\bar{s} \cup \bar{s})^\vee = \sigma_{\OG(6,12)}^{s_6s_3s_4s_5s_1s_2s_3s_4s_6}.
  \end{equation*}
  Lifting $(\bar{s} \cup \bar{s})^\vee$ to $\rE_7/\rP_1$ by appending $s_1$ to the Weyl group element (and keeping in mind the relabelling vertices in Dynkin diagrams) we get the Schubert class
  \begin{equation*}
    \sigma_{\rE_7/\rP_1}^{s_7s_5s_6s_3s_4s_5s_2s_4s_3s_1} = \sigma_{\rE_7/\rP_1}^{\alpha(1112110)},
  \end{equation*}
  where we used \cite{LRCalc} on $\rE_7/\rP_1$ to convert into the root labelling (see Section \ref{subsection:quantum-schubert-calculus-for-coadjoint-varieties}).
  Now the claim follows by \eqref{eq:simple-fact-proofs}.
\end{proof}

\begin{proposition}
\label{proposition:E7-presentation-big-QH}
  In $\BQH_{s,t}(X)$, we have the following equalities modulo $\gt\gm$:
  \begin{equation}\label{eq:presentation-big-QH-E7-relation-1}
    h^{12} - 6h^8s - 4h^6t + 9h^4s^2 + 12h^2st - s^3 + 3t^2 \equiv - qt_{\delta_2} \ ({\rm mod}\  \gt\gm)
  \end{equation}
  \begin{equation}\label{eq:presentation-big-QH-E7-relation-2}
    h^{14} - 6h^{10}s - 2h^8t + 9h^6s^2 + 6h^4st - h^2s^3 + 3s^2t \equiv  qt_{\delta_1} \ ({\rm mod}\  \gt\gm)
  \end{equation}
  \begin{multline}\label{eq:presentation-big-QH-E7-relation-3}
    232h^{18} - 1444h^{14}s - 456h^{12}t + 2508h^{10}s^2 + \\
    + 1520h^8st - 988h^6s^3 + 133h^2s^4 + 36hq \equiv 0 \ ({\rm mod}\  \gt\gm)
  \end{multline}
\end{proposition}

\begin{proof}
For degree reasons \eqref{eq:presentation-big-QH-E7-relation-3} holds automatically. Thus, we only need to take care of \eqref{eq:presentation-big-QH-E7-relation-1} and \eqref{eq:presentation-big-QH-E7-relation-2}.

\medskip

\noindent \emph{Relation \eqref{eq:presentation-big-QH-E7-relation-1}.}
The lefthand side of \eqref{eq:presentation-big-QH-E7-relation-1} is of the form $h \sigma - s^3 + 3 t^2$,
where $\sigma \in H^{22}(X, \bQ)$ is a linear combination of Schubert classes.
Thus, in the classical (and in the small quantum) cohomology of $X$ we have the equality $h \sigma = s^3 - 3 t^2$. By the hard Lefschetz theorem, multiplication by $h$, considered as a map from $H^{22}(X,\QQ) \to H^{24}(X,\QQ)$, is injective. Hence, there exist a unique class $\sigma \in H^{22}(X,\QQ)$ so that $h \sigma = s^3 - 3 t^2$.
Using \eqref{eq:E7-t^2-via-Schubert}, \eqref{eq:E7-s^3-via-Schubert}, and \cite{LRCalc} we compute
\begin{equation*}
\sigma = 3 \sigma_{\alpha(0111111)} + 4 \sigma_{\alpha(0112110)} + 4 \sigma_{\alpha(1111110)} + 2 \sigma_{\alpha(1112100)}.
\end{equation*}
Setting $\tau = s^2$, we can rewrite the left hand side of \eqref{eq:presentation-big-QH-E7-relation-1} as $h\sigma - s\tau + 3t^2$. Thus, now we need to compute $(h\sigma - s\tau + 3t^2)^\star$. For degree reasons the linear term of $(h\sigma - s\tau + 3t^2)^\star$ is given by
\begin{equation*}
  \left( \scal{h,\sigma, \ptclass ,t}_1 - \scal{s,\tau, \ptclass ,t}_1 + 3\scal{t,t, \ptclass ,t}_1 \right) q t_{\delta_2}.
\end{equation*}
Applying \eqref{eq:E7-s^2-via-Schubert} and Proposition \ref{proposition:E7-4-point-GW-invariants} we conclude that the linear term is indeed equal to $-qt_{\delta_2}$, and the claim is proved.

\bigskip
\noindent \emph{Relation \eqref{eq:presentation-big-QH-E7-relation-2}.}
The lefthand side of \eqref{eq:presentation-big-QH-E7-relation-2} is of the form $-h\gamma + 3s^2t$,
where $\gamma \in H^{26}(X, \bQ)$ is a linear combination of Schubert classes.
Thus, in the classical (and in the small quantum) cohomology of $X$ we have the equality $h \gamma = 3 s^2t$. By the hard Lefschetz theorem, multiplication by $h$, considered as a map from $H^{26}(X,\QQ) \to H^{28}(X,\QQ)$, is injective. Hence, there exist a unique class $\gamma \in H^{26}(X,\QQ)$ so that $h \gamma = 3 s^2t$.
Using \eqref{eq:E7-ts^2-via-Schubert} and \cite{LRCalc} we compute
\begin{equation*}
  \gamma = 4 \sigma_{\alpha(0011110)} + 2 \sigma_{\alpha(0101110)} + 7 \sigma_{\alpha(0111100)} + 4 \sigma_{\alpha(1011100)}
  + 2\sigma_{\alpha(1111000)}.
\end{equation*}

Setting $\eta = st = \sigma_{\alpha(1112110)}$, we can rewrite the left hand side of \eqref{eq:presentation-big-QH-E7-relation-1} as $- h\gamma + 3s\eta$. Thus, now we need to compute $(- h\gamma + 3s\eta)^\star$. For degree reasons the linear term of $(- h\gamma + 3s\eta)^\star$ is given by
\begin{equation*}
  \left( -\scal{h,\gamma, \ptclass ,s}_1 + 3 \scal{s,\eta, \ptclass ,s}_1 \right) q t_{\delta_1}.
\end{equation*}
Applying \eqref{eq:E7-st-via-Schubert} and Proposition \ref{proposition:E7-4-point-GW-invariants} we conclude that the linear term is indeed equal to $qt_{\delta_1}$, and the claim is proved.
\end{proof}

As in the case of $\OG(2,2n)$, we obtain the following immediate corollary.
\begin{corollary}\label{corollary:E7-regularity-and-semisimplicity-of-BQH}
  The following statements hold:
  \begin{enumerate}
    \item $\BQH(X)$ is a regular ring.

    \item $\BQH(X)$ is generically semisimple.
  \end{enumerate}
\end{corollary}

\bigskip
\subsection{Type $\rE_8$}

Let $X = \rE_8/\rP_8$ be the coadjoint variety of type $\rE_8$. It is shown in \cite{ChPe} that the cohomology classes
\begin{equation*}
  \begin{aligned}
    & h = \sigma^{s_8}, \\
    & s = \sigma^{s_2s_4s_5s_6s_7s_8}, \\
    & t = \sigma^{s_6s_5s_4s_3s_2s_4s_5s_6s_7s_8}.
  \end{aligned}
\end{equation*}
generate the classical and the small quantum cohomology rings, and their presentations are given. Note that here we have
\begin{equation*}
 \deg(q) = 29.
\end{equation*}
\begin{proposition}[{\cite[Proposition 5.7]{ChPe}}]
\label{proposition:E8-presentation-small-QH}
  The small quantum cohomology $\QH(X)$ is the quotient of the polynomial ring $K[h,s,t]$ modulo the ideal generated by
  \begin{equation*}
    h^{14}s + 6h^{10}t - 3h^8s^2 - 12h^4st - 10h^2s^3 + 3t^2,
  \end{equation*}
  \begin{equation*}
    29h^{24} - 120h^{18}s + 15h^{14}t + 45h^{12}s^2 - 30h^8st + 180h^6s^3 - 30h^2s^2t + 5s^4,
  \end{equation*}
  \begin{multline*}
    - 86357 h^{30} + 368652 h^{24}s - 44640 h^{20}t - 189720 h^{18}s^2 + \\
    + 94860 h^{14}st - 473680 h^{12}s^3 + 74400h^8s^2t - 1240 h^2s^3t + 60hq.
  \end{multline*}
\end{proposition}

\vspace{10pt}

As already mentioned in the introduction, the above proposition implies Theorem \ref{theorem:introduction-uniform-presentation-for-QH} in type $\rE_8$.

\begin{lemma}
  \label{lemma:description-fat-point-E8}
  The small quantum cohomology $\QH(X)$ is not semisimple. Its unique non-semisimple factor is supported at the point $h = s = t = 0$ and is isomorphic to the Jacobian algebra of the isolated hypersurface singularity of type $\rE_8$.
\end{lemma}

\begin{proof}
  The proof is identical to the proof of Lemma \ref{lemma:description-fat-point-E6}.
\end{proof}

\begin{lemma}
  In $H^*(X, \bQ)$ we have the equalities
  \begin{multline}\label{eq:E8-t^2-via-Schubert}
    t^{\cup 2}  = 4 \sigma_{\alpha(01122111)} + 7 \sigma_{\alpha(01122210)} + 8 \sigma_{\alpha(11221110)} + \\
    + 16 \sigma_{\alpha(11222100)} + 2 \sigma_{\alpha(11121111)} + 14 \sigma_{\alpha(11122110)}
  \end{multline}
  \begin{multline}\label{eq:E8-s^3-via-Schubert}
    s^{\cup 3} = 6 \sigma_{\alpha(01122221)} + 58 \sigma_{\alpha(11222111)} + 85 \sigma_{\alpha(11222210)} + \\
    + 111 \sigma_{\alpha(11232110)} + 34 \sigma_{\alpha(11122211)} + 25 \sigma_{\alpha(12232100)}
  \end{multline}
  \begin{multline}\label{eq:E8-s^4-via-Schubert}
    s^{\cup 4} = 1668 \sigma_{\alpha(01011110)} + 3957 \sigma_{\alpha(01121000)} + 5600 \sigma_{\alpha(01111100)} + \\
    + 432 \sigma_{\alpha(00011111)} + 2256 \sigma_{\alpha(00111110)} + \\
    + 2888 \sigma_{\alpha(11111000)} + 2048 \sigma_{\alpha(10111100)}
  \end{multline}
\end{lemma}

\begin{proof}
This is a routine calculation using the Littlewood-Richardson rule for $\rE_8/\rP_8$. We used \cite{LRCalc} for this.
\end{proof}

\begin{proposition}
  \label{proposition:E8-4-point-GW-invariants}
  We have
  \begin{equation*}
    \scal{\ptclass,t,t,t}_1 = 2
  \end{equation*}
  \begin{equation*}
    \scal{\ptclass, h, t, \gamma}_1 = \coeff_{\sigma_{\alpha(01122211)}}(\gamma)
  \end{equation*}
  \begin{equation*}
    \scal{\ptclass, h, s, \gamma}_1 = \coeff_{\sigma_{\alpha(01011111)}}(\gamma)
  \end{equation*}
  \begin{equation*}
    \scal{\ptclass, s, s, \gamma}_1 = 2\coeff_{\sigma_{\alpha(11122211)}}(\gamma) + 2\coeff_{\sigma_{\alpha(11222111)}}(\gamma)
  \end{equation*}
\end{proposition}

\begin{proof}
  The proof is a combination of Lemma \ref{lemma:quantum-to-schubert-calculus-long-root} and some computations in the classical cohomology ring of the Freudenthal variety $\rE_7/\rP_7$. As such it is very similar to the proofs of Propositions
  \ref{proposition:E6-4-point-GW-invariants} and \ref{proposition:E7-4-point-GW-invariants}
  and so we try to be very concise here.

  \bigskip

  \noindent \emph{Invariant $\scal{\ptclass,t,t,t}_1$.} By Lemma \ref{lemma:quantum-to-schubert-calculus-long-root} we have
  \begin{equation*}
    \scal{\ptclass,t,t,t}_1 = \deg_{\rE_7/\rP_7}(\bar{t} \cup \bar{t} \cup \bar{t}).
  \end{equation*}
  Using \cite{LRCalc} we compute
  \begin{equation*}
    \bar{t} \cup \bar{t} \cup \bar{t} = 2 \pointclass.
  \end{equation*}
  Hence, we get
  \begin{equation*}
    \deg_{\rE_7/\rP_7}(\bar{t} \cup \bar{t} \cup \bar{t}) = 2.
  \end{equation*}

  \bigskip

  \noindent \emph{Invariant $\scal{\ptclass, h, t, \gamma}_1$.} By Lemma \ref{lemma:quantum-to-schubert-calculus-long-root} we have
  \begin{equation*}
    \scal{\ptclass, h, t, \gamma}_1 = \deg_{\rE_7/\rP_7}(\bar{h} \cup \bar{t} \cup \bar{\gamma}) = \coeff_{\bar{t}^\vee}(\bar{\gamma}).
  \end{equation*}
  Using \cite{LRCalc} on $\rE_7/\rP_7$ we get
  \begin{equation*}
    \bar{t}^\vee  = \sigma_{\rE_7/\rP_7}^{s_7s_1s_3s_4s_5s_6s_2s_4s_5s_3s_4s_2s_1s_3s_4s_5s_6s_7}.
  \end{equation*}
  Lifting $\bar{t}^\vee$ to $\rE_8/\rP_8$ by appending $s_8$ to the Weyl group element we get the Schubert class
  \begin{equation*}
    \sigma_{\rE_8/\rP_8}^{s_7s_1s_3s_4s_5s_6s_2s_4s_5s_3s_4s_2s_1s_3s_4s_5s_6s_7s_8} = \sigma_{\rE_8/\rP_8}^{\alpha(01122211)},
  \end{equation*}
  where we used \cite{LRCalc} on $\rE_8/\rP_8$ to convert into the root labelling (see Section \ref{subsection:quantum-schubert-calculus-for-coadjoint-varieties}).
  Now the claim follows by \eqref{eq:simple-fact-proofs}.

  \bigskip

  \noindent \emph{Invariant $\scal{\ptclass, h, s, \gamma}_1$.}
  By Lemma \ref{lemma:quantum-to-schubert-calculus-long-root} we have
  \begin{equation*}
    \scal{\ptclass, h, s, \gamma}_1 = \deg_{\rE_7/\rP_7}(\bar{h} \cup \bar{s} \cup \bar{\gamma}) = \coeff_{\bar{s}^\vee}(\bar{\gamma}).
  \end{equation*}
  Using \cite{LRCalc} on $\rE_7/\rP_7$ we get
  \begin{equation*}
    \bar{s}^\vee  = \sigma_{\rE_7/\rP_7}^{s_3s_4s_5s_6s_7s_1s_3s_4s_5s_6s_2s_4s_5s_3s_4s_1s_3s_2s_4s_5s_6s_7}.
  \end{equation*}
  Lifting $\bar{s}^\vee$ to $\rE_8/\rP_8$ by appending $s_8$ to the Weyl group element we get the Schubert class
  \begin{equation*}
    \sigma_{\rE_8/\rP_8}^{s_3s_4s_5s_6s_7s_1s_3s_4s_5s_6s_2s_4s_5s_3s_4s_1s_3s_2s_4s_5s_6s_7s_8} = \sigma_{\rE_8/\rP_8}^{\alpha(01011111)},
  \end{equation*}
  where we used \cite{LRCalc} on $\rE_8/\rP_8$ to convert into the root labelling (see Section \ref{subsection:quantum-schubert-calculus-for-coadjoint-varieties}).
  Now the claim follows by \eqref{eq:simple-fact-proofs}.

  \bigskip

  \noindent \emph{Invariant $\scal{\ptclass, s, s, \gamma}_1$.} By Lemma \ref{lemma:quantum-to-schubert-calculus-long-root} we have
  \begin{equation*}
    \scal{\ptclass,s,s,\gamma}_1 = \deg_{\rE_7/\rP_7}(\bar{s} \cup \bar{s} \cup \bar{\gamma}).
  \end{equation*}
  Using \cite{LRCalc} on $\rE_7/\rP_7$ we get
  \begin{equation*}
    \left( \bar{s} \cup \bar{s} \right)^\vee =
    2 \sigma_{\rE_7/\rP_7}^{s_6s_7s_4s_5s_6s_2s_4s_5s_3s_4s_1s_3s_2s_4s_5s_6s_7} + 2 \sigma_{\rE_7/\rP_7}^{s_7s_3s_4s_5s_6s_2s_4s_5s_3s_4s_1s_3s_2s_4s_5s_6s_7}.
  \end{equation*}
  Lifting each summand of $\left( \bar{s} \cup \bar{s} \right)^\vee$ to $\rE_8/\rP_8$ by appending $s_8$ to the Weyl group element we get
  \begin{equation*}
    \begin{aligned}
      & \sigma_{\rE_8/\rP_8}^{s_6s_7s_4s_5s_6s_2s_4s_5s_3s_4s_1s_3s_2s_4s_5s_6s_7s_8} = \sigma_{\rE_8/\rP_8}^{\alpha(11222111)} \\
      & \sigma_{\rE_8/\rP_8}^{s_7s_3s_4s_5s_6s_2s_4s_5s_3s_4s_1s_3s_2s_4s_5s_6s_7s_8} = \sigma_{\rE_8/\rP_8}^{\alpha(11122211)}
    \end{aligned}
  \end{equation*}
  where we used \cite{LRCalc} on $\rE_8/\rP_8$ to convert into the root labelling (see Section \ref{subsection:quantum-schubert-calculus-for-coadjoint-varieties}).
  Now the claim follows by \eqref{eq:simple-fact-proofs}.
\end{proof}

\begin{proposition}
\label{proposition:E8-presentation-big-QH}
  In $\BQH_{s,t}(X)$ we have the following equalities modulo $\gt\gm$:
  \begin{multline}\label{eq:presentation-big-QH-E8-relation-1}
    h^{14}s + 6h^{10}t - 3h^8s^2 - 12h^4st
    - 10h^2s^3 + 3t^2 \equiv - qt_{\delta_2} \ ({\rm mod}\  \gt\gm)
  \end{multline}
  \begin{multline}\label{eq:presentation-big-QH-E8-relation-2}
    29h^{24} - 120h^{18}s + 15h^{14}t + 45h^{12}s^2 - 30h^8st + \\
    + 180h^6s^3 - 30h^2s^2t + 5s^4 \equiv  - qt_{\delta_1} \ ({\rm mod}\  \gt\gm)
  \end{multline}
  \begin{multline}\label{eq:presentation-big-QH-E8-relation-3}
  - 86357 h^{30} + 368652 h^{24}s - 44640 h^{20}t - \\
  - 189720 h^{18}s^2 + 94860 h {14}st - 473680 h^{12}s^3 + \\
  + 74400h^8s^2t - 1240 h^2s^3t + 60hq \equiv 0 \ ({\rm mod}\  \gt\gm)
  \end{multline}
\end{proposition}

\begin{proof}
For degree reasons \eqref{eq:presentation-big-QH-E8-relation-3} holds automatically. Thus, we only need to take care of \eqref{eq:presentation-big-QH-E8-relation-1} and
\eqref{eq:presentation-big-QH-E8-relation-2}.

\medskip

\noindent \emph{Relation \eqref{eq:presentation-big-QH-E8-relation-1}.} The lefthand side of \eqref{eq:presentation-big-QH-E8-relation-1} is of the form $- h \sigma + 3 t^2$,
where $\sigma \in H^{38}(X, \bQ)$ is a linear combination of Schubert classes.
Thus, in the classical (and in the small quantum) cohomology of $X$ we have the equality $h \sigma = 3 t^2$. By the hard Lefschetz theorem, multiplication by $h$, considered as a map from $H^{38}(X,\QQ) \to H^{40}(X,\QQ)$, is injective. Hence, there exist a unique class $\sigma \in H^{38}(X,\QQ)$ so that $h \sigma = 3 t^{\cup 2}$. Using \eqref{eq:E8-t^2-via-Schubert} and \cite{LRCalc} we compute
\begin{multline*}
  \sigma = 7 \sigma_{\alpha(01122211)} + \sigma_{\alpha(11221111)} + 23 \sigma_{\alpha(11222110)} + \\
  + 25 \sigma_{\alpha(11232100)} + 5 \sigma_{\alpha(11122111)} + 14 \sigma_{\alpha(11122210)}.
\end{multline*}
For degree reasons the linear term of $(- h \sigma + 3 t^2)^\star$ is given by
\begin{equation*}
  \left( -\scal{h,\sigma, \ptclass ,t}_1 + 3\scal{t,t, \ptclass ,t}_1 \right) q t_{\delta_2}.
\end{equation*}
Applying Proposition \ref{proposition:E8-4-point-GW-invariants} we conclude that the linear term is indeed equal to $-qt_{\delta_2}$, and the claim is proved.

\bigskip
\noindent \emph{Relation \eqref{eq:presentation-big-QH-E8-relation-2}.} The left hand side of
\eqref{eq:presentation-big-QH-E8-relation-2} is of the form $-h\gamma + 5s^4$, where
$\gamma = \gamma_1 + \gamma_2$ with $\gamma_1 \in H^{46}(X,\QQ)$ and $\gamma_2 \in H^6(X,\QQ) \cdot \, qt_{\delta_2}$.
Since we are only interested in the linear terms of $(-h\gamma + 5s^4)^\star$, we proceed as if we had
$\gamma_2 = 0$.
By the hard Lefschetz theorem, the classical multiplication by $h$, considered as a map from $H^{46}(X,\QQ) \to H^{48}(X,\QQ)$, is injective. Hence, there exist a unique class $\gamma \in H^{46}(X,\QQ)$ so that $h \cup \gamma = 5s^{\cup 4}$. Using \eqref{eq:E8-s^4-via-Schubert} and \cite{LRCalc} we compute
\begin{multline*}
  \gamma = 921 \sigma_{\alpha(01011111)} + 12963 \sigma_{\alpha(01121100)} + \\
  + 7419 \sigma_{\alpha(01111110)} + 1239 \sigma_{\alpha(00111111)} + 6822 \sigma_{\alpha(11121000)} + \\
  + 7618 \sigma_{\alpha(11111100)} + 2622 \sigma_{\alpha(10111110)}.
\end{multline*}
Setting $\tau = s^3$ we can rewrite the left hand side of \eqref{eq:presentation-big-QH-E8-relation-2} as $- h\gamma + 5s\tau$. For degree reasons the linear term of $(- h\gamma + 5s\tau)^\star$ is given by
\begin{equation*}
  (-\scal{h,\gamma, \ptclass ,s}_1 + 5 \scal{s, \tau , \ptclass ,s}_1) q t_{\delta_1}
\end{equation*}
Applying Proposition \ref{proposition:E8-4-point-GW-invariants} and \eqref{eq:E8-s^3-via-Schubert} we conclude that the linear term is indeed equal to $-qt_{\delta_1}$, and the claim is proved.
\end{proof}

As in the case of $\OG(2,2n)$, we obtain the following immediate corollary.
\begin{corollary}\label{corollary:E8-regularity-and-semisimplicity-of-BQH}
  The following statements hold:
  \begin{enumerate}
    \item $\BQH(X)$ is a regular ring.

    \item $\BQH(X)$ is generically semisimple.
  \end{enumerate}
\end{corollary}

\section{Type $\rF_4$}
\label{section:type-F}

Here we consider the coadjoint variety in type $\rF_4$. Let us recall the corresponding
Dynkin diagram
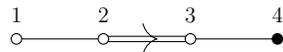
\begin{equation*}\label{diagram:F4}
  \dynkin[edge length = 2em,labels*={1,...,4}, arrow width=2mm, scale=1.4]F{ooo*}
\end{equation*}
where the labelling of vertices follows \cite{Bo} and the black colored vertex corresponds
to the maximal parabolic subgroup defining the respective coadjoint variety
\begin{equation*}
  X = \rF_4/\rP_4.
\end{equation*}

It is shown in \cite{ChPe} that the cohomology classes
\begin{equation*}
  \begin{aligned}
    & h = \sigma^{s_4}, \\
    & s = \sigma^{s_1s_2s_3s_4}.
  \end{aligned}
\end{equation*}
generate the classical and the small quantum cohomology rings, and their presentations are given. Note that here we have
\begin{equation*}
  \deg(q) = 11.
\end{equation*}

\begin{proposition}[{\cite[Proposition 5.3]{ChPe}}]
\label{proposition:F4-presentation-small-QH}
  The small quantum cohomology $\QH(X)$ is the quotient of the polynomial ring $K[h,s]$ modulo the ideal generated by
  \begin{equation*}
    2h^8 - 6h^4s + 3 s^2,
  \end{equation*}
  \begin{equation*}
    - 11h^{12} + 26h^8 s + 3hq.
  \end{equation*}
\end{proposition}

As already mentioned in the introduction, the above proposition implies Theorem \ref{theorem:introduction-uniform-presentation-for-QH} in type $\rF_4$.

\begin{lemma}
  \label{lemma:description-fat-point-F4}
  The small quantum cohomology $\QH(X)$ is not semisimple. Its unique non-semisimple factor is supported at the point $h = s = 0$ and is isomorphic to the Jacobian algebra of the isolated hypersurface singularity of type $\rA_2$.
\end{lemma}

\begin{proof}
  The proof is identical to the proof of Lemma \ref{lemma:description-fat-point-E6}.
\end{proof}

To compute the required $4$-point GW invariants we apply Corollary \ref{corollary:quantum-to-schubert-calculus-non-simply-laced}.

\begin{proposition}
\label{proposition:F4-4-point-GW-invariants}
  \begin{equation*}
    \scal{\ptclass, s, s, s}_1 = 1
  \end{equation*}

  \begin{equation*}
    \scal{\ptclass, h, s, \gamma}_1 = \coeff_{\sigma_{\alpha(0010)}}(\gamma)
  \end{equation*}
\end{proposition}

\begin{proof}
As in Section \ref{subsection:gw-invariants-for-coadjoint-varieties-of-non-simply-laced-groups}
we consider the embedding \eqref{eq:hyperplane-section-embedding} as a hyperplane section
\begin{equation*}
  j \colon \rF_4/\rP_4 \to \rE_6/\rP_1.
\end{equation*}
Since $\dim \rF_4/\rP_4 = 15$, the induced map $j^* \colon H^*(\rE_6/\rP_1, \bQ) \to H^*(\rF_4/\rP_4, \bQ)$
is an isomorphism up to degree~$7$. Let $\hat{s} \in H^8(\rE_6/\rP_1,\bQ)$ be
the unique Schubert class such that $j^*\hat{s} = s$. By Lemma~\ref{lemma:weyl-group-push-pull}
we have
\begin{equation}\label{eq:F4-lift-of-s-to-E6/P1}
  \hat{s} = \sigma_{\rE_6/\rP_1}^{s_2s_4s_3s_1}.
\end{equation}
Moreover, we have the identification $\Fpt(\rE_6/\rP_1) = \rD_5/\rP_4$ and the cohomology class
appearing in Corollary~\ref{corollary:quantum-to-schubert-calculus-non-simply-laced}
and corresponding to $\hat{s}$ is given by
\begin{equation*}
  \bar{\hat{s}} = \sigma_{\rD_5/\rP_4}^{s_5s_3s_4} \in H^6(\rD_5/\rP_4,\QQ),
\end{equation*}
where we have kept track of the relabelling from $\rE_6$ to $\rD_5$.

\bigskip

\noindent \emph{Invariant $\scal{\ptclass,s,s,s}_1$.} By Corollary
\ref{corollary:quantum-to-schubert-calculus-non-simply-laced} we have
\begin{equation*}
  \scal{\ptclass,s,s,s}_1 = \deg_{\rD_5/\rP_4} \left(\bar{\hat{s}} \cup \bar{\hat{s}} \cup \bar{\hat{s}} \cup h_{\rD_5/\rP_4} \right),
\end{equation*}
where
\begin{equation*}
  h_{\rD_5/\rP_4} = \sigma_{\rD_5/\rP_4}^{s_4}.
\end{equation*}
Using \cite{LRCalc} we compute
\begin{equation*}
  \bar{\hat{s}} \cup \bar{\hat{s}} \cup \bar{\hat{s}} \cup h_{\rD_5/\rP_4} = \pointclass.
\end{equation*}
Hence, we get
\begin{equation*}
  \deg_{\rD_5/\rP_4} \left(\bar{\hat{s}} \cup \bar{\hat{s}} \cup \bar{\hat{s}} \cup h_{\rD_5/\rP_4} \right) = 1.
\end{equation*}

\bigskip

\noindent \emph{Invariant $\scal{\ptclass, h, s, \gamma}_1$.} By the dimension axiom for
Gromov--Witten invariants, the invariant $\scal{\ptclass, h, s, \gamma}_1$ vanishes unless
$\gamma \in H^{14}(\rF_4/\rP_4, \bQ)$. Applying Lemma \ref{lemma:weyl-group-push-pull} we see that
such a class $\gamma$ has a unique lift $\hat{\gamma} \in H^{14}(\rE_6/\rP_1, \bQ)$.

By Corollary \ref{corollary:quantum-to-schubert-calculus-non-simply-laced} we have
\begin{equation*}
  \scal{\ptclass, h, s, \gamma}_1 = \deg_{\rD_5/\rP_4} \left(\bar{\hat{s}} \cup \bar{\hat{\gamma}} \cup h_{\rD_5/\rP_4} \right)
\end{equation*}
Using \cite{LRCalc} we get
\begin{equation*}
  \left( \bar{\hat{s}} \cup h_{\rD_5/\rP_4} \right)^\vee = \sigma_{\rD_5/\rP_4}^{s_3s_5s_1s_2s_3s_4}.
\end{equation*}
Lifting $\left( \bar{\hat{s}} \cup h_{\rD_5/\rP_4} \right)^\vee$ to $\rE_6/\rP_1$ by appending $s_1$ to the Weyl group element gives $\sigma_{\rE_6/\rP_1}^{s_4s_2s_6s_5s_4s_3s_1}$. Restricting it to $\rF_4/\rP_4$ using Lemma \ref{lemma:weyl-group-push-pull}(4) (and at each step keeping in mind the relabelling vertices in Dynkin diagrams) we get the Schubert class
\begin{equation*}
  \sigma_{\rF_4/\rP_4}^{s_4s_2s_3s_1s_2s_3s_4} = \sigma_{\rF_4/\rP_4}^{\alpha(0010)},
\end{equation*}
where we used \cite{LRCalc} on $\rF_4/\rP_4$ to convert into the root labelling (see Section \ref{subsection:quantum-schubert-calculus-for-coadjoint-varieties}).
Now the claim follows by \eqref{eq:simple-fact-proofs}.
\end{proof}

\begin{proposition}\label{proposition:F4-presentation-big-QH}
  In $\BQH_{s}(X)$, we have the following equalities modulo $\gt\gm$
  \begin{equation}\label{eq:presentation-big-QH-F4-relation-1}
    2h^8 - 6h^4s + 3 s^2 \equiv qt_{\delta_1} \ ({\rm mod}\  \gt\gm)
  \end{equation}
  \begin{equation}\label{eq:presentation-big-QH-F4-relation-2}
    - 11h^{12} + 26h^8 s + 3hq \equiv 0 \ ({\rm mod}\  \gt\gm)
  \end{equation}
\end{proposition}

\begin{proof}
  For degree reasons \eqref{eq:presentation-big-QH-F4-relation-2} holds automatically
  and, hence, we only need to consider \eqref{eq:presentation-big-QH-F4-relation-1}.
  The left hand side of \eqref{eq:presentation-big-QH-F4-relation-1} is of the
  form $-h\sigma + 3s^2$, with $\sigma = - 2h^7 + 6h^3s$. Using \cite{LRCalc} we
  easily compute
  \begin{equation*}
    \sigma = 2 \sigma^{s_4s_2s_3s_1s_2s_3s_4}_{\rF_4/\rP_4} + 2 \sigma^{s_3s_2s_3s_1s_2s_3s_4}_{\rF_4/\rP_4}
    = 2\sigma_{\rF_4/\rP_4}^{\alpha(0010)} + 2 \sigma_{\rF_4/\rP_4}^{\alpha(0001)},
  \end{equation*}
  For degree reasons the linear term of $(-h\sigma + 3s^2)^\star$ is of the form
  \begin{equation*}
    \left(- \scal{h,\sigma, \ptclass ,s}_1 + 3\scal{s,s, \ptclass ,s}_1 \right) q t_{\delta_1}.
  \end{equation*}
  Applying Proposition \ref{proposition:F4-4-point-GW-invariants} we conclude that the linear term is in fact equal to $q t_{\delta_1}$.
\end{proof}

As in the case of $\OG(2,2n)$, we obtain the following immediate corollary.
\begin{corollary}\label{corollary:F4-regularity-and-semisimplicity-of-BQH}
  The following statements hold:
  \begin{enumerate}
    \item $\BQH(X)$ is a regular ring.

    \item $\BQH(X)$ is generically semisimple.
  \end{enumerate}
\end{corollary}

\section{Relation to the unfoldings of ADE-singularities}
\label{section:singularity-theory}

The goal of this section is to strengthen the relation between the quantum cohomology of
coadjoint varieties and unfoldings of ADE-singularities given by Theorem \ref{theorem:introduction-fat-points-of-QH}.
To do that we need to pass to the world of \textsf{$F$-manifolds}. For a thourough treatment
of the background material on $F$-manifolds we refer to \cite{HeMa, Hertling}. To a
reader interested only in a very consice summary of the necessary facts we recommend~\cite[Section~7]{CMMPS}.

Let us briefly explain how we get an $F$-manifold from the big quantum cohomology of a
Fano variety $X$. Recall that the big quantum product is defined using the GW
potential \eqref{eq:GW-potential}. Since $X$ is a Fano variety, the dimension axiom
for GW invariants implies that the coefficients \eqref{eq:GW-potential-coefficients}
of the GW potential are polynomial in $q$. Hence, it makes sense to specialize
the formulas \eqref{eq:GW-potential}--\eqref{eq:small-quantum-product} to $q = 1$.

Viewing \eqref{eq:GW-potential} specialized at $q = 1$ as a formal power series
in $t_0, \dots, t_s$ we can ask ourselves the question, wether this series has a
non-trivial convergence domain in $\bC^{s+1} = H^*(X, \bC)$. In general, the answer to this question is not known.
Thus, we add a convergence assumption to our setup.
\begin{assumption}[\emph{Convergence assumption}]
  The power series $\Phi_{q=1}$ converges in some open neighbourhood $M \subset \bC^{s+1}$ of the origin.
\end{assumption}
Under this assumption \eqref{eq:big-quantum-product} endows $M$ with
the structure of an analytic Frobenius manifold. In particular, forgetting the metric,
we get an $F$-manifold structure on $M$. Below we work with the germ of
this $F$-manifold at the origin $t_0 = t_1 = \dots = t_s$, which corresponds to the
small quantum cohomology at $q = 1$.

As in Section \ref{section:proof-of-KuSm21},
after setting $q = 1$ in $\QH(X^\coadj)$ we can consider
the finite scheme $\QS_{X^\coadj} = \Spec (\QH(X^\coadj))$ endowed with a morphism to
the affine line $\kappa \colon \QS_{X^\coadj} \to \bA^1$ given by the anticanonical
class. We define $\QS_{X^\coadj}^\times$ as the preimage of $\bA^1 \setminus \{0\}$
under $\kappa$ and $\QSo_{X^\coadj}$ as its complement. Since in our case the
anticanonical class is proportinal to the hyperplane class $h$, $\QS_{X^\coadj}^\times$
and $\QSo_{X^\coadj}$ are the vanishing and the non-vanishing loci of $h$ considered
as a function on $\QS_{X^\coadj}$.

\begin{theorem}
  \label{theorem:F-manifolds}
  Let $X^\coadj$ be the coadjoint variety of a simple algebraic group $\rG$ not of type $\rA$ and
  let $M_{\rG}$ be the germ of the $F$-manifold of $\BQH(X^\coadj)$ as above.
  \begin{enumerate}
    \item
    \label{item:theorem-F-manifolds-decomposition}
    The $F$-manifold germ $M_{\rG}$ decomposes into the direct product of
    irreducible germs of $F$-manifolds
    \begin{equation*}
      M_{\rG} = M_{\rG,0} \times \prod_{x \in \QSx_{X^\coadj}} M_{\rG,x}
    \end{equation*}
    and $M_{\rG,0}$ corresponds to the unique fat point of $\QH(X^\coadj)$.

    \medskip

    \item
    \label{item:theorem-F-manifolds-reduced-points}
    The $F$-manifold germs $M_{\rG,x}$ for $x \in \QSx_{X^\coadj}$ are
    one-dimensional and isomorphic to the base space of a semiuniversal
    unfolding of an isolated hypersurface singularity of type $\rA_1$.

    \medskip

    \item
    \label{item:theorem-F-manifolds-spectral-cover}
    The spectral cover of $M_{\rG}$ is smooth.

    \medskip

    \item
    \label{item:theorem-F-manifolds-fat-point}
    The $F$-manifold germ $M_{\rG,0}$ is isomorphic to the base space of
    a semiuniversal unfolding of a simple hypersurface singularity of
    Dynkin type $\rTs(\rG)$.
  \end{enumerate}
\end{theorem}

\begin{proof}
  To prove (\ref{item:theorem-F-manifolds-decomposition}) we apply \cite[Theorem 2.11]{Hertling}
  and Theorem \ref{theorem:introduction-fat-points-of-QH}.

  To prove (\ref{item:theorem-F-manifolds-reduced-points}) we just note that the
  base space of a semiuniversal unfolding of an isolated hypersurface singularity
  of type $\rA_1$ is the unique one-dimensional $F$-manifold germ up to isomorphism.

  To prove (\ref{item:theorem-F-manifolds-spectral-cover}) it is enough to consider
  the component $M_{\rG,0}$ and compute the rank of the Jacobian matrix as in
  the proof of Corollary \ref{corollary:regularity-and-semisimplicity-of-BQH-type-D}(1).

  To prove (\ref{item:theorem-F-manifolds-fat-point}) we proceed as follows.
  By (\ref{item:theorem-F-manifolds-spectral-cover}) and \cite[Theorem 5.6]{Hertling}
  it follows that $M_{\rG,0}$ is isomorphic to the base space of a semiuniversal
  unfolding of some isolated hypersurface singularity $g$. Thus, we just need to prove
  that this singularity is stably right equivalent to a simple singularity $f$
  of type $\rTs(\rG)$. Here we denote by $g, f \in \bC \{x_1, \dots, x_n \}$ the germs
  of holomorphic functions defining these singularities. Since $f$ is
  quasi-homogeneous, a Mather-Yau-type statement holds. Indeed, a theorem by Shoshitaishvili
  \cite{Sho} (see \cite[Theorem 2.29]{GLS} for a more convenient reference) implies
  that if the Jacobian algebras $M_{f}$ and $M_{g}$ are isomorphic as $\bC$-algebras,
  then $f$ and $g$ are right equivalent.
\end{proof}

\section{Proof of \cite[Theorem 1.6]{KuSm21}}
\label{section:proof-of-KuSm21}

We begin by recalling some notation on quantum cohomology from \cite[Section 1]{KuSm21}.
We warn the reader that this notation is slightly different from the conventions in
the present paper outlined in Section \ref{subsection:conventions-and-notation-for-quantum-cohomology}.

\smallskip

Let $X$ be a Fano variety of Picard rank $r$ and vanishing odd cohomology.
Then the small quantum cohomology $\QH(X)$ is a commutative algebra over
the ring $\bQ[q_1, \dots, q_r]$ of functions on the affine
space $\Pic(X) \otimes \bQ$ over $\bQ$. Let
\begin{equation*}
  \QHcan(X) \coloneqq \QH(X) \otimes_{\bQ[q_1, \dots, q_r]} \bC
\end{equation*}
be the base change of $\QH(X)$ to the point of $\Spec(\bQ[q_1, \dots, q_r])$
corresponding to the canonical class of $X$. Thus, $\QHcan(X)$ is a finite
dimensional commutative $\bC$-algebra, whose underlying verctor space is
canonically isomorphic to $H^*(X, \bC)$.

If the Picard rank is one, so that there is only one parameter $q$, then to get
$\QHcan(X)$ one simply needs to put $q = 1$. This also applies to the examples
treated in the present paper, i.e., to obtain a presentation for $\QHcan(X)$ of
coadjoint varieties from the presentation given in Theorem \ref{theorem:introduction-uniform-presentation-for-QH}
one simply needs to replace the ground field $K$ by $\bC$ and set $q = 1$ in the
relations.

Now we can define the \textsf{quantum spectrum} as
\begin{equation*}
  \QS_{X} \coloneqq \Spec(\QHcan(X)),
\end{equation*}
which is a finite scheme endowed with an action of the group $\mu_m$, where
$m$ is the Fano index of $X$. The anticanonical class $-K_X$ defines a morphism
\begin{equation*}
  \kappa \colon \QS_{X} \to \bA^1,
\end{equation*}
which is equivariant with repsect to the standard action of $\mu_m$ on $\bA^1$.
Finally, we define
\begin{equation*}
  \QSx_X \coloneqq \kappa^{-1}(\bA^1 \setminus \{ 0 \}) \quad \text{and} \quad
  \QSo_X \coloneqq \QS_X \setminus \QSx_X
\end{equation*}
Now we are ready to state and prove \cite[Theorem 1.6]{KuSm21}.

\begin{theorem}[{\cite[Theorem 1.6]{KuSm21}}]
  Let $X^\adj$ and $X^\coadj$ be the adjoint and coadjoint varieties of a simple
  algebraic group $\rG$, respectively.
  \begin{enumerate}
  \item If $\rT(\rG) = \rA_{2n}$, then $\QSo_{X^\adj} = \QSo_{X^\coadj} = \emptyset$.
  \item If $\rT(\rG) \neq \rA_{2n}$, then $\QSo_{X^\coadj}$ is a single non-reduced
  point and the localization of $\QH_{\rm can}(X^\coadj)$ at this point is isomorphic
  to the Jacobian ring of a simple hypersurface singularity of type $\rT_{\rm short}(\rG)$.
  \item If $\rT(\rG)$ is simply laced, then we have $X^\adj = X^\coadj$ and $\QSo_{X^\adj} = \QSo_{X^\coadj}$.
  \item  If $\rT(\rG)$ is not simply laced, then $\QSo_{X^\adj} = \emptyset$.
  \end{enumerate}
\end{theorem}

\begin{proof}
  (1) In type $\rA_n$ the algebra $\QHcan(X^\coadj)$ is obtained
  by setting $q_1 = q_2 = 1$ in the presentation of Proposition \ref{proposition:coadjoint-type-A}.
  and $\QSo_{X^\coadj}$ is supported in the locus $h_1 + h_2 = 0$.
  Assuming $n$ is even and using \eqref{eq:coadjoint-type-A-relations} we see that
  $\QSo_{X^\coadj}$ is empty.

  (2) For $\rG$ of type $\rA_n$ we argue as in (1) using Proposition
  \ref{proposition:coadjoint-type-A}. For $\rG$ not of type $\rA$ the statement
  follows from Theorem \ref{theorem:introduction-fat-points-of-QH} and the fact
  that $\QSo_{X^\coadj}$ is supported in the locus~$h = 0$.

  (3) This part is automatic.

  (4) In all cases the claim is obtained by setting $h = 0$ in the
  presentation for the small quantum cohomology and checking easily that
  there are no solutions to the equations. For $\rB_n/\rP_2$ one can either use
  the presentation from \cite[Theorem 2.5]{BKT} or the analysis of the solution set
  done the proof of \cite[Proposition 6.2]{ChPe}. For $\rC_n/ \rP_1 = \bP^{2n-1}$
  this is well-known. For $\rG_2/\rP_2$ and $\rF_4/\rP_1$ one can use presentations
  given in \cite[Proposition 5.1]{ChPe} and \cite[Proposition 5.2]{ChPe}.
\end{proof}

\bibliographystyle{plain}
\bibliography{refs}

\end{document}